\documentclass{amsart}
\usepackage[foot]{amsaddr}
\usepackage{amsfonts}
\usepackage{amsmath}
\usepackage{dsfont}
\usepackage{amssymb}
\usepackage{mathtools}
\usepackage{color}
\usepackage{amsthm}
\usepackage{ifpdf}
\usepackage{graphicx, caption}
\usepackage{bm}
\usepackage{soul}

\usepackage{todonotes}

\usepackage[left=4.5cm, right=4.5cm, top=4cm]{geometry}

\ifpdf
\else
\fi
\theoremstyle{plain}
\newtheorem{theorem}{Theorem}[section]

\newtheorem{lemma}[theorem]{Lemma}
\newtheorem{proposition}[theorem]{Proposition}
\theoremstyle{definition}
\newtheorem{assumption}[theorem]{Assumption}
\newtheorem{definition}[theorem]{Definition}
\newtheorem{example}[theorem]{Example}
\theoremstyle{remark}
\newtheorem{remark}[theorem]{Remark}
\newcommand{\R}{\mathbb{R}}
\newcommand{\N}{\mathbb{N}}

\renewcommand{\bar}{\overline}
\renewcommand{\tilde}{\widetilde}
\renewcommand{\hat}{\widehat}

\newcommand{\Var}{\operatorname{Var}}
\newcommand{\Cov}{\operatorname{Cov}}

\numberwithin{equation}{section}
\usepackage[section]{placeins}

\def\Cov{{\rm Cov\,}}

\newcommand{\field}[1]{\mathbb{#1}}

\def\sig{\mathrm{sign}}
\def\H{{\mathbb H}}
\def\b{{\beta}}

\newcommand{\Z}{\field{Z}}

\newcommand{\Corr}{{\rm Corr}}

\def\<{\langle}
\def\>{\rangle}

\def\dTV{{d_{\mathrm{TV}}}}
\def\dW{{d_{\mathrm{W}}}}

\def\SS2{{\mathbb{S}^2}}

\def\E{{\mathbb{ E}}}

\def\P{{\mathbb{P}}}

\def\a{\alpha}

\def\1{\mathds 1}

\def\H{\mathbb H}
\def\HH{\mathbf H}

\begin{document}

\title{The multivariate fractional Ornstein-Uhlenbeck process}

\author[Ranieri Dugo]{Ranieri Dugo}
\address{
University of Rome Tor Vergata,
Department of Economics and Finance,  Via Columbia 2, 00133 Roma, Italy. 
}
\email{ranieri.dugo@students.uniroma2.eu}

\author[Giacomo Giorgio]{Giacomo Giorgio}
\address{
University of Rome Tor Vergata,
Department of Mathematics, 
Via della Ricerca Scientifica 1, I-00133 Roma, Italy}
\email{giorgio@mat.uniroma2.it}

\author[Paolo Pigato]{Paolo Pigato}
\address{
University of Rome Tor Vergata,
Department of Economics and Finance,  Via Columbia 2, 00133 Roma, Italy. 
}
\email{paolo.pigato@uniroma2.it}

\date{\today }

\begin{abstract}

Starting from the notion of multivariate fractional Brownian Motion introduced in [F. Lavancier, A. Philippe, and D. Surgailis. Covariance function of vector self-similar processes. Statistics \&
Probability Letters, 2009] we define a multivariate version of the fractional Ornstein--Uhlenbeck process. This multivariate Gaussian process is stationary, ergodic and allows for different Hurst exponents on each component. We characterize its correlation matrix and its short and long time asymptotics. Besides the marginal parameters, the cross correlation between one-dimensional marginal components is ruled by two parameters. We consider the problem of their inference, proposing two types of estimator, constructed from discrete observations of the process. We establish their asymptotic theory, in one case in the long time asymptotic setting, in the other case in the infill and long time asymptotic setting. The limit behavior can be asymptotically Gaussian or non-Gaussian, depending on the values of the Hurst exponents of the marginal components.
The technical core of the paper relies on the analysis of asymptotic properties of functionals of Gaussian processes, that we establish using Malliavin calculus and Stein's method. 
We provide numerical experiments that support our theoretical analysis and also suggest a conjecture on the application of one of these estimators to the multivariate fractional Brownian Motion.

\end{abstract}

\maketitle

\newpage

\section{Introduction} 
In this paper, we define a multivariate version of the fractional Ornstein-Uhlenbeck process (fOU), i.e. the solution to a stochastic differential equation (SDE) with affine drift and constant volatility, driven by a fractional Brownian motion (fBm). We define each component of this multivariate process as the fOU solution to a one-dimensional SDE driven by one component of the multivariate fBm (mfBm) introduced by Lavancier et al. in \cite{lav09}, which allows for different Hurst exponents on each component and for non-trivial interdependencies. The resulting multivariate fOU (mfOU) is a multivariate stationary and ergodic fractional process, with smoothness/regularity degree that can be different on each component. This process has a richer correlation structure than that of the classic Markovian case, in the sense that the correlation between the $i-th$ and the $j-th$ components depends on $\rho_{ij}$, analogous to the correlation coefficient of the Markovian case, and also on a parameter $\eta_{ij}$ that rules the time-reversibility of the process, which is also inherited from the mfBm.

We propose a method of moments estimator for these correlation parameters, $\rho=(\rho_{ij})_{i,j=1,\dots, d}$ and $\eta=(\eta_{ij})_{i,j=1,\dots, d}$, based on discrete observations. We study consistency and the asymptotic law of the rescaled errors, which can be normal or non-normal, depending on the value of the Hurst parameters, as the number of equally spaced observations of the process goes to infinity. This estimator presupposes the knowledge of the parameters of the marginal one-dimensional fOU processes. Even if not ideal, this seems a reasonable setting since the problem of estimating a one-dimensional fOU has already been widely considered in the literature both in theory and practice. A potential problem with this approach is, for example, the estimation of the mean reversion parameter, since errors in the estimation of the marginal processes would propagate in the estimation of the cross-correlation parameters. For this reason, leveraging a short-time expansion of the cross-covariance function, we also propose a modified estimator for $\rho$ and $\eta$ that does not depend on the mean reversion parameters of the one-dimensional marginal processes, for which we show consistency in the infill and long-span asymptotic setting. We also show asymptotic normality for this estimator of $\rho$, for values of the Hurst parameters in a certain interval. Since for this estimator we consider the infill asymptotics, we refer to this second estimator as ``high frequency estimator'' to differentiate it from the first ``low frequency estimator''.

In the one-dimensional case, an analogous derivation leads to two estimators for the volatility parameter of the fOU process, for which we provide the asymptotic theory as above. In particular, one of these two estimators, although derived in a different way, closely resembles the estimator of the volatility of  volatility parameter used in \cite{GJR18}, which was implemented there as a regression.

Finally, we perform a Monte Carlo study on the mfOU to evaluate in finite samples the goodness of the asymptotic theory of our estimators. We also test our ``high frequency estimator'' on the mfBM (that is, the mfOU with vanishing mean reversion) finding that the estimator performs well in the ``rough'' case (i.e., when the Hurst parameters are not too large). Our numerical results are consistent with
what we expect from our theoretical asymptotics, thus confirming their validity and viability in practical applications.

\emph{Related work:} The one-dimensional fOU process has been widely studied, starting from the work by Cheridito et al. \cite{cheridito2003}. To define its multivariate version, we combine it here with the mfBM, introduced in \cite{lav09} and thoroughly investigated in \cite{multivar,AC11,CAA13}. The resulting multidimensional fOU process could be interpreted as the solution to a multivariate fractional SDE, but not in the sense of the standard theory, which assumes the same Hurst parameter in all components \cite{Nualart2002}. The technical mathematical core of our paper relies on Malliavin calculus and Stein's methods for Gaussian processes, for which we refer to \cite{MN24,N13,NP12,PN12,Nua}, and use them to analyse the asymptotic distribution of functionals of stationary Gaussian processes, following \cite{T75, DM79, M81, A94, A00}.

The fOU process is relevant in several applications, notably in volatility modelling, following the groundbreaking work on fractional volatility by \cite{CR98} and on rough volatility by \cite{GJR18}. For other applications of fOU to (rough) volatility modeling we refer to
\cite{GS17,GS18,GS20hedging,GPP23,horvath_jacquier_lacombe_2019,bayer2016pricing,BFGHS}
 and for electricity prices modeling to \cite{mishura2023gaussian, giordano2019fractional}.
Motivated by these applications, the problem of estimating fOU parameters has received considerable attention both in the mathematical \cite{HN10,CL21,KLB02,HNZ19,HH21,XZX11} and the econometrical \cite{Bolko_et_al,WXY23,ES20,BIANCHI2023} community, where a particularly important issue is the estimation of the Hurst regularity parameter \cite{CHLRS2,CHLRS1,FTW,LMPR}. 
In practical scenarios, before using our estimators for the cross-correlation parameters of a mfOU, one should estimate the parameters of the marginals fOUs, following these methods. 
When estimating the Hurst parameter on different log-volatilities time series, one finds different values, all consistently lower than $1/2$. Therefore, a realistic multivariate model for the log-volatility should allow for different values of the Hurst parameters on different components, a feature that can be accommodated by the mfOU process we are proposing here \cite{dgp.eco}, while the available literature \cite{T06,M09}, even in a fractional (rough) setting  \cite{SHI2023173,Rosenbaum_Tomas_2021}, assumes that the Hurst parameter is uniform over the components. 
Besides volatility modelling, multivariate (fractional) time series with a flexible cross memory structure have applications in 
econometrics, physics, physiology, genomics and other sciences \cite{PDHGS06,PDHGS01,WPHS11,PS}. Concerning time-reversibility, the topic has been widely considered in the financial literature, and in general from the point of view of stochastic processes, mostly in the unidimensional setting, see e.g. \cite{zumbach2009time,cordi}.

\medskip

Let us write $\overset{d}{\to}$ to denote convergence in distribution of random variables.
We also denote by $o(f(\alpha))$ a function $g(\alpha)$ such that $g(\alpha)/f(\alpha)\to 0$, as $\alpha\to \bar\alpha$, and by $O(f(\alpha))$ a function $g(\alpha)$ such that $g(\alpha)/f(\alpha)$ is bounded in a neighborhood of $\bar\alpha$.

\medskip

\emph{Outline:} In Section \ref{sec:properties} we define the mfOU process and establish its main properties. In Section \ref{sec:estimation} we propose two types of estimators for its correlation parameters, establishing their asymptotic theory. 
In Section \ref{sec:numerics} we test these results on simulations, confirming the validity of the asymptotic theory and exploring possible extensions. In Section \ref{sec:proofs} we collect proofs and technical material. In the appendixes we recall several useful results and techniques we use in the paper.

\section{Definition of the mfOU process and main properties}\label{sec:properties}
	In this section we introduce  the mfOU process and state its main properties. Our definition of the mfOU is based on the mfBm defined in \cite{AC11,multivar}, that we recall in the next section.

\subsection{The multivariate fractional Brownian motion}\label{sec:mfbm}
We start recalling the definition and some properties of the mfBm, which is fundamental for our work. 
\begin{definition}\label{fBm2}
	Fixed $d\in\N$, the $d$-variate fractional Brownian motion ($d$-fBm) $(B^{H_1}_t, \ldots, B^{H_d}_t)_{t\in \R}$ with $H_i \in (0,1)$ for $i=1,\ldots, d$, is a centered Gaussian process taking values in $\R^d$ such that:
	\begin{itemize}
		\item $B^{H_i}$, for $i=1,\ldots, n$, is a fBm with Hurst index $H_i \in (0,1)$;
		\item it is a self-similar process with parameter $(H_1, \ldots, H_d)$, i.e. 
		$$
		(B^{H_1}_{\lambda t},\ldots,  B^{H_d}_{\lambda t})_{t\in \R}= (\lambda^{H_1} B^{H_1}_t, \ldots, \lambda^{H_n} B^{H_d}_t)_{t\in \R}
		$$ 
		in the sense of finite-dimensional distribution. 
		\item it has stationary increments.
	\end{itemize}
\end{definition}
The cross-covariance functions of a $d$-fBm have a precise form that is described in \cite{lav09,multivar}.
\begin{theorem}\label{cov_nfBm}
	Let $(B^{H_1}_t, \ldots, B^{H_d}_t)_{t\in \R}$, $H_i \in (0,1)$ for $i=1,\ldots, d$, be the $d$-fBm in Definition \ref{fBm2}. The cross-covariance functions have the following representations: 
	\begin{itemize}
		\item for $i\neq j$, if $H_{ij}=H_i+H_j \neq 1$, there exist $\rho_{ij}=\rho_{ji}\in [-1,1]$ and $\eta_{ij}=-\eta_{ji}\in \R$ such that $\rho_{ij}=\mathrm{Corr}(B^{H_i}_1, B^{H_j}_1)$ and 
		\begin{align}\label{cH}
			\Cov(B_t^{H_i}, B_s^{H_j})=\frac{\sigma_i \sigma_j }{2}\big((\rho_{ij}+&\sig(t)\eta_{ij})|t|^{H_{ij}}+(\rho_{ij}-\sig(s)\eta_{ij})|s|^{H_{ij}} \notag\\
			&-(\rho_{ij}-\sig(s-t)\eta_{ij})|s-t|^{H_{ij}}\big)
		\end{align}
		where $\sigma_i^2 =\Var(B^{H_i}_1),\, \sigma_j^2 =\Var(B^{H_j}_1)$ and $\sig:\R \rightarrow \{-1, 1\}$ is given by $\sig(x)=1$ for $x\geq 0$ and $\sig(x)=-1$ for $x<0$;
		\item for  $i\neq j$, if $H_i+H_j = 1$ there exist $\rho_{ij}=\rho_{ji}\in [-1,1]$ and $\eta_{ij}=-\eta_{ji}\in \R$ such that $\rho_{ij}=\mathrm{Corr}(B^{H_i}_1, B^{H_j}_1)$ and
		\begin{align}\label{cH1}
			\Cov(B_t^{H_i}, B_s^{H_j})=\frac{\sigma_i \sigma_j }{2}\big(\rho_{ij}&(|s|+|t|-|s-t|)+ \notag\\
			&+\eta_{ij}(s\log|s|-t\log|t|- (s-t)\log|s-t|)\big) 
		\end{align}
		where $\sigma_i^2 =\Var(B^{H_i}_1),\, \sigma_j^2 =\Var(B^{H_j}_1)$.
	\end{itemize}
\end{theorem}
The covariance structure of the mfBm is subject to numerous constraints due to its joint self-similarity property. This characteristic has been thoroughly examined in a broader context in studies such as \cite{lav09}, followed by more specific investigations in \cite{multivar} and \cite{AC11}. From Theorem \ref{cov_nfBm} one can see that the covariance structure relies on $d^2$ parameters: $(\rho_{ij})_{i,j=1, i\neq j}^d \in [-1,1]$, $(\eta_{ij})_{i,j=1,i\neq j}^d\in \R$ and $(\sigma_i)_{i=1}^d>0$. Here, $\rho_{ij}$ represents the correlation between $B^{H_i}_1$ and $B^{H_j}_1$, forming a symmetric parameter ($\rho_{ij}=\rho_{ji}$). The parameter $\sigma_i$ denotes the standard deviation of $B_1^{H_i}$ while $\eta_{ij}$ is antisymmetric ($\eta_{ij}=-\eta_{ji}$) and linked to the time reversibility of the process, see Remark \ref{rm:timerev} for a definition. 
In \cite{multivar}, the authors also investigated specific parameter choices such as $(\eta_{ij})_{i,j}$ depending on $(\rho_{ij})_{i,j}$ or when $\eta_{ij}=0$ (the time reversible case). We refer to Appendix \ref{sec:app_mfbm} for other properties of the mfBm.

\subsection{Definition of the mfOU process and alternative formulations}
	Using the mfBm defined above and the standard notion of univariate fOU process (recalled in Appendix \ref{sec:1dimfou}), we can now introduce our notion of multivariate fOU process (mfOU).
		\begin{definition}\label{fou}
			Let $d\in\N$, $\a_i,\nu_i>0$ for $i=1,\ldots, d$ and $H_1,\ldots, H_d\in (0,1)$. A multivariate fOU process (mfOU) $Y=(Y^1_t, \ldots, Y^d_t)_{t\in \R}$ is a centered Gaussian process such that, for all $t\in \R$, 
			\begin{equation*}
				Y_t^i=\nu_i\int_{-\infty}^t e^{-\a_i(t-s)}dB_s^{H_i} \qquad{i=1,\ldots, d}
			\end{equation*}
			where $(B^{H_1,\ldots, H_d}_t)_{t\in\R}=(B_t^{H_1},\ldots, B_t^{H_d})_{t\in\R}$ is a mfBm as in Definition \ref{fBm2}, with variance at time $1$ set to $\sigma^2_{1}=\cdots=\sigma^2_{d}=1$ (see Theorem \ref{cov_nfBm}).
		\end{definition}
		The integrals above are meant path-wise, in the Riemann-Stieltjes sense.
		With this definition, each component satisfies separately (is the stationary solution to) the SDE driven by the one-dimensional fBm
		\begin{equation}\label{Langeq.multi}
			dY_t^i = - \alpha_{i}  Y^{i}_t dt + \nu_{i} d B_t^{H_{i}}, \quad i=1,\ldots, d 
		\end{equation}
		where differentials are again to be interpreted in the sense of path-wise integration in the Riemann-Stieltjes sense (see \cite{cheridito2003}). 
		If $H_i \in (0,1)\setminus  \{1/2\},\, i=1,\ldots, d$  , each process $Y^{i}$ also has the moving average representation
		\begin{align}\label{movAverY}
				&Y^i_{t}=\sum_{j=1}^d \int_{\R} K^j_i(t,s) W_j(ds),
			\end{align}
where $W$ is a $d$-dimensional white noise, meaning that $W=(W_1,\ldots, W_d)$ is a $d$-dimensional Gaussian white noise with zero mean, independent components and covariance given by $\E[W_i(ds)W_j(ds)]=\delta_{ij}ds$, and 		
\begin{align*}
				&K^j_i(t,s)= \nu_i\Big( M_{ij}^+((t-s)_+^{H_i-\frac 12}-(-s)_{+}^{H_i-\frac 12})+M_{ij}^-((t-s)_-^{H_i-\frac 12}-(-s)_{-}^{H_i-\frac 12}) -\\
				&-\a_i\int_{s}^t e^{-\a_i(t-u)} \Big(M_{ij}^+((u-s)_+^{H_i-\frac 12}-(-s)_{+}^{H_i-\frac 12})+M_{ij}^-((u-s)_-^{H_i-\frac 12}-(-s)_{-}^{H_i-\frac 12}) \Big)du,
			\end{align*}
			and $M_{{ij}}^{\pm}$ given in \cite{multivar}, $i,j=1,\ldots, d$.
This follows from Theorem \ref{MovAverageFBM} and some standard computations. From this representation follows that 
			\begin{align*}
				\Var(Y^i_t)=\sum_{j=1}^d \int_{\R} K_i^j(t,s)^2 ds <\infty,
			\end{align*}
			since $K_i^j(t, \cdot) \in L^2(\R)$ for all $t\in\R$, and so $K_i(t, \cdot)=(K_i^1(t, \cdot),\ldots, K_i^d(t, \cdot)) \in L^2(\R;\R^d)$. 
			
Therefore, \emph{$Y$ is a Gaussian process}.

		\subsection{The covariance function of the mfOU}\label{prop_2fou_section}
		Let us assume from now on that $H_i \in (0,1)\setminus  \{1/2\}, i=1,\ldots, d$.
	The following theorem completely characterizes the mfOU process (which is Gaussian) using mean and covariance. The same result is also formulated using correlation functions instead of covariances. Let us write $H_{ij}:=H_i+H_j$, $i,j=1,\ldots, d$, $i\neq j$. 
		\begin{theorem}\label{StatC}
			The process $Y=(Y^1, \ldots, Y^d)$ is strongly stationary with 
			$
			\E[Y_t]=(0,\dots,0).
			$
			For $i,j\in\{1,\ldots, d\}$ and $t>s$, denoting $I_{ij}(t-s):=\int_0^{t-s} e^{\a_i u} \Big(\int_{-\infty}^{0} e^{\a_j v} (u-v)^{H_{ij}-2} dv\Big) du$, the covariance function $r_{ij}(t-s)=\Cov(Y_t^i, Y_s^j)$, is given as follows:
			\begin{itemize}
				\item when $H_{ij}=H_i+H_j\neq 1$ 
				\begin{align}\label{eq:crosscov}
					r_{ij}(t-s)&=e^{-\a_i(t-s)}\Cov(Y_0^i, Y_0^j)+\nu_i \nu_j e^{-\a_i (t-s)}H_{ij}(H_{ij}-1)\frac{\rho_{ij}+\eta_{ij}}{2} I_{ij}(t-s), 
				\end{align}
where
\begin{equation}\label{lag_0}
			\Cov(Y_0^i,Y^j_0)=\frac{\Gamma(H_{ij}+1)\nu_i\nu_j}{2(\a_i+\a_j)}\Big( (\a_i^{1-H_{ij}}+\a_j^{1-H_{ij}})\rho+(\a_j^{1-H_{ij}}-\a_i^{1-H_{ij}})\eta_{ij}     \Big);
		\end{equation}
				\item when $H_{ij}=H_i+H_j=1$, 
				\begin{align*}
					r_{ij}(t-s)=e^{-\a_i (t-s)}\Cov(Y^i_0,Y^j_0)-\nu_i\nu_j e^{-\a_i(t-s)} \frac{\eta_{ij}}{2}I_{ij}(t-s),
				\end{align*}
where
\begin{equation}\label{1_lag_0}
			\Cov(Y_0^i,Y^j_0)=\frac{\nu_i \nu_j }{\a_i+\a_j}\Big( \rho_{ij} + \frac{\eta_{ij}}{2}(\log \a_j-\log \a_i)\Big).
		\end{equation}
			\end{itemize}
Alternative (more convenient) expressions for the case $i=j$ are given in \cite{cheridito2003} and \cite{GS18}. The correlation $\Corr(Y_t^i, Y_t^j)$ is given by
		\begin{itemize}
			\item for $H_{ij}=H_i+H_j\neq 1$,   
			\begin{align*}
				\Corr(Y_t^i &,Y^j_t)=\frac{\Gamma(H_{ij}+1)}{\sqrt{\Gamma(2H_i+1)\Gamma(2H_j+1)}}\Big(\frac{\a_i^{H_i}\a_j^{H_j}}{\a_i+\a_j} \Big) \Big( (\a_i^{1-H_{ij}}+\a_j^{1-H_{ij}})\rho_{ij}+(\a_j^{1-H_{ij}}-\a_i^{1-H_{ij}})\eta_{ij}     \Big);
			\end{align*} 
			\item for $H_{ij}=H_i+H_j=1$,
			\begin{align*}
				\Corr(Y_t^i &,Y^j_t)=\frac{1}{\sqrt{\Gamma(2H_i+1)\Gamma(2H_j+1)}}\Big(\frac{2\a_i^{H_i}\a_j^{H_j}}{\a_i+\a_j} \Big) \Big( \rho_{ij} + \frac{\eta_{ij}}{2}(\log \a_j-\log \a_i)\Big).
			\end{align*} 
		\end{itemize}
\end{theorem}
The next results look at the asymptotic behavior of the cross-covariance $\Cov(Y^i_t, Y^j_{t+s})$, $i,j=1,\ldots, d$, $i\neq j$, when $s\to +\infty$ and when $s\to 0$. They will be key in the estimation of the cross-correlation parameters presented in Section \ref{sec:estimation}.
			\begin{theorem}\label{cov}
			Let $H_i\in (0,1)\setminus  \{1/2\}$,
			 $i,j\in\{1,\ldots, d\}$ and $H_{ij}=H_i+H_j\neq 1$, $i\neq j$  and $N\in \N$.
			  Then for fixed $t\in \R$, as $s\to \infty$ 
			\begin{align}\label{decCr}
				&\Cov(Y_t^i,Y_{t+s}^j)=\notag \\
				&=\frac{\nu_i\nu_j(\rho_{ij}+\eta_{ji})}{2(\a_i+\a_j)}\sum_{n=0}^N \Big(\frac{(-1)^n}{\a_j^{n+1}}+\frac{1}{\a_i^{n+1}}\Big)  \Big ( \prod_{k=0}^{n+1} (H_{ij}-k) \Big)s^{H_{ij}-2-n}+O(s^{H_{ij}-N-3}).
			\end{align}
		
			When $H_{ij}=1$ we have
			\begin{align}\label{decCr1}
				&\Cov(Y_t^i,Y_{t+s}^j)=\notag \\
				&=-\frac{\nu_i\nu_j\eta_{ij}}{2\a_i\a_j}\frac{1}{s}-\frac{\nu_i\nu_j\eta_{ij}}{2(\a_i+\a_j)}\sum_{n=1}^N \Big(\frac{(-1)^n}{\a_j^{n+1}}+\frac{1}{\a_i^{n+1}}\Big)  \Big( \prod_{k=0}^{n-1} (-k-1) \Big)s^{-1-n}+O(s^{-N-2}).
			\end{align}		\end{theorem}

A consequence of Theorem \ref{cov} is the \emph{cross-covariance-ergodicity of the mfOUs process} (detailed proof in Section \ref{sec:proofs}).		
		\begin{lemma}\label{cross-ergod}
		The multivariate process $Y$ is cross-covariance ergodic, i.e. for all $\tau\in \R$ and $i,j\in\{1,\ldots, d\}$, $i\neq j$, 
				\[
		\hat r_{ij}(\tau):=\frac 1{2T} \int_{-T}^T Y_{t+\tau}^i Y_{t}^j dt 
\to \E[Y_{\tau}^i Y_0^j]
\]
in probability, as $T\to \infty$  (recall $\E[Y_t^i]=0$, $i=1,\ldots, d$, $t\in \R$).
		\end{lemma}

In the short-lag asymptotic setting, we have the following lemma.
		\begin{lemma}\label{shorttime}
	For all $t\in \R$ and $s\to 0$, when $H_{ij}=H_i+H_j\neq 1$ and $i,j\in\{1,\ldots, d\}$, $i\neq j$, we have that
			\begin{align*}
				&\Cov(Y_t^i, Y_{t+s}^j)\\
				&=\Cov(Y_0^i, Y_0^j)-\nu_i \nu_j \frac{\rho_{ij}-\eta_{ij}}{2}s^{H_{ij}} +\Big(-\a_j \Cov(Y_0^i, Y_0^j)+\a_i^{1-H_{ij}} \Gamma(H_{ij}+1) \nu_i\nu_j\frac{\rho_{ij}-\eta_{ij}}{2}\Big) s+ \\
				&+\frac{(\a_j-\a_i)\nu_i \nu_j }{H_{ij}+1}\frac{\rho_{ij}-\eta_{ij}}{2} s^{1+H_{ij}}+\Big(\frac{\a_j^2}{2} \Cov(Y_0^i, Y_0^j)-\frac 12\nu_i\nu_j \frac{\rho_{ij}-\eta_{ij}}{2} \Gamma(H_{ij}+1)(\a_j\a_i^{1-H_{ij}} -\a_i^{2-H_{ij}})\Big)s^2\\
				&+o(s^{\max\{2,1+H_{ij}\}}).
			\end{align*}
			When $H_{ij}=1$ and $i\neq j$ we have
			\begin{align*}
				&\Cov(Y_t^i, Y_{t+s}^j)=\Cov(Y_0^i, Y_0^j)-\nu_i\nu_j \frac{\eta_{ij}}{2} s\log s +o(s^2\log s).
			\end{align*}
		\end{lemma}

\subsection{Comments}

\begin{remark}
In Definition \ref{fou}, when giving the moving average representation we exclude the possibility that $H_i=1/2$, for $i=1,\ldots, d$. We do so because in \cite{multivar} the moving average representation in \eqref{movAverY} is given for $H_i\neq 1/2$, for all $i\in 1,\ldots, d$. The authors conjecture that an analogous representation holds when there exists $i$ such that $H_i=1/2$, but this is not proved (see Remark 6 in \cite{multivar}).
\end{remark}

\begin{remark} \label{rm:timerev} 
The mfOU is time reversible if and only if $\alpha_{1}=\dots=\alpha_{d}$ and $\eta_{ij}=0$ for all $i,j=1,\dots,d$, as these conditions imply $r_{ij}(\tau)=r_{ji}(\tau)$ for all $i,j=1,\dots,d$.

We recall that time-reversibility amounts to temporal symmetry in the probabilistic structure of a strictly stationary process. More precisely, a process $(Z_t)_{t}$ is said to be time-reversible if the joint distributions of 
	$$
	(Z_t, Z_{t+\tau_1},\ldots, Z_{t+\tau_k})
	$$ 
	and 
	$$(Z_t, Z_{t-\tau_1},\ldots,Z_{t-\tau_k})$$ are equal for all $k\in \N$ and $\tau_1,\ldots, \tau_k \in \R$. 

\end{remark}

\begin{remark}
A special case of the mfOU process, named \emph{causal} because it depends only on the past realizations of the driving white noise, is obtained when the kernel that defines the moving average representation \eqref{movAverY} is characterized by $M_{ij}^{-}=0$. In this case, the cross-covariance depends only on one free parameter, say $\rho_{ij}$, and the other parameter $\eta_{ij}$ is directly deduced by the causality condition. See \cite{multivar} for details.
\end{remark}

		\begin{remark}
		Note that \eqref{Langeq.multi}, for $i=1,\dots,d$, can be read as a $d$-dimensional SDE driven by the mfBm in Section \ref{sec:mfbm}. However, fBm driven SDEs, in a multidimensional framework, have been considered only with a driving fBm noise with the same Hurst coefficient on each component ($H_{1}=H_{2}=\dots=H_{d}$), see e.g. \cite{Nualart2002}, while, to the best of our knowledge, SDEs driven by the mfBm in Section \ref{sec:mfbm} have never been considered. Since our Definition \ref{fou} of process $Y$ does not rely on this interpretation, we avoid further discussion on the topic here, and leave this for future work.
		\end{remark}

\begin{remark}
			Taking $\a_i=\a_j=\a$, $H_i=H_j$, $\nu_i=\nu_j$, $\rho_{ij}=1$ and $\eta_{ij}=0$ in Theorem \ref{cov} and Lemma \ref{shorttime} we recover the long-time asymptotics in the one-dimensional case  in Theorem \ref{Decay1} and the short-time asymptotics in the one-dimensional case in Lemma \ref{shorttime1dim}.
		\end{remark}

\section{Estimation of the correlation parameters}\label{sec:estimation}

In this section we consider the estimation of the cross-correlation parameters $\rho$ and $\eta$ based on discrete observations of $Y$. We assume the marginal one-dimensional parameters to be known, and we consider the correlation parameters between a fixed pair of marginal one-dimensional processes $Y^{1},Y^{2}$. To estimate the correlation parameters in the $d$-dimensional case, the procedure has to be repeated for all possible couples. Practically, in this way we have reduced the estimation problem of a $d$-dimensional fOU to the estimation of a bivariate fOU (2fOU). Therefore, from now on we consider a 2fOU with two correlation parameters $\rho=\rho_{12}=\rho_{21}$ and $\eta_{12}=-\eta_{21}$, that we aim to infer from discrete observations. We also denote $H=H_{12}=H_{1}+H_{2}$.

In Section \ref{Firstkind} we propose an estimator, obtained by inverting the expression for the cross-covariance in \eqref{eq:crosscov} for $\rho$ and $\eta$, as a function of the zero lag cross-covariance and the lagged cross-covariance. We develop an asymptotic theory in long time, relying on the ergodicity of the process. In Section \ref{Secondkind} we propose a variation of the estimator, obtained similarly, but instead of using the exact cross-covariance, we use the short-time approximation in Lemma \ref{shorttime}. The main practical difference in this approach is that the first terms in the asymptotic formula in Lemma \ref{shorttime}  do not depend on the mean reversion parameters $\alpha_{1},\alpha_{2}$. Therefore, one may  hope that this estimator is more robust to a poor estimation of $\alpha_{i}$ in the preliminary estimation of the one-dimensional marginal processes. Moreover, this estimator does not rely on the numerical integration for computing $I_{ij}(\tau)$, which can be delicate for certain choices of the parameters. To develop an asymptotic theory for this second estimator, since we leverage relation \eqref{shorttime}, we have to assume, as before, that the time horizon goes to infinity, but also that the time lag shrinks to $0$ (referred to as high frequency observations or infill asymptotics). To differentiate between these two type of estimator, we refer to the first one as  ``low frequency estimator'', and to the second one as ``high frequency estimator'', since we consider it in the infill asymptotic setting.

		\subsection{Low frequency estimator}\label{Firstkind}
		
		Let us consider the equations for the cross-covariance in \eqref{eq:crosscov} and \eqref{lag_0}. For fixed $t,s \in \R$, inverting the equations for $\rho\pm\eta_{12}$, recalling $\eta_{12}=-\eta_{21}$, we obtain
		\begin{align}\label{r+}
			&\rho+\eta_{12}=2\frac{\Cov(Y_{t+s}^1, Y_t^2)-e^{-\a_1 s}\Cov(Y_t^1, Y_t^2)}{\nu_1\nu_2 H(H-1)e^{-\a_1 s}I_{12}(s)}\\
			&\label{r-}
			\rho-\eta_{12}=2\frac{\Cov(Y_t^1, Y_{t+s}^2)-e^{-\a_2 s}\Cov(Y_t^1, Y_t^2)}{\nu_1\nu_2 H(H-1)e^{-\a_2 s}I_{21}(s)}.
		\end{align}
		Combining \eqref{r+} and \eqref{r-}, it follows that
		\begin{equation}\label{111}
			\rho=a_{1}(s) \,\Cov(Y_t^1,Y_t^2)+a_{2}(s)\,\Cov(Y_{t+s}^1,Y^2_t)+a_3(s)\,\Cov(Y^1_{t}, Y_{t+s}^2)
		\end{equation}
		and
		\begin{equation}\label{222}
			\eta_{12}=b_1(s)\,\Cov(Y_t^1,Y_t^2)+b_2(s)\,\Cov(Y_{t+s}^1,Y^2_t)+b_3(s)\,\Cov(Y_t^1, Y_{t+s}^2).
		\end{equation}
		where 
		\begin{align}\label{coefficients}
			&a_1(s)=-\frac{I_{12}(s)+I_{21}(s)}{\nu_1\nu_2H(H-1) I_{12}(s)I_{21}(s)},\\
			&a_2(s)=\frac{1}{\nu_1\nu_2  H(H-1)e^{-\a_1 s}I_{12}(s)},\notag\\
			&a_3(s)=\frac{1}{\nu_1\nu_2  H(H-1)e^{-\a_2 s}I_{21}(s)},\notag
		\end{align}
	and
		\begin{align}\label{coefficients2}
			&b_1(s)=\frac{I_{12}(s)-I_{21}(s)}{\nu_1\nu_2 H(H-1) I_{12}(s)I_{21}(s)},\\
			&b_2(s)=\frac{1}{\nu_1\nu_2 H(H-1)e^{-\a_1 s}I_{12}(s)},\notag\\
			&b_3(s)=-\frac{1}{\nu_1\nu_2 H(H-1)e^{-\a_2 s}I_{21}(s)}\notag.
		\end{align}

		One can also consider an analogous representation via correlations. The difference is not significant for the asymptotic theory that we develop in the present paper, but could be relevant from the point of view of applications. See next Remark \ref{rm:correlations} for details. 
Motivated by equations \eqref{111} and \eqref{222} and Lemma \ref{cross-ergod} ($Y$ is cross-covariance ergodic), we define the following estimators for $\rho$ and $\eta_{12}$ based on discrete observations, substituting the sample covariances to the theoretical ones.
		\begin{definition}\label{first_kind_estimator}
	Let $s\in \N$. Let us consider $Y_k=(Y_k^1, Y_k^2)$ for $k=0,\ldots, n$. We define 
		\begin{equation}\label{estRHO}
			\hat \rho_n=a_1(s)\,\frac{1}{n}\sum_{j=1}^n Y_j^1 Y_j^2+a_2(s)\,\frac{1}{n}\sum_{j=1}^{n-s }Y_{j+s}^1 Y_j^2+ a_3(s)\,\frac{1}{n}\sum_{j=1}^{n-s} Y_{j}^1 Y_{j+s}^2
		\end{equation}
		and
		\begin{equation}\label{estETA}
			\hat \eta_{12,n}=b_1(s) \,\frac{1}{n}\sum_{j=1}^n Y_j^1 Y_j^2+b_2(s)\,\frac{1}{n}\sum_{j=1}^{n-s} Y_{j+s}^1 Y_j^2+b_3(s)\,\frac{1}{n}\sum_{j=1}^{n-s} Y_{j}^1 Y_{j+s}^2.
		\end{equation}
	where $a_i(s), b_i(s), i=1,2,3$ are given in \eqref{coefficients} and \eqref{coefficients2}.
\end{definition}

		\begin{lemma}\label{bias_lemma}
			Let $n \in \N$. The estimators $\hat \rho_n$ in \eqref{estRHO} and $\hat \eta_{12, n}$ in  \eqref{estETA} are asymptotically unbiased estimators for $\rho$ and $\eta_{12}$ respectively, as $n\to \infty$. 
		\end{lemma}
		We consider now consistency and asymptotic normality for these estimators. For fixed $s$, let us fix three constant $a_{1},a_{2}, a_{3}$ and introduce
		\begin{equation}\label{nohatS_n}
		\begin{split}
S_n&=\frac {a_1}{n} \sum_{k=1}^n\Big(  Y_k^1Y_k^2-\E[Y_k^1Y_k^2]\Big)+\frac {a_2}{n} \sum_{k=1}^{n-s}\Big(Y_{k+s}^1 Y_k^2-\E[Y_{k+s}^1Y_k^2]\Big)\\
&+\frac {a_3}{n} \sum_{k=1}^{n-s}\Big( Y_k^1 Y_{k+s}^2 -\E[  Y_k^1 Y_{k+s}^2]  \Big).
\end{split}
\end{equation}

Remark that for $a_1=a_1(s)$, $a_2=a_2(s)$ and $a_3=a_3(s)$, with $a_1(s), a_2(s), a_3(s)$ in \eqref{coefficients}, one has $ S_n=\hat\rho_n-\rho$.  For $a_1=b_1(s)$, $a_2=b_2(s)$ and $a_3=b_3(s)$, with $b_1(s), b_2(s), b_3(s)$ in \eqref{coefficients2}, one has $ S_n=\hat\eta_{12,n}-\eta_{12,n}$. So, using $ S_n$, we can express the error of our estimators.

		\begin{theorem}\label{VAR_SN}
			Let $S_n$ be as in \eqref{nohatS_n}. Then, for $H<\frac 32$, 
	\begin{align*}
	&\lim_{n\to +\infty}	\Var(\sqrt{n} S_n)=\Var(a_1Y_0^1Y_0^2+a_2Y_{s}^1Y_0^2+ a_3 Y_0^1 Y_{s}^2)\\
	&+2\sum_{k=1}^{+\infty} \Cov(a_1Y_0^1Y_0^2+a_2Y_{s}^1Y_0^2+ a_3 Y_0^1 Y_{s}^2,a_1Y_k^1Y_k^2+a_2Y_{k+s}^1Y_k^2+ a_3 Y_k^1 Y_{k+s}^2)<+\infty.\\
\end{align*}
			For $H=\frac 32$, as $n\to \infty$, we have
			$$
			\Var(S_n)=O\Big(\frac{\log n}n\Big)
			$$
			and for $H>\frac 32$, as $n\to \infty$, we have
			$$
			\Var(S_n)=O\Big(\frac1{n^{4-2H}}\Big).
			$$
		\end{theorem}
The weak consistency of $\hat \rho_n$ and $\hat \eta_{12,n}$ follows from Theorem \ref{VAR_SN}.
				\begin{theorem}\label{consistency}
			Let $n\in \N$ and  $\hat \rho_n$, $\hat \eta_{12,n}$ given in \eqref{estRHO} and \eqref{estETA}. Then, for any $H\in (0,2)$, $\hat \rho_n $ and $\hat \eta_{12,n}$ converge to $\rho$ and $\eta_{12}$ in $L^2(\P)$ and so in probability.
		\end{theorem}
From next Lemma \ref{gammaAn} also follows the convergence in distribution of $\sqrt{n}S_n$ to a Gaussian random variable, see the detailed proof in Section \ref{sec:proofs}.
		\begin{theorem}\label{S_n_conv}
			Let $H<	\frac 32$ and $N\sim\mathcal N(0,\sigma^2)$, where $\sigma^2=\underset{n\to+\infty}{\lim}\Var(\sqrt{n}S_n)\in (0,+\infty)$. Then, as $n\to \infty$, 
			$$
			\sqrt{n}S_n \overset{d}{\to} N. 
			$$
		\end{theorem}
From Theorem \ref{S_n_conv}, follows the asymptotic normality of the estimators.
		\begin{theorem}\label{CLT_hat_rho_eta}
		Assume $H<\frac 32$.
			Let $\hat \rho_n$ in \eqref{estRHO} and $\hat \eta_n$ in \eqref{estETA}. Let $\sigma_\rho^2=\underset{n\to+\infty}{\lim}\Var(\sqrt{n}(\hat \rho_n-\rho))>0$ and $\sigma_\eta^2=\underset{n\to+\infty}{\lim}\Var(\sqrt{n}(\hat \eta_{12, n}-\eta_{12}))>0$. Then
	$$
			\sqrt{n}(\hat \rho_n-\rho)\overset{d}{\to} N_\rho
			$$
			and 
			$$
			\sqrt{n}(\hat \eta_{12,n}-\eta_{12})\overset{d}{\to} N_\eta
			$$
			where $N_\rho\sim \mathcal N(0,\sigma_\rho^2)$ and $N_\eta\sim\mathcal N(0,\sigma_\eta^2)$. Moreover 
			\begin{equation*}
				\sqrt{n}(\hat \rho_n- \rho, \hat \eta_{12, n}-\eta_{12}) \overset{d}{\to} (N_\rho, N_\eta)
			\end{equation*} 
			where $\Cov(N_\rho, N_\eta)=\lim_{n\to \infty} n\E[(\hat \rho_n-\rho)(\hat \eta_{12, n}- \eta_{12})]$.
		\end{theorem}

A result analogous to Theorem \ref{VAR_SN} holds in the case $H=3/2$, with rescaling given by $\sqrt{n/\log(n)}$ instead of $\sqrt{n}$ (see next Theorem \ref{CLT_hat_rho_eta32}). As a consequence we have the following results for our estimators.

\begin{theorem}\label{CLT_hat_rho_eta32}
		 Assume $H=\frac 32$.
			Let $\hat \rho_n$ in \eqref{estRHO} and $\hat \eta_n$ in \eqref{estETA}. Let $\sigma_\rho^2=\underset{n\to+\infty}{\lim}\Var(\sqrt{n/\log(n)}(\hat \rho_n-\rho))>0$ and $\sigma_\eta^2=\underset{n\to+\infty}{\lim}\Var(\sqrt{n/\log(n)}(\hat \eta_{12, n}-\eta_{12}))>0$.  Then
			$$
			\sqrt{\frac{n}{\log n}} (\hat \rho_n-\rho)\overset{d}{\to} N_\rho
		$$
		and
			$$
		\sqrt{\frac{n}{\log n}} (\hat \eta_{12,n}-\eta_{12})\overset{d}{\to} N_\eta
		$$
			where $N_\rho\sim \mathcal N(0,\sigma_\rho^2)$ and $N_\eta\sim\mathcal N(0,\sigma_\eta^2)$. Moreover 
			\begin{equation*}
				\sqrt{\frac{n}{\log n}}(\hat \rho_n- \rho, \hat \eta_{12, n}-\eta_{12}) \overset{d}{\to} (N_\rho, N_\eta)
			\end{equation*} 
			where $\Cov(N_\rho, N_\eta)=\lim_{n\to \infty} (n/\log n)\E[(\hat \rho_n-\rho)(\hat \eta_{12, n}- \eta_{12})]$
		\end{theorem}

When $H>\frac 32$ we have to consider a different rescaling of our estimators. Let us define 
	\begin{align}\label{newSn}
	\tilde S_n&=n^{2-H}\Big(\frac 1n\sum_{k=1}^n\big(  a_1Y_k^1Y_k^2+a_2Y_{k+s}^1 Y_k^2+a_3Y_k^1 Y_{k+s}^2 -\rho  \big)\Big)\notag\\
	&=n^{1-H}\sum_{k=1}^{n}\big(  a_1 Y_k^1Y_k^2+a_2Y_{k+s}^1 Y_k^2+a_3Y_k^1 Y_{k+s}^2 -\rho  \big).
\end{align}
 Adopting the approach outlined in \cite[\S 7.3]{N13}, we prove the following theorem. 
	\begin{theorem}\label{cumulantNC}
	Let $\tilde S_n$ be defined in \eqref{newSn} and $\kappa_p(\tilde S_n)$ be the cumulant of order $p$ of $\tilde S_n$. Let $a_1, a_2, a_3$ be such that $a_1+a_2+a_3\neq0$. Then, for all $p\geq 2$ 
	\begin{align*}
		&\lim_{n\to+\infty} \kappa_p(\tilde S_n)\\
		&=2^{p-1}(p-1)!\hat C_p(a_1, a_2, a_3)\sum_{i_1,\ldots, i_p=1}^2 \sum_{\underset{i_2\neq j_1,\ldots, i_p\neq j_{p-1}, i_1\neq j_p}{j_1,\ldots, j_p=1}}^2\int_{[0,1]^p}  z_{i_1 j_1}(x_1, x_2) z_{i_2 j_2}(x_2, x_3)\ldots z_{i_p j_p}(x_p, x_1) dx
	\end{align*}
	where 
	$$
	z_{11}(x,y)=  \frac{2\nu_1^2H_1(2H_1-1)}{\a_1^2} |x-y|^{2H_1-2}
	$$
	$$
	z_{22}(x,y)= \frac{2\nu_2^2H_2(2H_2-1)}{\a_2^2} |x-y|^{2H_2-2},
	$$
	$$
	z_{12}(x, y)=  \frac{2\nu_1\nu_2H(H-1) }{\a_1\a_2} \times \begin{cases}
		(\rho-\eta_{12})(x-y)^{H-2} \quad{x>y}\\
		(\rho+\eta_{12})  (y-x)^{H-2} \quad{x\leq y}
	\end{cases}
	$$
	$$
	z_{21}(x, y)=  \frac{2\nu_1\nu_2H(H-1) }{\a_1\a_2} \times \begin{cases}
		(\rho+\eta_{12})(x-y)^{H-2} \quad{x>y}\\
		(\rho-\eta_{12})  (y-x)^{H-2} \quad{x\leq y}
	\end{cases}
	$$
	and $\hat C_p(a_1, a_2, a_3)=(a_1+a_2+a_3)^p$. 
	Additionally
	\begin{align*}
		&\lim_{n\to +\infty} \frac{\Var(\tilde S_n)}{(a_1+a_2+a_3)^2}=\frac{64\Big((\rho^2-\eta_{12}^2)H^2(H-1)^2+4H_1H_2(2H_1-1)(2H_2-1)\Big)}{\a_1^{2}\a_2^{2}(2H-3)(2H-2)}>0.
	\end{align*}
\end{theorem}
In Theorem \ref{cumulantNC} we prove that the sequence of the cumulants of $\tilde S_n$ converges for all $p\geq 2$ to a well determined limit. We also prove that the limit of the sequence of the variances (cumulants of order two) is strictly positive. This is enough to determine the speed of convergence of the estimators, but not asymptotic normality, that now 
 depends on whether for certain values of the real parameters $\rho$ and $\eta_{12}$ we have that $\lim_{n\to \infty} \kappa_4(\tilde S_n)=0$ or not. In our numerical experiments (cf. Figures \ref{fig:clt_lf} and \ref{fig:nclt_lf}) we observe a limit behavior that is clearly asymmetric, pointing to non-normality. However, we cannot exclude that for some choices of $\rho$ and $\eta_{12}$ the limit behavior could be normal.
\begin{theorem}\label{th:nogaus}
		Let $\hat \rho_n$ in \eqref{estRHO} and $\hat \eta_n$ in \eqref{estETA}. Let us assume that $H>\frac 32$. Let $z_{11}, z_{22}, z_{12}$ and $z_{21}$ be the functions given in Theorem \ref{cumulantNC}. Then 
		\begin{align*}
			&n^{2-H}(\hat\rho_{n}-\rho)\overset{d}{\to} I_2(f_\rho)+N_{\rho}\\
			&n^{2-H}(\hat\eta_{12,n}-\eta_{12})\overset{d}{\to} I_2(f_\eta)+N_{\eta}
		\end{align*}
		where $f_\rho, f_\eta \in L^2(\R^2; \R)$, $N_{\rho}, N_{\eta}$ are Gaussian random variables such that $N_{\rho}$ is independent of $I_2(f_\rho)$ and $N_{\eta}$ is independent of $I_2(f_{\eta})$. Here, $I_{2}$ denotes the double Wiener-It\^o integral, see Definition \ref{def:mul.int}. 
		\end{theorem}

\subsection{High frequency estimator}\label{Secondkind}
	Using the results in Lemma \ref{shorttime}, we derive formulas for $\rho$ and $\eta_{12}$ exploiting the cross-covariances at lag $0, s, -s$, which hold asymptotically as the time lag vanishes, $s\to 0$. In Lemma \ref{shorttime} we consider the expansion of the covariance up to terms of order $1+H$ and $2$, because the next expression \eqref{rho_s} relies on the cancellation of the term of order $1$, so these are the relevant terms in the remainder. 
		\begin{lemma}\label{short-time-inverse}
			 As $s\to 0$, $\rho$ and $\eta_{12}$ satisfy the following formulas:
			for all $H\in(0,2)$, $H\neq 1$, 
			\begin{equation}\label{rho_s}
				\rho=\frac{2\Cov(Y_0^1, Y_0^2)-\Cov(Y_s^1,Y_0^2)-\Cov(Y_0^1,Y_s^2)}{\nu_1\nu_2s^H}+O(s^{\min(1,2-H)})
			\end{equation}
			whereas, only for $H\in(0,1)$, 
			\begin{equation}\label{eta_s}
				\eta_{12}=
				\frac{\Cov(Y_0^1,Y_s^2)-\Cov(Y_s^1,Y_0^2)}{\nu_1\nu_2s^H}	+O(s^{1-H}).
			\end{equation}
		\end{lemma}

Let us consider $n\in\mathbb{N}$ and $T_n>0$, along with the discretization of the time interval $[0,T_n]$ given by $t_k^n=k\frac{T_n}{n}$ for $k=0,\ldots,n$. We denote the time step as $\Delta_n=\frac{T_n}{n}$. Motivated by Lemma \ref{short-time-inverse}, we define the estimators based on $n$ discrete observations
		\begin{equation}\label{hat_rho_n}
			\tilde\rho_n=\frac{1}{\nu_1\nu_2n\Delta_n^H} \sum_{k=0}^{n-1}\Big(Y_{(k+1)\Delta_n}^1-Y_{k\Delta_n}^1\Big)\Big(Y_{(k+1)\Delta_n}^2-Y_{k\Delta_n}^2\Big)
		\end{equation}
		and 
		\begin{equation}\label{hat_eta_n}
			\tilde\eta_{12,n}=\frac{1}{\nu_1\nu_2n\Delta_n^H} \sum_{k=0}^{n-1}\Big(Y_{k\Delta_n}^1Y_{(k+1)\Delta_n}^2-Y_{(k+1)\Delta_n}^1Y_{k\Delta_n}^2\Big).
		\end{equation}
		We consider the asymptotic framework of $\Delta_n\to 0$, as $n\to +\infty$, in order to take advantage of the small lag asymptotic relations \eqref{rho_s} and \eqref{eta_s}. Recall that in the estimators in \eqref{estRHO} and \eqref{estETA} the time-lag was fixed. 
				\begin{proposition}\label{unbiasEst}
			Let $\tilde\rho_n$ and $\tilde\eta_{12,n}$ be the random variables in \eqref{hat_rho_n} and \eqref{hat_eta_n}. If $T_n\to +\infty$ and $\Delta_n\to 0$ as $n\to \infty$, then for $H\in(0,2)\setminus\{1\}$
			$$
			\E[\tilde\rho_n] \underset{n\to +\infty}{\to} \rho
			$$
			and for	$H\in(0,1)$
			$$
			\E[\tilde\eta_{12,n}]\underset{n\to +\infty}{\to} \eta_{12}. 
			$$
			Then $\tilde\rho_n$ and $\tilde\eta_{12,n}$ are asymptotically unbiased estimators for $\rho$ and $\eta_{12}$.
		\end{proposition}

Let us consider the following	assumptions.
	\begin{assumption}\label{AssumptionDelta_a}
			As $n\to +\infty$
			\begin{itemize}
				\item[1)] $\Delta_n\to 0$ and
				\item[2)] $n\Delta_n\to +\infty$.
		   \end{itemize}
	   \end{assumption}
   \begin{assumption}\label{AssumptionDelta_b}
 As $n\to+\infty$
	        \begin{itemize}
				\item[3)] $n\Delta_n^2 \to 0$ and
				\item[4)] $n \Delta_n^{4-2H}\to 0$. 
			\end{itemize}

			Note that when $H<1$, we have that $3)$ implies $4)$, while when $H>1$,  we have that  $4)$ implies $3)$. 
		\end{assumption}	
		
Under Assumptions \ref{AssumptionDelta_a}, respectively for $H\in(0,2)\setminus\{1\}$ and for	$H\in(0,1)$, estimators $\tilde \rho_n$ and	 $\tilde \eta_{12, n}$ are consistent.
\begin{theorem}\label{Cons_rho}
Suppose Assumption \ref{AssumptionDelta_a} is in force. As $n\to \infty$, the following convergences hold in $L^2(\P)$ and so in probability.			
	\[
	\begin{split}
	\tilde \rho_n&\to \rho \quad\quad\quad\mbox{ if } H\in(0,2)\setminus \{1\} \\
	\tilde \eta_{12, n}&\to \eta_{12}  \quad\quad\mbox{ if }H\in(0,1)
	\end{split}
	\]				
	\end{theorem}

When $H<\frac 32$ we have the following result on  the asymptotic distribution of $\sqrt{n}(\tilde \rho_n-\rho)$.  Let us denote 
			$$
			\sigma_n^2=\Var\Big(\frac{1}{n}\sum_{k=0}^{n-1} (B_{k+1}^{H_1}-B_k^{H_1})(B_{k+1}^{H_2}-B_k^{H_2})\Big).
			$$
		
		\begin{theorem}\label{CLT_tilde_rho}
			Let $H<\frac 32$ and Assumption \ref{AssumptionDelta_a} and \ref{AssumptionDelta_b} be in force.
			Then there exists 
			\begin{align*}
				\lim_{n\to +\infty} \sigma_n^2 =\Var(B_1^{H_1} B_1^{H_2})+2\sum_{k=1}^{+\infty} \Cov\Big(B_1^{H_1} B_1^{H_2}, (B_{k+1}^{H_1}-B_k^{H_1})(B_{k+1}^{H_2}-B_k^{H_2})\Big)=:\sigma^2>0.
			\end{align*}
			Let $N\sim \mathcal N(0,\sigma^2)$. Then
			\begin{equation*}
				\sqrt{n}(\tilde \rho_n-\rho) \overset{d}{\to}N.
			\end{equation*}
		\end{theorem}

\subsection{Estimating the volatility parameter in the one-dimensional case}

The methods we used above to derive and study our correlation estimators, when applied to the univariate fOU process in \eqref{fOU1}, provide estimators of the volatility parameter $\nu$.
We denote the parameters of the univariate fOU as $\alpha, \nu$ and $\HH$, to avoid confusion in the notation with $H=H_{1}+H_{2}$. 

One can easily verify that
$$
\nu^2=\frac{\Cov(Y_{t+s}, Y_t)-e^{-\a s}\Var(Y_t)}{\HH(2\HH-1)e^{-\a s} I(s)}
$$ 
where 
$$
I(s)=\int_0^s e^{\a u} \int_{-\infty}^0 e^{\a v} (u-v)^{2\HH-2} dv du.
$$
Then, we define the estimator 
\begin{align*}
\hat\nu_n^2& = \frac{1}{n\HH(2\HH-1)e^{-\a s} I(s)} \sum_{k=0}^{n-1} Y_{k+s} Y_k-e^{-\a s} Y_k^2  \\
&=\frac{\overline a_1(s)}{n} \sum_{k=0}^{n-1}Y_{k}^2 +\frac{\overline a_2(s)}{n} \sum_{k=0}^{n-s} Y_{k+s} Y_k 
\end{align*}
where $\overline a_1(s)=\nu^2 a_1(s)$ and $\overline a_2(s)=\nu^2 (a_2(s)+a_3(s))$, where $a_1(s), a_2(s), a_3(s)$ are given in \eqref{coefficients}, taking $\HH=H_1=H_2$, $\a=\a_1=\a_2$, $\nu=\nu_1=\nu_2$, $\rho=1$ and $\eta_{12}=0$. 
\begin{theorem}\label{NU_1}
	Let $\HH\in (0,\frac 12)\cup (\frac 12, 1)$. As $n\to +\infty$,  $\hat \nu_n^2$ is an asymptotically unbiased estimator for $\nu^2$. Moreover 
	$$
	\hat \nu_n^2 \to \nu^2
	$$
	in $L^2(\P)$ and then in probability. 
	\begin{itemize}
	\item
	When $\HH<\frac 34$, we also have that 
	$$
	\sqrt{n}(\hat \nu_n^2- \nu^2) \overset{d}{\to} N
	$$
	where $N\sim \mathcal N(0,\sigma^2)$ and $\sigma^2=\lim_{n\to \infty} \Var(\hat \nu_n^2)>0$. 
	\item
	When $\HH=\frac 34$, we have that
	$$
	\sqrt{\frac{n}{\log n}}(\hat\nu_n^2-\nu^2)\overset{d}{\to} N
	$$
	where $N\sim \mathcal N(0,\sigma^2)$ and $\sigma^2=\lim_{n\to \infty} \Var(\hat \nu_n^2)>0$.
	\item
	When $\HH>\frac 34$, we have that
	$$
n^{2(1-\HH)}(\overline \nu_n^2 -\nu^2)\overset{d}{\to} R^\HH,
$$
where
	$$
	R^\HH=I_2(g)
	\mbox{ 
and }
	g(t,t')=\frac{2(1-e^{-\a s})\nu^2}{\a^{2}e^{-\a s}I(s)\beta(2-2\HH, \HH-1/2)}\int_0^1(u-t)_+^{\HH-\frac 32}(u-t')_+^{\HH-\frac 32} du.
$$ 
\end{itemize}
\end{theorem}
The next estimator  was introduced ``in spirit''  in \cite{GJR18} using a linear regression, but no asymptotic theory was given. Here we motivate it using the short lag asymptotic form of the autocovariance function and establish its infill and long-span asymptotic theory. 
Indeed, by Lemma \ref{shorttime1dim} it follows
$$
\nu^2-2\frac{\Var(Y_t)-\Cov(Y_t, Y_{t+s})}{s^{2\HH}}=O(s^{2-2\HH}).
$$
Passing to the empirical counterpart and removing the remainder, one gets the estimator
\begin{equation}\label{nu_est}
\tilde	\nu_n^2=\frac{1}{n\Delta_n^{2\HH}} \sum_{k=0}^{n-1} (Y_{(k+1)\Delta_n}-Y_{k\Delta_n})^2.
\end{equation}
that, as we mentioned, corresponds to the one implemented as a regression in \cite{GJR18}.
\begin{theorem}\label{NU_2}
	Let $\HH\in (0,\frac 12)\cup (\frac 12, 1)$. Supposing that $\Delta_n\to 0$ as $n\to \infty$, then $\tilde\nu_n^2$ is asymptotically unbiased estimator for $\nu^2$, i.e. $\E[\tilde\nu_n^2]\to \nu^2$. If in addition $T_{n}=n\Delta_n\to +\infty$ as $n\to +\infty$, then 
	$$
	\tilde\nu_n^2\to \nu^2
	$$ 
	in $L^2(\P)$ and then in probability.
	Moreover, when $\HH<\frac 34$, supposing also that 
	$n\Delta_n^2\to 0$ and $n\Delta_n^{4-4\HH}\to 0$, we have that
	$$
	\sqrt{n}(\tilde \nu^2_n-\nu^2)\overset{d}{\to} N
	$$  
	where $N\sim \mathcal N(0,\sigma^2)$ and $\sigma^2=\lim_{n\to \infty} \Var(\sqrt{n}(\tilde\nu_n^2-\nu^2))$. 
\end{theorem}

\subsection{Comments}

\begin{remark}[The estimators in practice]
When using estimators $\hat\rho$ and $\hat\eta$ in practical situations, where the univariate marginal parameters are not known, one should first estimate the univariate parameters with some consistent estimators, and then plug-in these estimates in the estimators $\hat\rho$ and $\hat\eta$. 
In the case of the scale coefficients, $\nu_i$, we expect that substituting consistent estimators in place of the actual parameters should not to be a problem,
 since the $\nu_i$'s enter the expressions for $\hat\rho$ and $\hat\eta$ multiplicatively. So, consistency should still hold.
However, the speed of mean-reversion coefficients and the Hurst exponents parametrize the integrand in  $I_ij(t-s)$, so that it is hard to say whether consistency is preserved. 

Alternatively, based on the convergence of the empirical cross covariances in Lemma \ref{cross-ergod}, one can for example use the Generalized Method of Moments, which infers all the parameters at once by considering an overdetermined system of moment conditions (cross-covariances), and which delivers a lower asymptotic variance \cite{hansen1982large,dgp.eco}. 
\end{remark}

\begin{remark}[Convergence of the error]
Together, Theorems \ref{CLT_hat_rho_eta}, \ref{CLT_hat_rho_eta32} and \ref{th:nogaus} imply that \\ $O((\E|\hat{\rho}_{n} -\rho|^{2})^{1/2}) =n^{-\min\{1/2,2-2H\}\}}$, except for $H=3/2$, in which case a logarithmic correction is present. See Figure \ref{fig:conv_rates_lf}.
\end{remark}

\begin{remark}[On the dependence on the Hurst exponents]
The fact that the behavior in Theorem \ref{th:nogaus} is different depending on the position of $H=H_{1}+H_{2}$ with respect to $3/2$ is in line with known statistical results on fractional processes, where, e.g. in the unidimensional case, the behavior depends on the position of the Hurst parameter with respect to $3/4$. For example, in \cite{HN10,HNZ19}, several statistical estimators for the parameters of a uni-dimensional fOU  are examined, finding central limit theorems when the Hurst parameter is below $3/4$, slower speed of convergence and Rosenblatt limit when the Hurst parameter is above $3/4$. In our situation, one can interpret the change of behavior as that when $H<3/2$ the memory is weak enough so that in the limit we get Gaussian fluctuations, while when $H>3/2$ the strong memory is visible in the limit, which appears more ``slowly''
and is non-Gaussian. $H=3/2$ is the ``phase transition''. See also \cite{Pipiras_Taqqu_2017} for more on this phenomenon.
\end{remark}

\begin{remark}
The Gaussian results such as the in Theorem \ref{CLT_hat_rho_eta} could be used to compute standard errors and associated confidence intervals. To do so, the asymptotic variance can be computed using the moving block bootstrap for functionals of (long-memory) dependent sequences with Gaussian limit given in \cite{A94}. 
\end{remark}

\begin{remark}[Estimators based on sample correlations]\label{rm:correlations}
The estimators in Definition \ref{first_kind_estimator} are motivated by \eqref{111} and \eqref{222}. Analogous expressions can be obtained based on sample correlations instead of sample covariances. Indeed, starting from Theorem \ref{StatC}, after normalizing \eqref{r+} and \eqref{r-} with $\sqrt{\Var(Y^{1}_{0})\Var(Y^{2}_{0})}$, one can easily write\begin{itemize}
	\item when $i=j$
	\begin{align*}
		\Corr(Y_t^i, Y_s^i)= \frac{\a_i^{2H_i}\sin \pi H_i}{2\pi} \int_{-\infty}^{\infty} e^{i(t-s)x} \frac{|x|^{1-2H_i}}{\a_i^2+x^2}dx;
	\end{align*} 
	\item when $i\neq j$ and $H=H_1+H_2\neq 1$, 
	\begin{align}\label{corr_0_exp}
		\Corr(Y_t^i,& Y_s^j)\notag\\&=e^{-\a_i(t-s)}\Corr(Y_0^1, Y_0^2)+\frac{H(H-1)\a_1^{H_1} \a_2^{H_2} e^{-\a_i (t-s)}}{\sqrt{\Gamma(2H_1+1)\Gamma(2H_2+1)}}(\rho+\eta_{ij}) I_{ij}(t-s);
	\end{align}
	\item when $i\neq j$ and $H=H_1+H_2=1$, 
	\begin{align*}
		\Corr(Y_t^i,& Y_s^j)=e^{-\a_i (t-s)}\Corr(Y^1_0,Y^2_0)- \frac{\a_1^{H_1}\a_2^{H_2} e^{-\a_i(t-s)}}{\sqrt{\Gamma(2H_1+1)\Gamma(2H_2+1)}} \eta_{ij} I_{ij}(t-s);
	\end{align*}
\end{itemize}
 which is explicit once the marginal parameters are know, following \cite{cheridito2003}.
		In this case we have
		\begin{align*}
			\rho&=-\frac{\sqrt{\Gamma(2H_1+1)\Gamma(2H_2+1)}}{2\a_1^{H_1}\a_2^{H_2}H(H-1)}\Big(\frac{1}{I_{12}(s)}+\frac{1}{I_{21}(s)}\Big)\Corr(Y_t^1, Y_t^2)\\
			&+\frac{\sqrt{\Gamma(2H_1+1)\Gamma(2H_2+1)}}{2\a_1^{H_1}\a_2^{H_2}H(H-1)}\frac{e^{\a_1 s}}{I_{12}(s)}\Corr(Y^1_{t+s},Y^2_t)\\
			&+\frac{\sqrt{\Gamma(2H_1+1)\Gamma(2H_2+1)}}{2\a_1^{H_1}\a_2^{H_2}H(H-1)}\frac{e^{\a_2 s}}{I_{21}(s)}\Corr(Y^1_t, Y^2_{t+s})\\
		\end{align*}
		and
		\begin{align*}
			\eta_{12}&=-\frac{\sqrt{\Gamma(2H_1+1)\Gamma(2H_2+1)}}{2\a_1^{H_1}\a_2^{H_2}H(H-1)}\Big(\frac{1}{I_{12}(s)}-\frac{1}{I_{21}(s)}\Big)\Corr(Y_t^1, Y_t^2)\\
			&+\frac{\sqrt{\Gamma(2H_1+1)\Gamma(2H_2+1)}}{2\a_1^{H_1}\a_2^{H_2}H(H-1)}\frac{e^{\a_1 s}}{I_{12}(s)}\Corr(Y^1_{t+s},Y^2_t)\\
			&-\frac{\sqrt{\Gamma(2H_1+1)\Gamma(2H_2+1)}}{2\a_1^{H_1}\a_2^{H_2}H(H-1)}\frac{e^{\a_2 s}}{I_{21}(s)}\Corr(Y^1_t, Y^2_{t+s}).\\
		\end{align*}
Moreover, if we suppose a priori that we know that $\eta=0$, we have the simpler representation for $\rho$
		\begin{align*}
			\rho=\frac{\sqrt{\Gamma(2H_1+1)\Gamma(2H_2+1)}}{\Gamma(H+1)}\frac{\a_1+\a_2}{\a_1^{H_1}\a_2^{H_2}(\a_1^{1-H}+\a_2^{1-H})} \Corr(Y_t^1, Y_t^2).
		\end{align*}
Substituting sample correlations in place of theoretical correlations, we can obtain estimators for $\rho$ and $\eta$ similar to $\hat{\rho}_{n}$ and $ \hat\eta_{12,n}$. Let us denote these estimators based on sample correlations instead of covariances by $\hat{\rho}_{n,c}$ and
$\hat\eta_{12,n,c}$.	
		
		Let us denote by $\bar{\Var}$ and $\bar{\Cov}$ sample variances and covariances and consider $\hat{\rho}_{n,c}$ (a similar discussion would apply to $ \hat\eta_{12,n}$). We can then write, for suitable functions $\gamma_{1}(s),\gamma_{2}(s),\gamma_{3}(s)$, 
\[
\hat{\rho}_{n,c} = \sqrt{\frac{{\Var}(Y^{1}_{0}){\Var}(Y^{2}_{0})}{\bar{\Var}(Y^{1}_{0})\bar{\Var}(Y^{2}_{0})}}\,\big(
\gamma_{1}(s)
\bar\Cov(Y_t^1, Y_t^2)
+
\gamma_{2}(s)\bar\Cov(Y^1_{t+s},Y^2_t)
+
\gamma_{3}(s)\bar\Cov(Y^1_t, Y^2_{t+s})\big).
\]		
The second factor in the product above is exactly in the form  \eqref{nohatS_n} and can be handled as in the proof of Theorem \ref{VAR_SN}. 
Up to logarithmic corrections, it also holds that $O(\bar{\Var}(Y^{i}_{0})-{\Var}(Y^{i}_{0}))=n^{-\min\{1/2,2-2H_{i}\}}$, for 
		$i=1,2$. Putting together these estimates, with standard estimations one can check that
		\[		 
O((\E|\hat{\rho}_{n,c} -\rho|^{2})^{1/2}) =n^{-\min\{1/2,2-2\max\{H_{1},H_{2}\}\}}
		\]
		and the order of the error depends on $2\max\{H_{1},H_{2}\}$ instead of $H=H_{1}+H_{2}$ (cf. Theorems \ref{CLT_hat_rho_eta}, \ref{CLT_hat_rho_eta32}, \ref{th:nogaus}). So, the order of the error is the same for $\hat{\rho}_{n}$ and $\hat{\rho}_{n,c}$, except when  $H_{1}\neq H_{2}$ and  $\max\{H_{1}, H_{2}\}>3/4$, in which case the estimator based on covariances should perform better than the one based on correlations, in terms of the order of the error.
		\end{remark}

\begin{remark}

In Theorem \ref{th:nogaus}, if $\rho$ and $\eta_{12}$ are such that
			$$
			\sum_{i_1, i_2, i_3, i_4=1}^2 \sum_{\underset{i_2\neq j_1,i_3\neq j_2, i_4\neq j_{3}, i_1\neq j_4}{j_1,\ldots, j_4=1}}^2\int_{[0,1]^4}  z_{i_1 j_1}(x_1, x_2) z_{i_2 j_2}(x_2, x_3) z_{i_3, j_3}(x_3, x_4)z_{i_4 j_4}(x_4, x_1) dx =0,
			$$
			then $f_\rho=f_\eta=0$ and 
			\begin{align}\label{Norm>32} 
			&n^{2-H}(\hat \rho_n-\rho, \hat \eta_{12,n}-\eta_{12})\overset{d}{\to } ( N_{\rho}, N_{\eta})
			\end{align}
			where $(N_\rho, N_{\eta})\sim N(0, \Sigma)$ is a 2-dimensional Gaussian variable, with covariance matrix given by
			\begin{align*}
				&\Sigma_{11}=\lim_{n\to \infty}\Var(n^{2-H}(\hat\rho_n-\rho)), \\
				&\Sigma_{22}=\lim_{n\to \infty}\Var(n^{2-H}(\hat\eta_{12,n}-\eta_{12})),\\
				& \Sigma_{12}=\Sigma_{21}=\lim_{n\to \infty} n^{4-2H}\Cov(\hat\rho_n-\rho,\hat\eta_{12,n}-\eta_{12}).
			\end{align*}
In Theorem \ref{th:nogaus} we expect this not to be the case (so, $f_\rho$ and $f_\eta$ are not $0$) for most of the choices of $\rho$ and $\eta_{12}$. Indeed, there is no reason to assume the Gaussian asymptotics in \eqref{Norm>32} is in force when $H>3/2$. This is indirectly confirmed by our numerical experiments, which display a non-Gaussian, asymmetric limit distribution when $H>3/2$ (see Figure \ref{fig:nclt_lf}). 
\end{remark}

		\begin{remark}
		We define in \eqref{hat_rho_n} and \eqref{hat_eta_n} estimators for $\rho$ and $\eta_{12}$ based on \eqref{rho_s} and \eqref{eta_s}. One main difference of these estimators compared to the ones defined in \eqref{estRHO} and \eqref{estETA} is that they do not depend on $\a_1$ and $\a_2$. Since we estimate the correlation parameters $\rho$ and $\eta_{12}$ supposing the parameters of the one-dimensional marginals to be known, this can be useful if the previous estimate of $\a_1$ and $\a_2$ is poor, as for example in the case of volatility time series (see \cite{WXY23,dgp.eco}).

For $\eta_{12}$, when $H>1$, we could consider the linear term in $s$ and invert it, resulting in a different formulation.  However, this estimator would depend on $\a_1$ and $\a_2$. Consequently, in the high frequency setting, for the estimator for $\eta_{12}$ we confine our discussion to  $H<1$. Note that this should not be a significant limitation if we plan to use the estimators on log-volatility time series, where we expect $H=H_1+H_2<1$, see for example \cite{GJR18}. Moreover, we exclude the case $H=1$ from our analysis, because for this singular value of the $H$ parameter the time scaling is different and involves a logarithm, see Lemma \ref{shorttime}. 
		\end{remark}

\begin{remark}[On the high-frequency and long-span hypothesis] 
Assumption \ref{AssumptionDelta_a} is common in statistics of continuous time processes. Part 1) has the clear interpretation that the time lag between observations has to vanish as $n\to \infty$, which is the so called high frequency observations setting, while part 2) has the interpretation that the time horizon $n\Delta_n$ goes to infinity, which is the so called long-span asymptotics. In our case, both are in force, as for example in \cite{kessler,alaya_kebaier_2013,amorino_glotier,mazzonetto2020drift}. Assumption \ref{AssumptionDelta_b} is stronger and requires that the time lag vanishes fast enough with respect to the speed at which the time horizon goes to infinity. Assumptions of this type (in particular, analogous to part 3)) are also common when proving asymptotic normality and speed of convergence results;  Part 4) is similar in spirit, but it involves the Hurst exponent $H$, reflecting the fractional nature of our setting. See again \cite{kessler,alaya_kebaier_2013,amorino_glotier,mazzonetto2020drift}. 
\end{remark}

\begin{remark}
We develop the asymptotic theory of the estimators in \eqref{rho_s} and \eqref{eta_s} under the high frequency hypothesis $\Delta_{n} \to 0$. However, the estimator should work if the linear relation between $\Cov(Y_t^i, Y_{t+s}^j)$ and $s^H$ is (approximately) in force in the data. For example, for certain realized volatility time series we observe linearity for time lags up to 90 days in \cite{dgp.eco}.
\end{remark}

\section{Simulation of the process and implementation of the estimators}\label{sec:numerics}

In this section we evaluate the asymptotic results of Section \ref{sec:estimation} with a Monte Carlo study on trajectories of finite length. It is possible to simulate the mfOU process exactly using the explicit covariance function given in Theorem \ref{prop_2fou_section} and the Cholesky method for multivariate Gaussian random variables. However, when simulating very long trajectories, we encounter some instabilities with this approach due to the numerical integration needed to compute the covariance.  Alternatively, the process can be simulated using the Euler-Maruyama scheme, after having simulated the underlying mfBM (for example using the 
algorithm\footnote{We downloaded the code for simulating the mfBm at https://sites.google.com/site/homepagejfc/software}
 in \cite{multivar}). We do not encounter any difficulty in generating arbitrarily long trajectories with this second method.  

The figures we present here are obtained with the exact simulation method, which performs well with 
our parameter choices for producing trajectories up to $T=400$. This time horizon seems long enough to observe the long time asymptotic behavior, while avoiding the introduction of the discretization error of the Euler-Maruyama scheme. One can reduce this discretization error using finer partitions in the simulation than in the observation grid, at the price of increasing the computational load. We present here only the results based on the exact simulation. However, we obtained comparable results in the approximate simulation setting. We also obtained similar results for the correlations-based estimators in Remark \ref{rm:correlations}.

All figures are based on $M=10^5$ simulated trajectories of length $n=T_n=400$, for a given set of parameters. 

We first present the estimator based on low-frequency observations studied in Section \ref{Firstkind}. Theorem \ref{CLT_hat_rho_eta} and Theorem \ref{th:nogaus} establish the rate of convergence and the limit distributions of the rescaled estimation errors when $H<\frac{3}{2}$ and $H>\frac{3}{2}$, respectively. We reproduce these results numerically in Figures \ref{fig:clt_lf} and \ref{fig:nclt_lf}. \\
Figure \ref{fig:clt_lf} shows the estimation errors for $\hat\rho$ and $\hat\eta_{12}$ obtained  when $H=H_{1}+H_{2}=0.3$. On the left, the slope of the linear relationship of the logarithm of the root mean squared estimation error (RMSE) as a function of the logarithm of the length of the trajectory confirms the theoretical rate of convergence $\sqrt{n}$. On the right, we have the superposition of a centered Gaussian with the densities of the errors, rescaled with $\sqrt{n}$, for varying length of the trajectory.

\begin{figure}[ht]
\centering
\includegraphics[scale = 0.5]{"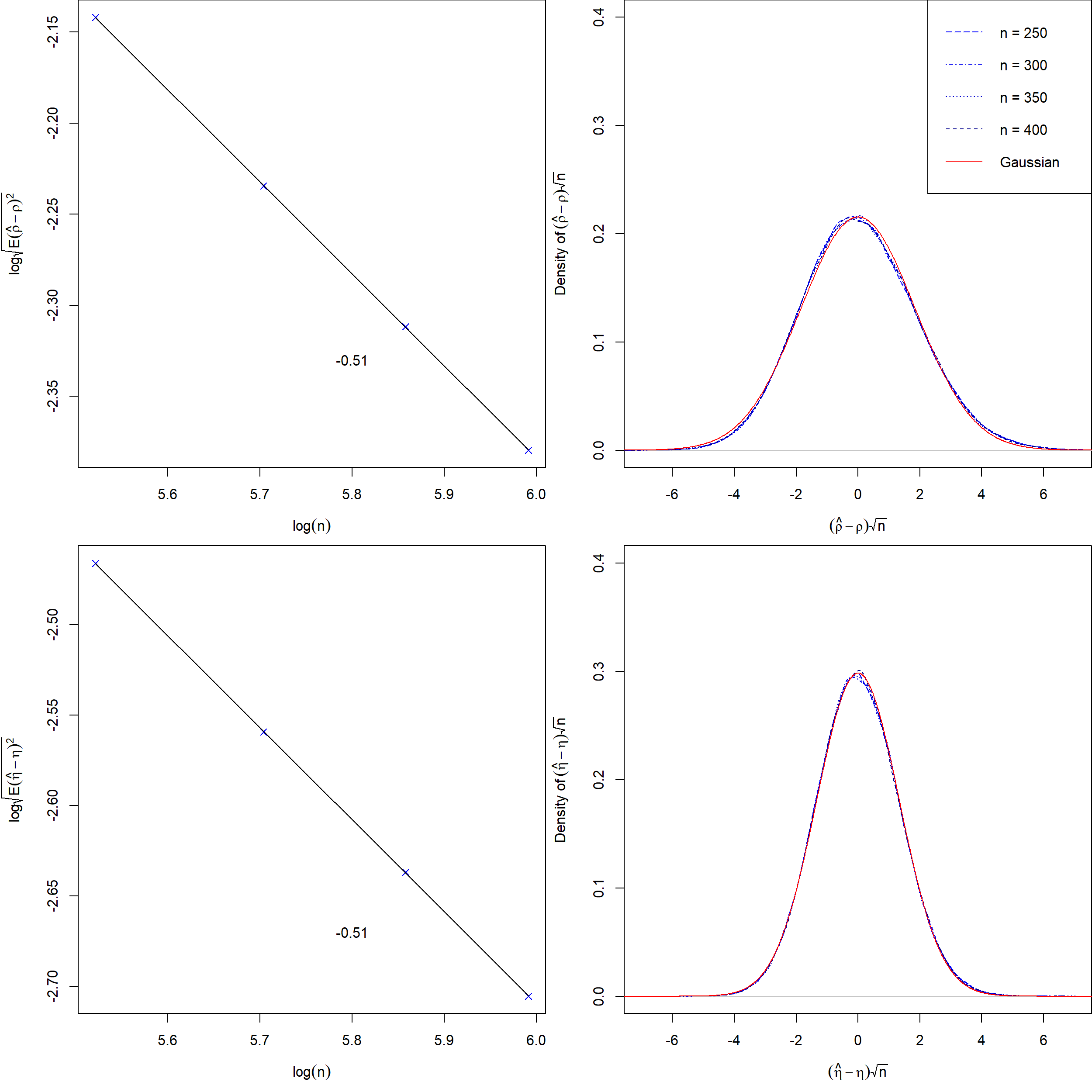"}
\captionof{figure}{Logarithm of the RMSE vs logarithm of the length of the trajectory (left) and superposition of the densities of the rescaled estimation errors of the low frequency estimator for varying length of the trajectory (right) for $\hat\rho$ (top) and $\hat\eta$ (bottom) - Simulation parameters: $\rho = 0.5,\ \eta = 0.2,\ H_1=0.1,\ H_2 = 0.2,\ \alpha_1=\alpha_2=0.5,\ \nu_1=\nu_2=1,\ n=T_n=400,\ M = 10^5$.
\label{fig:clt_lf}}
\end{figure}

A similar display in Figure \ref{fig:nclt_lf} for $H=H_{1}+H_{2}=1.7$ validates the results in Theorem \ref{th:nogaus}. Indeed, the estimated rate of convergence for $\hat\rho$ is now close to $n^{H-2}$ (it remains higher for $\hat\eta_{12}$) and the limit densities for both $\hat\rho$ and $\hat\eta$ are asymmetric. 
The densities overlap when rescaled by $n^{H-2}$ and appear skewed, confirming the convergence in law to a non Gaussian random variable. \\

\begin{figure}[ht]
\centering
\includegraphics[scale = 0.5]{"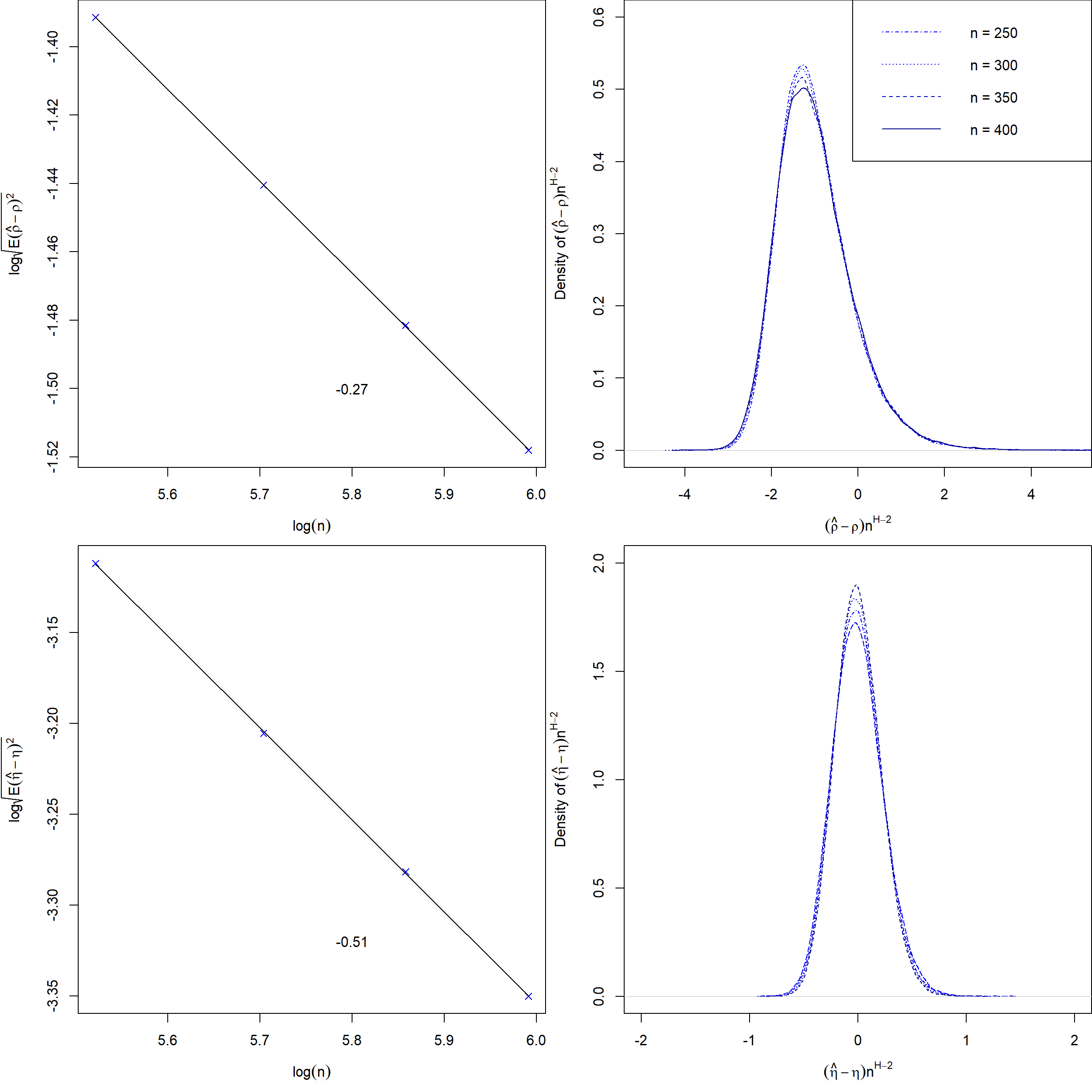"}
\captionof{figure}{Logarithm of the RMSE vs logarithm of the length of the trajectory (left) and superposition of the densities of the rescaled estimation errors of the low frequency estimators for varying length of the trajectory (right) for $\hat\rho$ (top) and $\hat\eta$ (bottom) - Simulation parameters: $\rho = 0.5,\ \eta = 0.2,\ H_1=0.8,\ H_2 = 0.9,\ \alpha_1=\alpha_2=0.5,\ \nu_1=\nu_2=1,\ n=T_n=400,\ M = 10^5$.
\label{fig:nclt_lf}}
\end{figure}

Figure \ref{fig:conv_rates_lf} shows the estimated rates of convergence of the estimation errors for varying values of $H$. The Monte Carlo estimates are obtained using the linear relationship between log-RMSE and the logarithm of the number of observations, as on the left-hand side of Figure \ref{fig:clt_lf} and Figure \ref{fig:nclt_lf}. The broken line reflects the prediction of Theorem \ref{CLT_hat_rho_eta} ($\sqrt{n}$ when $H\le\frac{3}{2}$) and Theorem \ref{th:nogaus} ($n^{2-H}$ when $H>\frac{3}{2}$), indicating that the convergence becomes very slow as $H\to 2$. The theoretical rate is matched closely by $\hat\rho$, while for  $H>\frac{3}{2}$,  in our finite sample experiment, $\hat\eta_{12}$ seems to converge faster than expected .  
\begin{figure}[ht]
\centering
\includegraphics[scale = 0.5]{"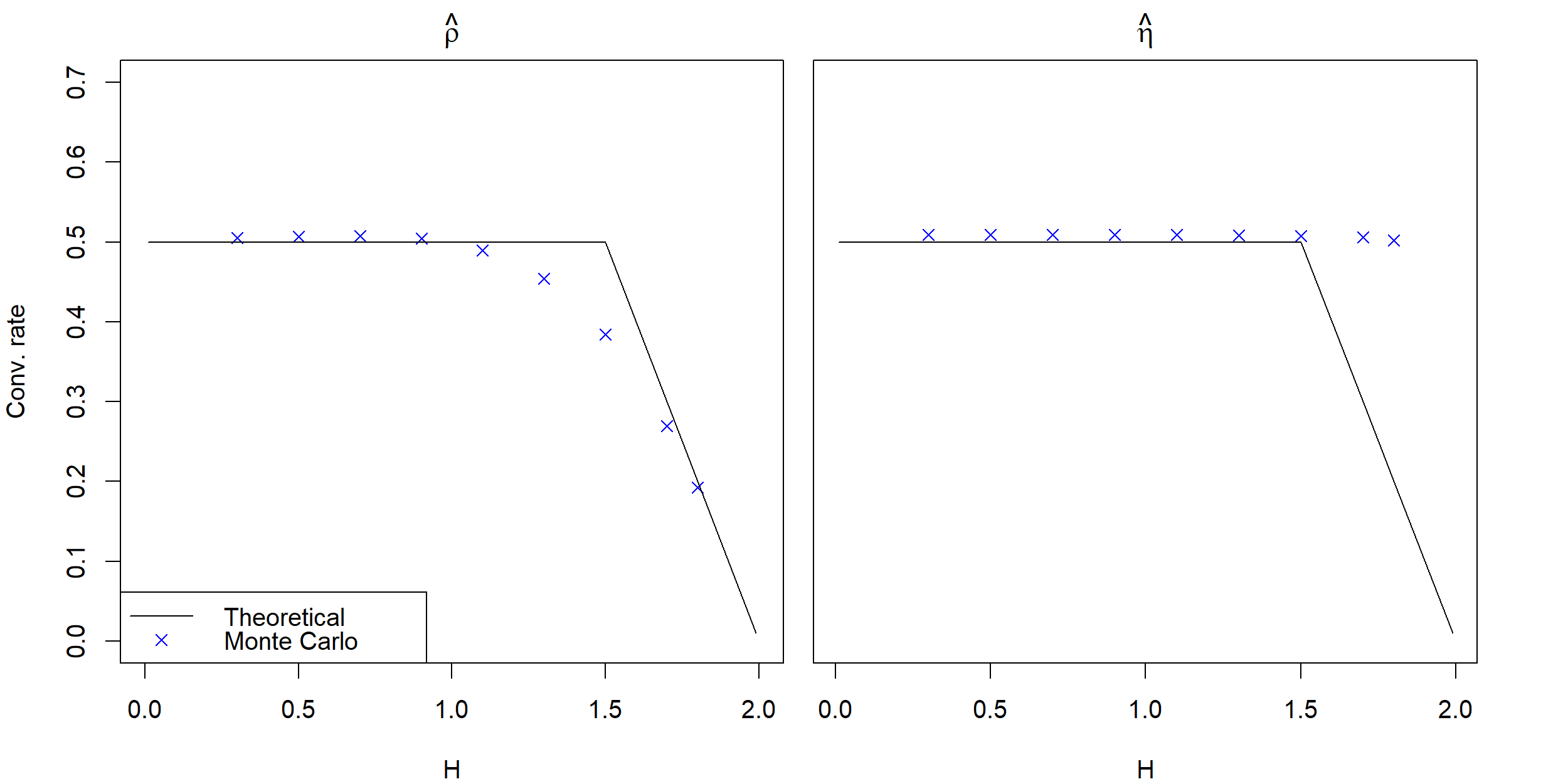"}
\captionof{figure}{Rates of convergence in the central and non-central limit theorems for $\hat\rho$ (left) and $\hat\eta$ (right) based on low-frequency observations - Simulation parameters: $\rho = 0.5,\ \eta = 0.2,\ H_2 = H_1+0.1,\ H = H_1+H_2,\ \alpha_1=\alpha_2=0.5,\ \nu_1=\nu_2=1,\ n=T_n=400,\ M = 10^5$.
\label{fig:conv_rates_lf}}
\end{figure}

\medskip

We now present the estimators  based on high-frequency observations studied in Section \ref{Secondkind}. In this setting, Theorem \ref{CLT_tilde_rho} only gives the speed of convergence and asymptotic normality for $\tilde\rho$ when $H\le\frac{3}{2}$. To approach the high frequency setting, instead of shrinking the time lag between observations we consider a small mean reversion parameter $\alpha_1=\alpha_2=0.1$ (this can be seen to be equivalent, cf. \cite{dgp.eco}), and we still take a fixed time lag. Figure \ref{fig:clt_hf} shows indeed that the central limit theorem holds for $\tilde\rho$ when $H=0.5$, analogously to the previous figures.  \\

\begin{figure}[ht]
\centering
\includegraphics[scale = 0.5]{"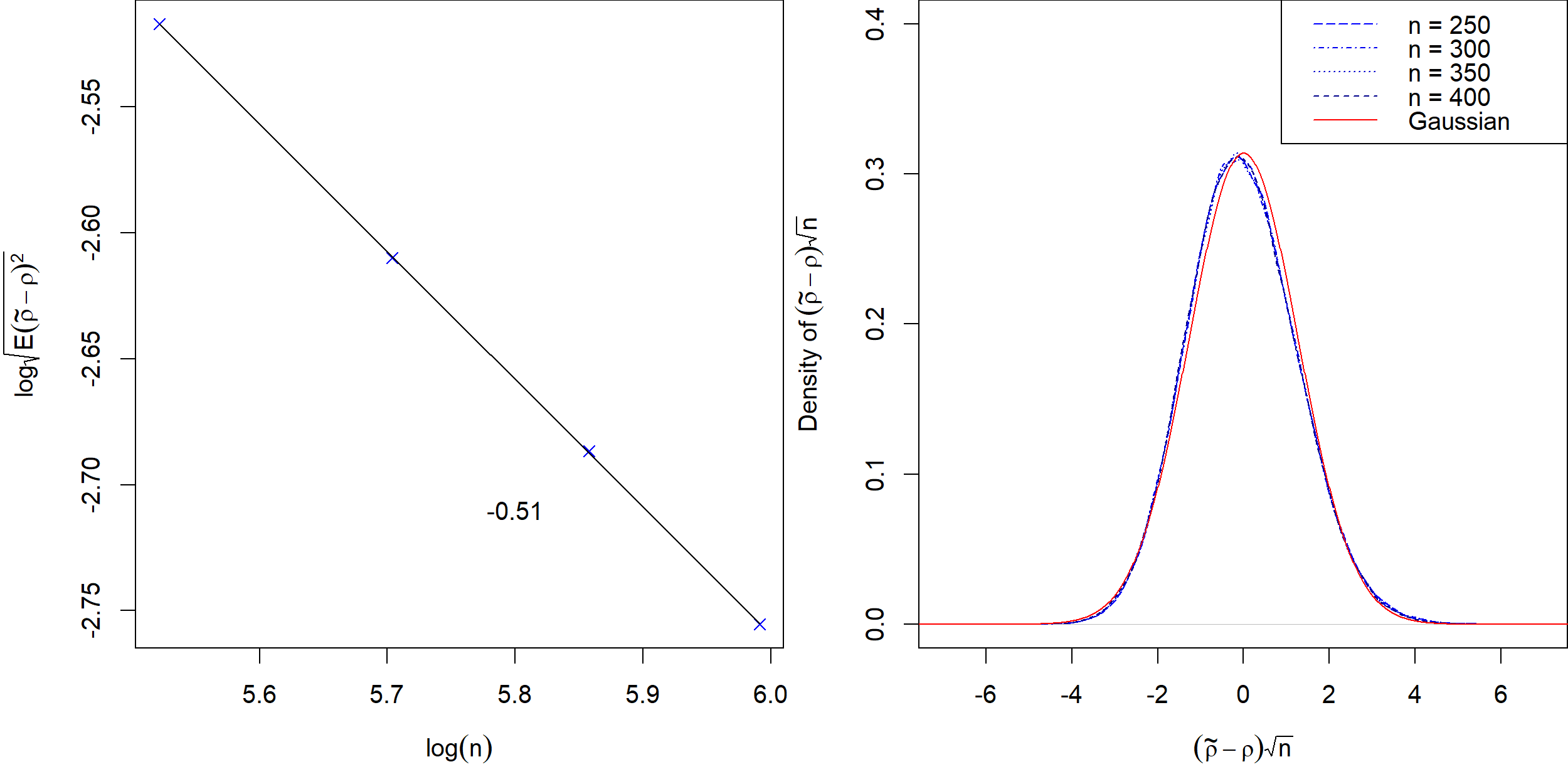"}
\captionof{figure}{Logarithm of the RMSE vs logarithm of the length of the trajectory (left) and superposition of the densities of the rescaled estimation errors for varying length of the trajectory (right) for $\tilde\rho$ - Simulation parameters: $\rho = 0.5,\ \eta = 0.2,\ H_1=0.2,\ H_2 = 0.3,\ \alpha_1=\alpha_2=0.1,\ \nu_1=\nu_2=1,\ n=T_n=400,\ M = 10^5$.
\label{fig:clt_hf}}
\end{figure}

Figure \ref{fig:conv_rates_hf} shows the convergence rates of the error for $\tilde\rho$ and $\tilde\eta$ when $H$ varies. Those of $\tilde\rho$ are close to $\sqrt{n}$ when $H\le\frac{3}{2}$ and seem close to $n^{H-2}$ when $H>\frac{3}{2}$. The latter is not covered by our asymptotic theory but it remains a reasonable conjecture by analogy with the low-frequency case $\hat\rho$. The speed of convergence for $\tilde\eta_{12}$ is not covered by our asymptotic theory. Numerically, the estimator seems to converge even when $H>1$, in which case even the consistency was not established. However, the speed of convergence seems lower than the one for $\tilde \rho$, for all values of $H$.  \\

\begin{figure}[ht]
\centering
\includegraphics[scale = 0.5]{"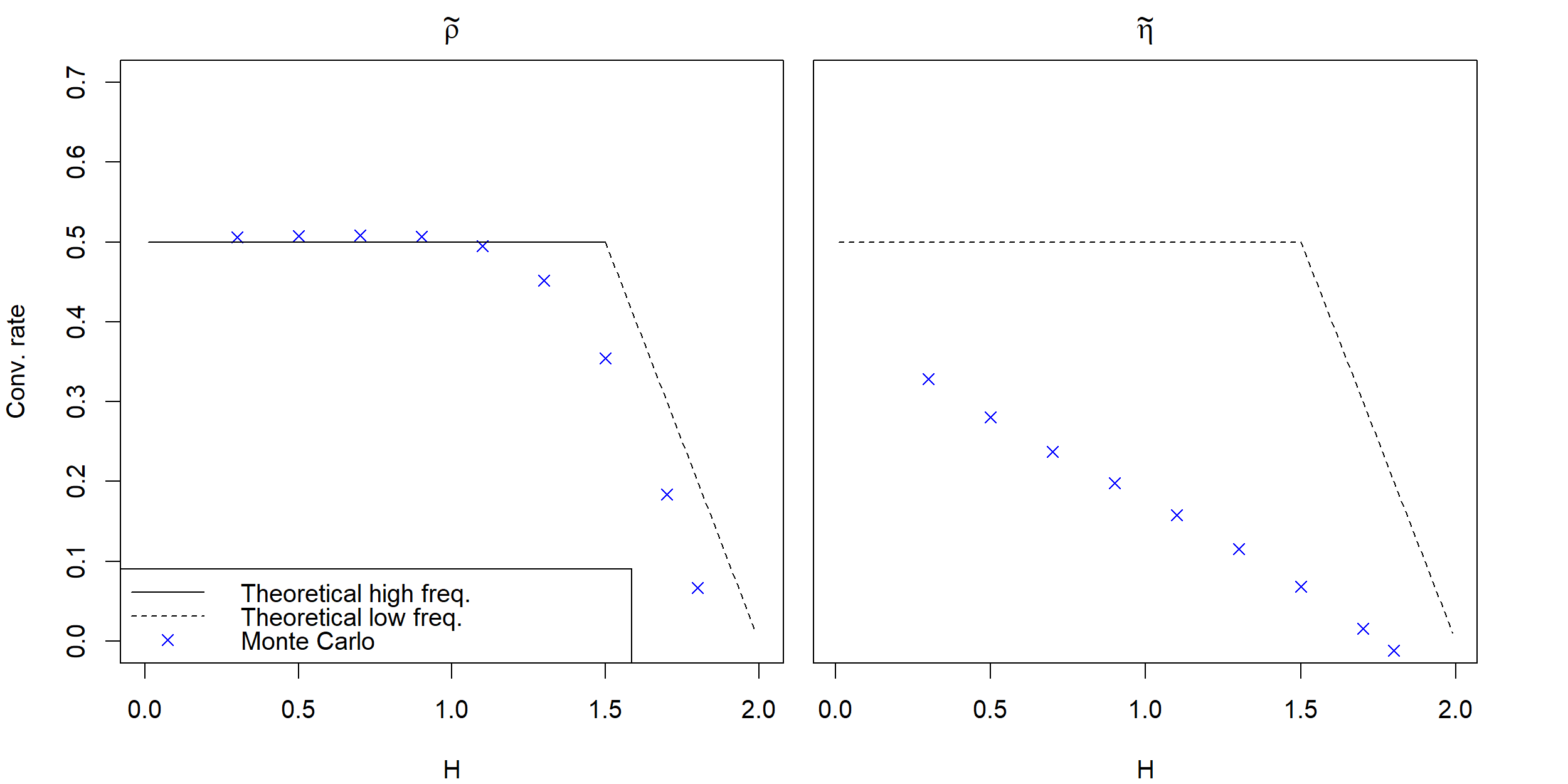"}
\captionof{figure}{Rates of convergence for $\tilde\rho$ (left) and $\tilde\eta$ (right) - Simulation parameters: $\rho = 0.5,\ \eta = 0.2,\ H_2 = H_1+0.1,\ H = H_1+H_2,\ \alpha_1=\alpha_2=0.1,\ \nu_1=\nu_2=1,\ n=T_n=400,\ M = 10^5$. The continuous line is the theoretical rate of convergence of this (high frequency) estimator, where we could compute it. We have reported as dotted line also the speed of convergence from Figure \ref{fig:conv_rates_lf} for the sake of completeness.
\label{fig:conv_rates_hf}}
\end{figure}

Due to the fact that the first two terms in the expansion of the cross-covariance in Lemma \ref{shorttime}, are the same terms that one obtains expanding for $(\alpha_1,\alpha_2)\to 0$, and in the limit for $(\alpha_1,\alpha_2)\to 0$ we have functional convergence of the mfOU to the mfBM (shown in \cite{dgp.eco}), we experiment with $\tilde\rho$ and $\tilde\eta$ in estimating the correlation parameters of the mfBM, i.e. the case $(\alpha_1,\alpha_2)=(0,0)$. Note that this process is not stationary.

The results that we obtained on simulations in this setting are favourable for small $H$. Indeed, as shown in Figure \ref{fig:conv_rates_mfbm}, the estimators $\tilde\rho$ and $\tilde\eta_{12}$ seem to work for $0 < H=H_{1}+H_{2} < 1$, which is half the range of the case for mfOU. In addition, the rate of convergence suggested by our numerical experiments seems to be $\sqrt{n}$ when $H<\frac{1}{2}$ and $n^{1-H}$ when $\frac{1}{2}\le H\le 1$. 

Figure \ref{fig:conv_rates_mfbm} suggests that our high-frequency estimator could be a good alternative to the estimator proposed in \cite{AC11} for the correlation parameters of the mfBM when $H$ is believed to be small, which for example is usually the case in log volatility time series \cite{GJR18}, since in the implementation in  \cite{AC11} the estimator for $\eta$ did not seem to clearly identify the sign of the parameter due to the high sensitivity to the choice of the dilation parameter in the filtering step. 

\begin{figure}[ht]
\centering
\includegraphics[scale = 0.5]{"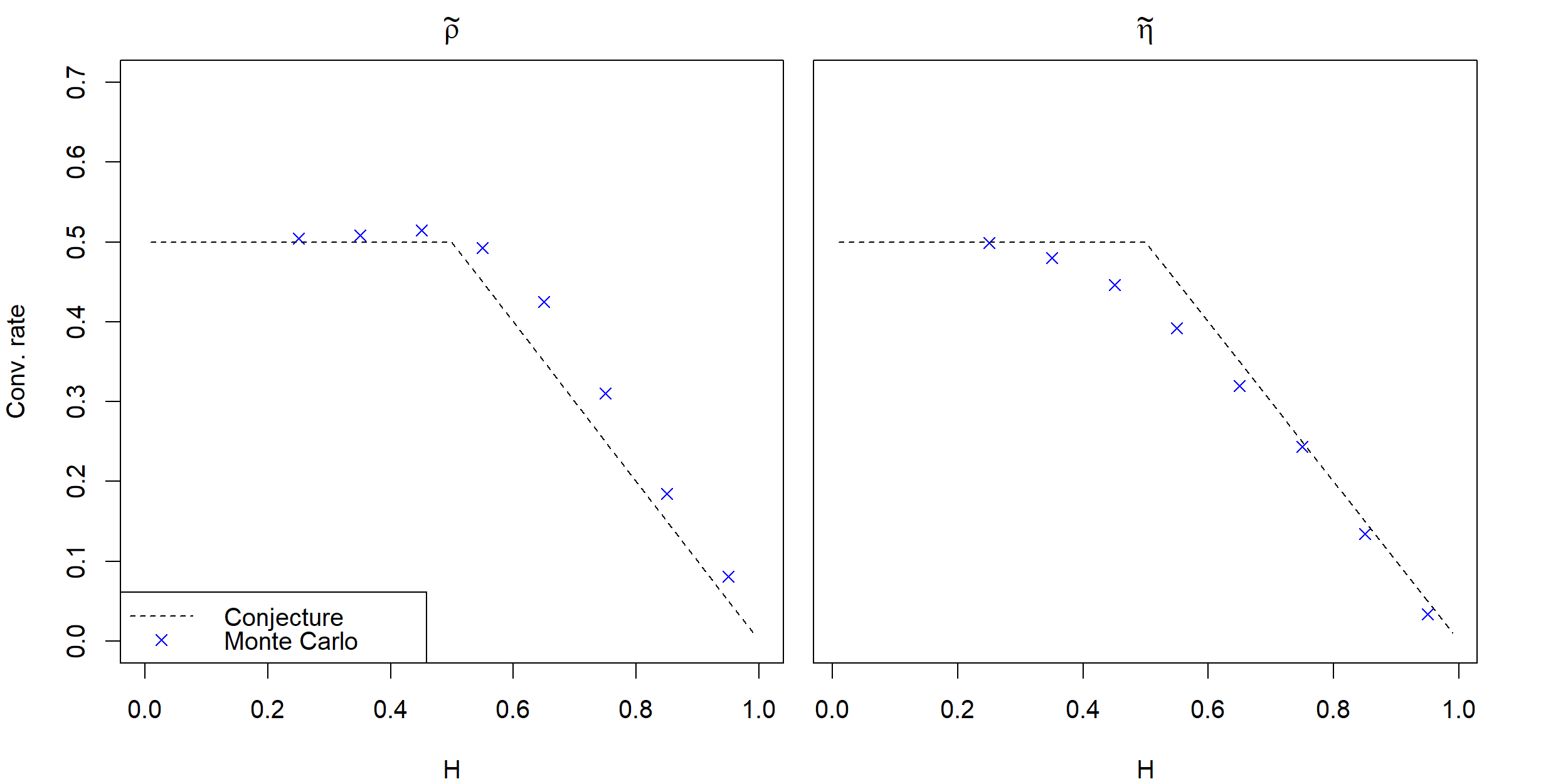"}
\captionof{figure}{Rates of convergence for $\tilde\rho$ and $\tilde\eta$ applied to the mfBM - Simulation parameters: $\rho = 0.5,\ \eta = 0.2,\ H_2 = H_1+0.05,\ H = H_1+H_2,\ \nu_1=\nu_2=1,\ n=T_n=400,\ M = 10^5$
\label{fig:conv_rates_mfbm}}
\end{figure}

For the sake of clarity, in all the plots in this section the rates of convergence are shown as a function of $H=H_{1}+H_{2}$, with $H_1-H_2$ fixed. However, we carried out similar experiments on a triangular grid for $H_1$ and $H_2$ and the results are consistent with those presented here. 

\section{Summary and outlook}

In this paper,  we have introduced a multivariate version of the fractional Ornstein Uhlenbeck (mfOU) process, extending the univariate fOU to multiple dimensions, building on the multivariate fractional Brownian motion (mfBm) framework from \cite{lav09,multivar}. This mfOU is a stationary, ergodic Gaussian process where each component has its own Hurst exponent, allowing for different degrees of roughness across dimensions. The cross-correlation between components is governed by two distinct parameters: the usual contemporaneous correlation $\rho$, and also a parameter $\eta$ that captures time-reversibility or asymmetry in the inter-component dynamics. We thoroughly characterized the process, analyzing its 
covariance and correlations structure and its short- and long-time asymptotics. 

We have then proposed two types of estimators for $\rho$ and $\eta$, constructed from discrete observations. Both assume known marginal fOU parameters and target the cross-correlation parameters. The first is obtained via inversion of the correlation function. Its fixed-lag, long-span asymptotic behavior (Gaussian or non-Gaussian) depends on the combined Hurst exponents. The second is obtained via inversion of the asymptotic correlation function, and it is a ``high-frequency'' estimator in the sense that we observe its asymptotic behavior in the vanishing-lag, long-span setting, when the Hurst exponents are in a certain range. The second estimator has the advantage of not depending on marginal mean-reversion parameters.

The analysis relies on Malliavin calculus and Stein's method to handle Gaussian functionals. We also performed a thorough Monte Carlo simulation study, confirming the asymptotic theory and supporting a conjecture on applying the high-frequency estimator to the pure mfBm case.

The mfOU process can be effectively used in Quantitative Finance to model a system of log-volatilities as shown in \cite{dgp.eco}. There, we estimate the model on multivariate realized volatility time series obtained from 5-minute price increments with a Generalized Method of Moments that identifies all the parameters at once by leveraging on a number of lags of the cross-covariance function. A related study was simultaneously carried out by \cite{bibinger}, where the authors employ a mfBm to forecast multivariate realized volatility time series. A possible extension would be modeling volatility itself with the exponential of the mfOU process within an asset pricing model. At that point, it might be possible to obtain a parametric covariance function of the asset prices, which can be useful in mainstream financial applications, e.g. solving the Markowitz problem in closed form.

\section{Proofs and technical results}\label{sec:proofs}
In this section we collect the proofs of our results. 

\begin{proof}[Proof of Theorem \ref{StatC}]
	
	Starting from the mfBm $(B^{H_1},\ldots, B^{H_d})$, let us define $B_t^{h,H_i}=B_{t+h}^{H_i}-B_h^{H_i}$, for $i=1,\ldots, d$, for all $h>0$. From the stationarity of the increments we have that
	$$
	\E[B_t^{h, H_i} B_{s}^{h, H_j}]=\E[(B_{t+h}^{H_i}-B_h^{H_i})(B_{s+h}^{H_j}-B_h^{H_j})]=\E[B_t^{H_i}B_s^{H_j}].
	$$
	Then, the vector $(B^{h, H_1},\ldots, B^{h, H_d})$ is a mfBm with the same covariance structure as $(B^{H_1},\ldots, B^{H_d})$. It follows that $dB_t^{h,H_i}=d B_{t+h}^{H_i}$ and so
	\begin{align*}
		\Cov(Y_{t+s}^i, Y_t^j)&=\E\Big[ \nu_1\int_{-\infty}^{t+s} e^{-\alpha_i(t+s-u)} dB_u^{H_i}\nu_2 \int_{-\infty}^{t} e^{-\alpha_j(t-v)} dB_v^{H_j}\Big]\\
		&=\E\Big[ \nu_1\int_{-\infty}^{t+s} e^{-\alpha_i(t+s-u)} dB_{u-t}^{t, H_i}\nu_2 \int_{-\infty}^{t} e^{-\alpha_j(t-v)} dB_{v-t}^{t, H_j}\Big]\\
		&=\E\Big[\nu_1\int_{-\infty}^s e^{-\alpha_i (s-u)} dB_u^{H_i} \int_{-\infty}^0 e^{\alpha_j v} dB_v^{H_j} \Big]\\
		&=\Cov(Y_s^i, Y_0^j).
	\end{align*}
	Then the mfOU is stationary. 
	Now, we can extend Lemma \ref{repLemma21} from the univariate fBm to the mfBm. Let $-\infty\leq a <b \leq c < d<+\infty$. Then, for $i,j\in \{1,\ldots, d\}, i\neq j$, 
	\begin{itemize}
		\item[1)] if $H_{ij}=H_i+H_j\neq 1$ we have
		\begin{align}\label{2}
			\E\Big[\int_a^b  e^{\a_j u}dB_u^{H_j}   &\int_c^d  e^{\a_i v}dB_v^{H_i} \Big]=H_{ij}(H_{ij}-1)\frac {\rho_{ij}+\eta_{ij}}{2}\int_c^d e^{\a_i v} \Big( \int_a^b  e^{\a_j u}(v-u)^{H_{ij}-2} du \Big) dv;
		\end{align}
		
		\item[2)]
		if $H_{ij}=H_i+H_j=1$ we have
		\begin{align}
			\E\Big[\int_a^b  e^{\a_j u}dB_u^{H_j}   &\int_c^d  e^{\a_i v}dB_v^{H_i} \Big] =-\frac{\eta_{ij}}{2}\int_c^d e^{\a_i v} \Big( \int_a^b  e^{\a_j u}(v-u)^{-1} du\Big) dv.
		\end{align}
		
	\end{itemize} 
	The above formulas follow from Lemma \ref{repLemma21}. 
	It follows that, for $t>s$, 
	\begin{align*}
	&	r_{ij}(t-s)=\\ &=\Cov(Y^i_{t}, Y^j_{s})=\Cov(Y^i_{t-s}, Y^j_0)=\nu_i\nu_j \E\Big[\int_{-\infty}^0 e^{\a_j u} dB^{H_j}_u \int_{-\infty}^{t-s} e^{-\a_i(t-s-v)} dB_v^{H_i}\Big]\\
		&=\nu_i\nu_j e^{-\a_i (t-s)}\Big(\E\Big[\int_{-\infty}^0  e^{\a_j u} dB^{H_j}_u \int_{-\infty}^0 e^{\a_i v} dB_v^{H_i}\Big]+\E\Big[\int_{-\infty}^0 e^{\a_j u} dB^{H_j}_u \int_{0}^{t-s} e^{\a_i v} dB_v^{H_i}\Big]\Big)\\
		&= e^{-\a_i (t-s)} \Cov(Y^i_0,Y^j_0) +\nu_i\nu_j e^{-\a_i (t-s)} H_{ij}(H_{ij}-1)\frac{\rho_{ij}+\eta_{ij}}{2} \int_0^{t-s} e^{\a_i v} \Big( \int_{-\infty}^0 e^{\a_j u}(v-u)^{H_{ij}-2} du\Big) dv.
	\end{align*}
	When $H_{ij}=1$, analogously 
	\begin{align*}
		r_{ij}(t-s)&=\Cov(Y^i_{t}, Y^j_{s})=\Cov(Y^i_{t-s}, Y^j_0)=\nu_i\nu_j \E\Big[\int_{-\infty}^0 e^{\a_j u} dB^{H_j}_u \int_{-\infty}^{t-s} e^{-\a_i(t-s-v)} dB_v^{H_i}\Big]\\
		&=\nu_i\nu_j e^{-\a_i(t- s)}\Cov(Y_0^i, Y_0^j)-\nu_i\nu_j e^{-\a_i (t-s)}\frac{\eta_{ij}}{2}\int_{-\infty}^{0} e^{\a_i v} \Big( \int_{0}^{t-s}  e^{\a_j u}(v-u)^{-1} du\Big) dv.
	\end{align*}
\end{proof}

\begin{proof}[Proof of Theorem \ref{cov}]
	We have that (see \cite{G24}[Lemma 5.4.4] for the details of the computation)
	\begin{align*}
		\Cov(Y_t^{i},Y_{t+s}^{j})&=\Cov(Y_0^{i},Y_{s}^{j})=\E\Big[\nu_i\int_{-\infty}^0 e^{\a_j u}dB_u^{H_i}\nu_j\int_{-\infty}^s e^{-\a_j(s-v)}dB_v^{H_j}\Big]\\
		&=\nu_i \nu_j e^{-\a_j s}\E\Big[ \int_{-\infty}^0 e^{\a_i u}dB_u^{H_i}\int_{-\infty}^{\frac{1}{\a_j}} e^{\a_j v}dB_v^{H_j}+\int_{-\infty}^0 e^{\a_i u}dB_u^{H_i}\int_{\frac{1}{\a_j}}^s e^{\a_j v}dB_v^{H_j}\Big]\\
		&=\nu_i\nu_j e^{-\a_j s}d_{H_{ij}}\int_{\frac{1}{\a_j}}^s e^{\a_j v} \int_{-\infty}^0  e^{\a_i u}(v-u)^{H_{ij}-2} du dv+O(e^{-\a_j s}).
	\end{align*}
	The constant $d_{H_{ij}}$ is equal to $H_{ij}(H_{ij}-1)\frac{\rho_{ij}-\eta_{ij}}{2}$ when $H_{ij}\neq 1$ and $d_{H_{ij}}=\frac{\eta_{ij}}{2}$ when $H_{ij}=1$. By employing the change of variables $y = v - u$ and $z = v + u$, we derive
	\begin{align*}
		\Cov(Y_t^{i},Y_{t+s}^{j})&=\frac{\nu_i \nu_j d_{H_{ij}}}{2} e^{-\a_j s}\Big( \int_{\frac{1}{\a_j}}^s y^{H_{ij}-2} e^{\frac{\a_j-\a_i}{2}y}  \int_{\frac{2}{\a_j}-y}^{y} e^{\frac{\a_j+\a_i}{2}z} dzdy+\\
		&+\int_s^{+\infty} y^{H_{ij}-2} e^{\frac{\a_j-\a_i}{2}y} \int_{\frac{2}{\a_j}-y}^{2s-y} e^{\frac{\a_j+\a_i}{2}z} dzdy \Big)+O(e^{-\a_j s})\\
		&=\frac{\nu_i \nu_j d_{H_{ij}}}{(\a_j+\a_i)}e^{-\a_j s}\Big( \int_{\frac{1}{\a_j}}^s y^{H_{ij}-2}e^{\a_j y}dy +e^{(\a_j+\a_i)s}\int_s^{+\infty} y^{H_{ij}-2}e^{-\a_i y}dy\\
		& -e^{\frac{\a_j+\a_i}{\a_j}}\int_{\frac{1}{\a_j}}^{+\infty} y^{H_{ij}-2} e^{-\a_i y}dy \Big)+O(e^{-\a_j s})\\
		&=\frac{\nu_i \nu_j d_{H_{ij}}}{(\a_j+\a_i)}\Big( \frac{1}{\a_j^{H_{ij}-1}}e^{-\a_j s}\int_1^{\a_j s} y^{H_{ij}-2}e^y dy+\frac{1}{\a_i^{H_{ij}-1}} e^{\a_i s}\int_{\a_i s}^{+\infty} y^{H_{ij}-2}e^{-y}dy\Big)+O(e^{-\a_j s})
	\end{align*}
	and, by Lemma 2.2 in \cite{cheridito2003}, we have
	\begin{align*}
		\Cov(Y_t^{i},Y_{t+s}^{j})&=\frac{\nu_i \nu_j d_{H_{ij}}}{(\a_j+\a_i)} \Big( \frac{1}{\a_j}s^{H_{ij}-2}+\sum_{n=1}^N \frac{(-1)^n}{\a_j^{n+1}}  \big[ \prod_{k=0}^{n-1} (H_{ij}-2-k) \big]s^{H_{ij}-2-n}\\
		&+\frac{1}{\a_i}s^{H_{ij}-2}+\sum_{n=1}^N \frac{1}{\a_i^{n+1}}  \big( \prod_{k=0}^{n-1} (H_{ij}-2-k) \big)s^{H_{ij}-2-n}\Big)+O(s^{H_{ij}-N-3})\\
		&=\frac{\nu_i\nu_j(\rho_{ij}-\eta_{ij})}{(\a_i+\a_j)}\sum_{n=0}^N \Big(\frac{(-1)^n}{\a_j^{n+1}}+\frac{1}{\a_i^{n+1}}\Big)  \big( \prod_{k=0}^{n+1} (H_{ij}-k) \big)s^{H_{ij}-2-n}+O(s^{H_{ij}-N-3}).
	\end{align*}
	Exchanging $Y^i$ and $Y^j$ we obtain the general form of the cross-covariance in \eqref{decCr}. For $H_{ij}=1$, analogous computations lead to \eqref{decCr1}.
\end{proof}

\begin{proof}[Proof of Lemma \ref{cross-ergod}]
	For $\tau>0$ we define 
	$
	\hat r_{ij}(\tau):=\frac 1{2T} \int_{-T}^T (Y_{t+\tau}^i-\E[Y_{t+\tau}^i])(Y_{t}^j-\E[Y_t^j]) dt
	$ 
	and we prove that $\hat r_{ij}^T(\tau)\underset{T\to+\infty}\to \E[Y_{\tau}^i Y_0^j]$ in probability. Clearly $\E[\hat r_{ij}^T(\tau)]=\E[Y^i_{\tau}Y^j_0]$. We study the variance.
	\begin{align*}
		\Var(\hat r_{ij}^T)&=\frac 1{4T^2}\int_{-T}^T \int_{-T}^T\Big( \Cov(Y_t^i, Y_s^i)\Cov(Y_t^j, Y_s^j)+\Cov(Y_{t+u}^i, Y_s^j)\Cov(Y_t^j, Y_{s+u}^i)	\Big)dt ds\\
		&=\frac 1T \int_{-2T}^{2T} \Big(1-\frac{|t|}{2T}\Big)\Big(\Cov(Y_t^i, Y_0^i)\Cov(Y_t^j, Y_0^j)+\Cov(Y_{t+\tau}^i,Y_0^j)\Cov(Y_{t-\tau}^i, Y_0^j)\Big)dt. 
	\end{align*}
	By Theorem \ref{Decay1} and Theorem \ref{cov}, when $H_{ij}=H_i+H_j<\frac 32$ the integrand is in $L^1(\R)$, then $\Var(\hat r_{ij}^T)\to 0$ as $T\to +\infty$. If $H_{ij}=\frac 32$, the integrand is $O(1/T)$ as $T\to \infty$, then the integral is $O(\log T)$, whereas for $H_{ij}>\frac 32$ the integral is $O(T^{2H_{ij}-3})$, then $\Var(\hat r_{ij}^T)=O(T^{2H_{ij}-4})$. In each case $\Var(\hat r_{ij}^T)\to 0$ as $T\to +\infty$. 
\end{proof}

\noindent\begin{proof}[Proof of Lemma \ref{shorttime}]
	Since $Y_{t+s}^j=e^{-\a s} Y_t^j-\a_j\nu_j e^{-\a_js} \int_{t}^{t+s} e^{\a_j u} B_u^{H_j} du$, we have
	\[
	\begin{split}
	&\Cov(Y^i_t, Y^j_{t+s})\\&=e^{-\a_j s} \Cov(Y^i_0,Y^j_0) -\nu_i\nu_j e^{-\a_j s} H_{ij}(1-H_{ij})\Big(\frac{\rho_{ij}-\eta_{ij}}{2}\Big) \int_0^s e^{\a_j v}  \int_{-\infty}^0 e^{\a_i u}(v-u)^{H_{ij}-2} du dv
	\end{split}
	\]
	for $t,s \in \R, s>0$. The first step of the proof is to develop the integral in the above equation. Let $i,j\in \{1,\ldots, d\}$, $i\neq j$, $H_{ij}\in(0,2)$. Using the suitable change of variables in the integral, and developing the integrand using Taylor's formula, we have that, for $H_{ij}\neq 1$ 
	
	\begin{align*}
		&\int_0^s dv\int_{-\infty}^0 e^{\a_i u +\a_j v}(v-u)^{H_{ij}-2} du\\
		&=\frac{s^{H_{ij}}}{H_{ij}(1-H_{ij})} - \frac{\a_i^{1-H_{ij}}\Gamma(H_{ij})}{1-H_{ij}}s+\frac{(H_{ij}\a_j+\a_i)}{H_{ij}(1-H_{ij})(H_{ij}+1)}s^{H_{ij}+1}\\
		&-\frac{(\a_i+\a_j)\a_i^{1-H_{ij}}}{2(1-H_{ij})}\Gamma(H_{ij}) s^2+\Big(\frac{\a_j^2}{6(1-H_{ij})} -\frac{(H_{ij}-3)\a_i\a_j}{6H_{ij}(1-H_{ij})}\\
		&+\Big(\frac{\alpha_i^2-\alpha_i\alpha_j+\alpha_j^2}{2(1-H_{ij})(2+H_{ij})}+\frac{\alpha_i\alpha_j}{H_{ij}(1-H_{ij})}+\frac{\alpha_i^2}{2H_{ij}(1+H_{ij})}\Big)s^{H_{ij}+2}+o(s^{H_{ij}+2}).
	\end{align*}
	while for $H_{ij}=1$
	$$
	\int_0^s  dv\int_{-\infty}^0 e^{\a_i u+\a_j v} (v-u)^{-1} du=-s\log s-\frac{\a_j-\a_i}{2}s^2\log s +o(s^2\log s).
	$$

	Denoting $K_{ij}= \frac{\rho_{ij}-\eta_{ij} }{2}$, we obtain
	\begin{align*}
		&\Cov(Y_t^i, Y_{t+s}^j)=\Cov(Y_0^i, Y_0^j)-K_{ij}\nu_i \nu_j s^{H_{ij}}+\Big(-\a_j \Cov(Y_0^i, Y_0^j)-\a_i^{1-H_{ij}} \Gamma(H_{ij}+1) \nu_i\nu_j K_{ij}\Big) s+ \\
		&-\frac{(\a_j-\a_i)\nu_i \nu_j }{H_{ij}+1} K_{ij} s^{1+H_{ij}}+\Big(\frac{\a_j^2}{2} \Cov(Y_0^i, Y_0^j)-\frac 12\nu_i\nu_j K_{ij} \Gamma(H_{ij}+1)(\a_j\a_i^{1-H_{ij}} -\a_i^{2-H_{ij}})\Big)s^2 -\\
		&+o(s^{\max{H_{ij}, H_{ij}+1}}) .
	\end{align*}

	When $H=H_i+H_j=1$, we have
	\begin{align*}
		&\Cov(Y_t^i, Y_{t+s}^j)-\Cov(Y_0^i, Y_0^j)\\
		&=-\a_j s \Cov(Y_0^i , Y_0^j)+o(s)+\nu_i\nu_j e^{-\a_j s} \frac{\eta_{ij}}{2}\int_0^s e^{\a_j v} \int_{-\infty}^0 e^{\a_i u} \frac 1{v-u} du dv +o(s^3)\\
		&=-\a_j s \Cov(Y_0^i , Y_0^j)+o(s)-\nu_i\nu_j \frac{\eta_{ij}}{2} s\log s +o(s^2\log s)\\
		&=-\nu_i\nu_j \frac{\eta_{ij}}{2} s\log s +o(s^2\log s).
	\end{align*}
\end{proof}

From now on, we prove the results in Section \ref{sec:estimation}. Therefore, as discussed, we can assume to be in dimension $d=2$. We also denote $H=H_{12}=H_1+H_2$ and $\rho_{12}=\rho_{21}=\rho$. 
Let us define
$$
\overline S_n=\frac{a_1}{n}\sum_{j=1}^n Y_{j}^1 Y_j^2+\frac{a_2}{n}\sum_{j=1}^{n} Y_{j+s}^1 Y_j^2+\frac{a_3}{n}\sum_{j=1}^n Y_{j}^1 Y_{j+s}^2
$$
and
\begin{align*}
	R_n=\overline S_n-S_n = \frac{a_2}{n}\sum_{j=n-s+1}^n Y_{j+s}^1 Y_j^2 +\frac{a_3}{n}\sum_{j=n-s+1}^n Y_j^1 Y_{j+s}^2. 
\end{align*}
We notice that, when $H<\frac 32$ 
\begin{align*}
 &\sqrt{n}\E[\overline S_n-S_n] \to 0\\
 &n\Var(\overline S_n- S_n) \to 0
\end{align*}
then $\sqrt{n}(\overline S_n-S_n){\to} 0$ in probability. When  $H=\frac 32$ then $\sqrt{{n}/{\log n}}(\overline S_n-S_n){\to} 0$ in probability. and when $H>\frac 32$ then $n^{1-H}(\overline S_n-S_n){\to} 0$ in probability.. Then we can prove the following results for $\overline S_n$ and they will hold for $S_n$ as well.  
\begin{proof}[Proof of Lemma 3.2]
		Computing the expectation of $\hat \rho_{n}$, we have
		\begin{align*}
			\E[\hat \rho_n]&= \frac{a_1(s)}{n}\sum_{j=1}^n \E[Y_j^1 Y_j^2]+\frac{a_2(s)}{n}\sum_{j=1}^{n-s} \E[Y_{j+s}^1 Y_j^2]+\frac{a_3(s)}{n}\sum_{j=1}^{n-s} \E[Y_{j}^1 Y_{j+s}^2]\\
			&=\frac{a_1(s)}{n}\sum_{j=1}^n \Cov(Y_0^1, Y_0^2)+\frac{a_2(s)}{n}\sum_{j=1}^{n-s} \Cov(Y_s^1,Y_0^2)+\frac{a_3(s)}{n}\sum_{j=1}^{n-s} \Cov(Y_0^1,Y_s^2)\\
			&=a_1(s)\, \Cov(Y_0^1, Y_0^2)+a_2(s)\frac{n-s}{n}\, \Cov(Y_s^1,Y_0^2)+a_3(s)\frac{n-s}{n}\, \Cov(Y_0^1,Y_s^2)\\
			&\to \rho.
		\end{align*}
		Similarly, it holds that $\E[\hat \eta_{12,n}]\to\eta_{12}$.
		
\end{proof}
From now on we denote $r_{11}(k)=\E[Y_k^1 Y_0^1]$, $r_{22}(k)=\E[Y_k^2Y_0^2]$, $r_{12}(|k|)=\E[Y_{|k|}^1 Y_0^2]$ and $r_{21}(|k|)=\E[Y_0^1 Y_{|k|}^2]$. 
\begin{proof}[Proof of Theorem \ref{VAR_SN}]

	The variance can be written as \begin{align*}
			&\Var(\sqrt n S_n)\\
			&=\frac { a_1^2}{n}\sum_{k,h=1}^n \Big( r_{11}(|k-h|) r_{22}(|k-h|)
			+r_{12}(|k-h|)r_{21}(|k-h|)\Big)+\\
			&+\frac { a_2^2}{n}\sum_{k,h=1}^n \Big( r_{11}(|k-h|) r_{22}(|k-h|)+
			r_{12}(|k+s-h|)r_{21}(|k-s-h|)\Big)\\
			&\cdots\\
			&+\frac { a_2 a_3}{n}\sum_{k,h=1}^n \Big( r_{11}(|k+s-h|) r_{22}(|k-s-h|)
			+r_{12}(|k-h|)r_{21}(|k-h|)\Big).
		\end{align*}
		We omit to write all the sums. Then, the variance of $S_n$ is a sum of sequences of the form 
		$$
		\frac 1{n}\sum_{k,h=1}^n r_{11}(|k-h|)r_{22}(|k-h|) \quad{\text{or }\,\,\,\,\, \frac 1{n}\sum_{k,h=1}^n r_{12}(|k-h|)r_{21}(|k-h|)}.
		$$
		We can have that the variables in functions $r_{ij}$ are shifted with the constant factor $s$, but the asymptotic behaviour does not change, so we can reduce the analysis to the above sequences. Then 
		\begin{align*}
			&\frac{1}{n} \sum_{k,h=1}^n r_{11}(|k-h|)r_{22}(|k-h|)= \sum_{|\tau|\leq n} \Big(1-\frac{|\tau|}n\Big)r_{11}(\tau)r_{22}(\tau)+\frac{1}{n}\\
			&=\sum_{\tau=1}^{\infty} \Big(1-\frac{|\tau|}n\Big)r_{11}(\tau)r_{22}(\tau)\1_{|\tau|\leq n}+\frac1n=\frac 2n \sum_{\tau=1}^{\infty} \Big(1-\frac{\tau}n\Big)r_{11}(\tau)r_{22}(\tau)\1_{1\leq\tau\leq n}+\frac1n 
		\end{align*}
		By Theorem \ref{Decay1} we have that, when $\tau\to \infty$, $r_{ii}(\tau)=O(\tau^{2H_i-2})$, then $r_{11}(\tau)r_{22}(\tau)=O(\tau^{2H-4})$, and it is summable when $H<\frac 32$. Then, using dominated convergence, 
		$$
		\lim_{n \to 
			+\infty}2\sum_{\tau=1}^{\infty} \Big(1-\frac{\tau}n\Big)r_{11}(\tau)r_{22}(\tau)\1_{\tau<n}+ r_{11}(0)r_{22}(0)=2\sum_{\tau=1}^{\infty} r_{11}(\tau)r_{22}(\tau)+r_{11}(0)r_{22}(0. 
		$$
		The same argument can be used for the second sequence, recalling that $r_{12}(\tau)=O(\tau^{H-2})$, $r_{21}(\tau)=O(\tau^{H-2})$, then $r_{12}(\tau) r_{21}(\tau)=O(\tau^{2H-4})$. 
		It follows that, when $H<3/2$, $\Var(S_n)=O(1/n)$. 
		Moreover, we easily see that 
		\begin{align*}
			&\lim_{n\to +\infty}	\Var(\sqrt{n} S_n)=\Var(a_1Y_0^1Y_0^2+a_2Y_{s}^1Y_0^2+ a_3 Y_0^1 Y_{s}^2)\\
			&+2\sum_{k=1}^{+\infty} \Cov(a_1Y_0^1Y_0^2+a_2Y_{s}^1Y_0^2+ a_3 Y_0^1 Y_{s}^2,a_1Y_k^1Y_k^2+a_2Y_{k+s}^1Y_k^2+ a_3 Y_k^1 Y_{k+s}^2).\\
		\end{align*}
		
		\noindent When $H= 3/2$, there exists a constant $C>0$ and an integer $N>0$ such that
		\begin{align*}
			&\frac{1}{n} \sum_{|\tau|=0}^n r_{11}(\tau)r_{22}(\tau)=O\Big(\frac1{n}\Big)+\frac{C}{n} \sum_{|\tau|\geq N}^n \frac{1}{\tau}+ \frac{C_2}{n}\sum_{|\tau|\geq N}^n O\Big(\frac1{\tau^{3}}\Big)=O\Big(\frac{\log n}n \Big).
		\end{align*}
		The same holds for $\frac{1}{n} \sum_{|\tau|=0}^n r_{12}(\tau)r_{21}(\tau)$.
		When $H>3/2$, we have
		\begin{align*}
			&\frac{1}{n} \sum_{|\tau|=0}^n r_{11}(\tau)r_{22}(\tau)=O\Big(\frac1{n}\Big)+\frac{C}{n} \sum_{|\tau|\geq N}^n \frac{1}{\tau^{4-2H}}+ \frac{C_2}{n}\sum_{|\tau|\geq N}^n O\Big(\frac1{\tau^{6-2H}}\Big)=O\Big(\frac1{n^{4-2H}} \Big).
		\end{align*}
		Again, the same holds for $\frac{1}{n} \sum_{|\tau|=0}^n r_{12}(\tau)r_{21}(\tau)$.

\end{proof}
\begin{proof}[Proof of Theorem \ref{consistency}]
	Since $\hat \rho_n$ and $\hat\eta_{12,n}$ are asymptotically unbiased and, by Theorem \ref{VAR_SN}, the sequences of their variances tend to $0$, then $\hat \rho_n$ and $\hat\eta_{12,n}$ converge in $L^2(\P)$ and then in probability to $\rho$ and $\eta_{12}$, respectively.
\end{proof}
To prove Theorem \ref{S_n_conv} we use the Malliavin-Stein's method (see \cite{NP12}) and the fourth moment theorem (Theorem \ref{FOURTH}). From now on, for the details of certain computations we refer the reader to Chapter 5, Sections 6 and 7 in \cite{G24}. Theorem \ref{MovAverageFBM} implies that we can write $Y_k^i=\int_{\R} \<  f_k^i(s), W(ds)\>_{\R^2}=I_1( f_k^i)$, where $ f_{k}^i\in L^2(\R; \R^2)$, $i=1,2$, $k\in \N$ and $W$ is a bidimensional Gaussian noise. Here $I_1$ denotes the Wiener-It\^o integral of order $1$ with respect to $W$. Then, by the product formula in \eqref{product_formula}, we have
\begin{align*}
	\overline S_n&=\frac 1n \sum_{k=1}^n\Big(  a_1 Y_k^1Y_k^2+a_2 Y_{k+s}^1 Y_k^2+a_3 Y_k^1 Y_{k+s}^2 -\E[ a_1 Y_k^1Y_k^2+a_2 Y_{k+s}^1 Y_k^2+a_3 Y_k^1 Y_{k+s}^2]  \Big)\\
	&=\frac 1{n} \sum_{k=1}^n\Big(  a_1 I_1(f_k^1)I_1(f_k^2)+a_2 I_1(f_{k+s}^1) I_1(f_k^2)+a_3 I_1(f_k^1)I_1(f_{k+s}^2)\\
	&-\E[a_1 I_1(f_k^1)I_1(f_k^2)+a_2 I_1(f_{k+s}^1) I_1(f_k^2)+a_3 I_1(f_k^1)I_1(f_{k+s}^2)]\Big)\\
	&=\frac 1{n} \sum_{k=1}^n I_2(a_1 f_k^1\tilde \otimes f_k^2+a_2f_{k+s}^1\tilde \otimes f_k^2+a_3f_k^1\tilde \otimes f_{k+s}^2).
\end{align*} 
Theorem \ref{FOURTH} requires that $\Var(\sqrt{n}\overline S_n)\to \sigma^2$ where $\sigma^2$ is a strictly positive and finite constant (we have proved that in Theorem \ref{VAR_SN}) and $\kappa_4(\sqrt{n}\overline S_n)\to 0$, where $\kappa_4$ denotes the cumulant of order $4$, or equivalently $\|\theta_n\otimes_1\theta_n\|_{L^2(\R^2;\R^2)}\to 0$, where $\theta_n=n^{-\frac 12} \sum_{k=1}^n \big(a_1 f_k^1\tilde \otimes f_k^2+a_2f_{k+s}^1\tilde \otimes f_k^2+a_3f_k^1\tilde \otimes f_{k+s}^2\big)
$ is the kernel of $S_n$ as a double Wiener-It\^o integral and $\theta_n \otimes_1 \theta_n$ is the contraction of order $1$ of $\theta_n$ with itself (see \eqref{contraction}). By linearity, $\theta_n \otimes_1 \theta_n =n^{-1}\sum_{k,h=1}^n z_k^s \otimes_1 z_h^s$ where $z_k^s= f_k^1\tilde \otimes f_k^2+a_2f_{k+s}^1\tilde \otimes f_k^2+a_3f_k^1\tilde \otimes f_{k+s}^2$, and so we have to compute contractions of the form $(f_{k}^1\tilde\otimes f_{k}^2)  \otimes_1 (f_{h}^1\tilde\otimes f_{h}^2)$ or analogous expressions, where we just change the functions $f_k^i$ with $f_{k+s}^i$ in a suitable manner. Then, denoting by $(e_r)_{r\in\N}$ an orthonormal basis of $L^2(\R;\R^2)$, we have
\begin{align*}
 (f_{k}^1\tilde\otimes f_{k}^2)  \otimes_1 (f_{h}^1\tilde\otimes f_{h}^2)& =\sum_{r_1, r_2, r_3=1}^{\infty} 
	\<f_{k}^1\tilde\otimes f_{k}^2, e_{r_1} \otimes e_{r_2}\> \<f_{h}^1\tilde\otimes f_{h}^2, e_{r_1} \otimes e_{r_3}\> e_{r_2}\otimes e_{r_3}\\
	&=\sum_{r_2, r_3=1}^{\infty} 
	\Big(\sum_{r_1=1}^{\infty}\<f_{k}^1\tilde\otimes f_{k}^2, e_{r_1} \otimes e_{r_2}\> \<f_{h}^1\tilde\otimes f_{h}^2, e_{r_1} \otimes e_{r_3}\> \Big)e_{r_2} \otimes e_{r_3}\\
	&=\sum_{r_2, r_3=1}^{\infty} 
	q(k,h, r_2, r_3)e_{r_2} \otimes e_{r_3},
\end{align*}
where 
	\begin{align*}
	q(k, h, r_2, r_3)&=\<f_{k}^2, e_{r_2}\>\<f_{h}^2, e_{r_3}\>\<f_{k}^1, f_{h}^1\>+\<f_{k}^2, e_{r_2}\>\<f_{h}^1, e_{r_3}\>\<f_{k}^1, f_{h}^2\>\\
	&+\<f_{k}^1, e_{r_2}\>\<f_{h}^2, e_{r_3}\>\<f_{k}^2, f_{h}^1\>+\<f_{k}^1, e_{r_2}\>\<f_{h}^1, e_{r_3}\>\<f_{k}^2, f_{h}^2\>.
\end{align*}
We recall that $\<f^i_k, f^j_h\>=r_{ij}(k-h)$, where $r_{ij}$ is the covariance (or cross-covariance) of our bivariate process. Then $\|\theta_n \otimes_1\theta_n\|_{L^2(\R;\R^2)}^2$ is a sum of expressions of the forms
\begin{align*}
	&\frac1{n^2}\sum_{k_1, \ldots, k_4=1}^n r_{11}(k_1-k_2)r_{22}(k_2-k_3)r_{11}(k_3-k_4)r_{22}(k_4-k_1),\\
	&\frac1{n^2}\sum_{k_1, \ldots, k_4=1}^n r_{12}(k_1-k_2)r_{12}(k_2-k_3)r_{12}(k_3-k_4)r_{12}(k_4-k_1),\\
	&\frac1{n^2}\sum_{k_1, \ldots, k_4=1}^n r_{12}(k_1-k_2)r_{12}(k_2-k_3)r_{11}(k_3-k_4)r_{22}(k_4-k_1),\\
	&\frac1{n^2}\sum_{k_1, \ldots, k_4=1}^n r_{12}(k_1-k_2)r_{11}(k_2-k_3)r_{21}(k_3-k_4)r_{22}(k_4-k_1),
\end{align*}
or similar to above expressions, with suitable shift of $k_{i+1}-k_i$ determined by the constant $s$. Since $s$ is fixed, the asymptotic behaviour of the sum and the convergence analysis remain unaffected for the second type of terms. In the following lemma we prove that the above expressions tend to $0$.  
	\begin{lemma}\label{gammaAn}
	Let us suppose that $H=H_1+H_2<3/2$. Let us denote by $\gamma_{ij}$, with $i,j\in\{1,2\}$, four real functions such that
	\begin{itemize}
		\item[1)] $|\gamma_{ij}(k)|\leq \gamma_{ij}(0)$;
		\item[2)]  there exists $\ell_{ij}>0$, $i,j\in\{1,2\}$ such that
		$$\lim_{k\to\infty} \frac{|\gamma_{ij}(k)|}{k^p}=
		\begin{cases} 
			\ell_{ii}  \quad{\text{if }i=j,p=2H_i-2}\\
			\ell_{ij}  \quad{\text{if }i\neq j, p=H-2}.
		\end{cases}
		$$
		
	\end{itemize}
	Let $F(k_1,k_2,k_3,k_4)$ be one of the following functions: 
	\begin{align*}
		&\gamma_{11}(k_1-k_2)\gamma_{22}(k_2-k_3)\gamma_{11}(k_3-k_4)\gamma_{22}(k_4-k_1),\\
		&\gamma_{12}(k_1-k_2)\gamma_{12}(k_2-k_3)\gamma_{12}(k_3-k_4)\gamma_{12}(k_4-k_1),\\
		&\gamma_{12}(k_1-k_2)\gamma_{12}(k_2-k_3)\gamma_{11}(k_3-k_4)\gamma_{22}(k_4-k_1),\\
		&\gamma_{12}(k_1-k_2)\gamma_{11}(k_2-k_3)\gamma_{21}(k_3-k_4)\gamma_{22}(k_4-k_1).
	\end{align*}
	Then 
	$$
	A_n=\frac 1{n^2}\sum_{k_1,\ldots, k_4=0}^{n-1} F(k_1,k_2,k_3, k_4) \overset{n\to \infty}{\to} 0. 
	$$

\end{lemma}
 \noindent \begin{proof} The proof is technical and requires several computations. The idea is to split the analysis according to the functions $F$, and then to the values of $\max\{H_1, H_2\}$ and $\min\{H_1, H_2\}$. When $\max\{H_1, H_2\}<1/2$, we easily obtain the statement from the following observation: being $|\gamma_{ii}(k)|\leq \ell_{ii}$ and $\gamma{ii}(|k|)\sim k^{2H_i-2}$, there exists a constant $C_i>0$ such that $|\gamma_{ii}(k)|\leq C_i k^{2H_i-2}$, $i=1,2$. The same holds for $\gamma_{ij}$, $i,j=1,2, i\neq j$, where $|\gamma_{ij}(k)|\leq C_{ij} k^{H-2}$. Then,
	\begin{align*}
	|A_n|&\leq \frac{1}{n^2}\sum_{k_1,\ldots, k_4=0}^{n-1} |\gamma_{11}(k_1-k_2)\gamma_{22}(k_2-k_3)\gamma_{11}(k_3-k_4)\gamma_{22}(k_4-k_1)| \\
	&\leq \frac{C}{n^2} \int_{[0,n]^4} |x_1-x_2|^{2H_1-2}|x_2-x_3|^{2H_2-2}|x_3-x_4|^{2H_1-2}|x_4-x_1|^{2H_2-2}dx\\
	&=\frac C{n^{6-4H}} \int_{[0,1]^4} |y_1-y_2|^{2H_1-2}|y_2-y_3|^{2H_2-2}|y_3-y_4|^{2H_1-2}|y_4-y_1|^{2H_2-2}dy
\end{align*}
 	and the integral is finite. Then $|A_n|\leq n^{4H-6}\to 0$. The same bound holds for all functions $F$ in the statement of the theorem. 
 	
 	When $\max\{H_1, H_2\}>1/2$, we have to refine the bound. The idea is to write $|A_n|$ as a discrete convolution of the functions $\gamma_{ii}$ and $\gamma_{ij}$. For example, denoting $\gamma_{11}^n (k):=|\gamma_{11}(k)|\1_{|k|< n}$ and $\gamma_{22}^n (k):=|\gamma_{22}(k)|\1_{|k|< n}$, we have  
 	\begin{align*}
 		|A_n|&\leq \frac 1{n^2}\sum_{k_1,k_2, k_3, k_4=1}^{n} |\gamma_{11}(k_1-k_2)\gamma_{22}(k_2-k_3)\gamma_{11}(k_3-k_4)\gamma_{22}(k_4-k_1)| \\
 		&\leq \frac 1{n^2}\sum_{k_1, k_3=1 }^{n} \sum_{k_2, k_4 \in \Z}  |\gamma_{11}(k_1-k_2)\gamma_{22}(k_2-k_3)\gamma_{11}(k_3-k_4)\gamma_{22}(k_4-k_1)|\\
 &\leq \frac{1}{n^2} \sum_{k_1, k_3=1}^n \Big(|\gamma_{11}^n|*|\gamma_{22}^n|(k_1-k_3)\Big)^2\leq \frac{1}{n} \sum_{k=-n}^n \Big(|\gamma_{11}^n|*|\gamma_{22}^n|(k)\Big)^2.
 	\end{align*}
 	 (the example is written for the function $F(k_1, k_2, k_3, k_4)=\gamma_{11}(k_1-k_2)\gamma_{22}(k_2-k_3)\gamma_{11}(k_3-k_4)\gamma_{22}(k_4-k_1)$). Then we apply Young's inequality for convolution: for $p,q, s\geq 1$ such that $\frac 1p+\frac 1q=1+\frac 1s$, we have 
 	$$
 	\|f*g\|_{\ell^s(\Z)}\leq \|f\|_{\ell^p(\Z)}\|g\|_{\ell^q(\Z)}.
 	$$
 	The use of Young's inequality in the proof changes, taking $s=2$ and $p,q$ according to the values of $H_1, H_2$. In each case there exist values of $p,q>1$ such that $|A_n|\to 0$ (see Lemma 6.1.7 in \cite{G24} for details). 
 	
 \end{proof}

\begin{proof}[Proof of Theorem \ref{CLT_hat_rho_eta}]
Taking $a_1=a_1(s)$, $a_2=a_2(s)$ and $a_3=a_3(s)$ given in \eqref{coefficients}, in Theorem \ref{VAR_SN} we proved that $\lim_{n\to \infty}\Var(\sqrt{n}(\hat \rho_n-\rho))=\sigma_\rho^2\in(0,+\infty)$. We apply Lemma \ref{gammaAn}, taking $\gamma_{ii}=r_{ii}$ and $\gamma_{ij}=r_{ij}$ for $i,j=1,2$. The expression for $\|\theta_n \otimes_1 \theta_n\|_{L^2}$ when $\theta_n$ is the kernel of $\hat \rho_n$, combined with Lemma \ref{gammaAn} implies that $(5)$ in Theorem \ref{FOURTH} holds. Then $\sqrt{n}(\hat \rho_n- \rho)\to N_{\rho}$, where $N_{\rho}\sim \mathcal N(0,\sigma_\rho^2)$. The same holds for $\hat \eta_{12, n}$, when $a_1=b_1(s)$, $a_2=b_2(s)$ and $a_3=b_3(s)$. Moreover, from straightforward computations it follows that
\begin{align*}
	\Cov(N_\rho, N_\eta)=\lim_{n\to \infty} n\E[(\hat\rho_n- \rho)(\hat \eta_{12, n}-\eta_{12})]<+\infty,
\end{align*}
and, by Theorem \ref{jointconvergence}, we have that $\sqrt{n}(\hat\rho_n- \rho,\hat \eta_{12, n}-\eta_{12})\overset{d}{\to}(N_\rho, N_{\eta})$.

\end{proof}
When $H=\frac 32$, the variance of the error changes but, under a different normalization, we have a analogous result. 
\begin{proof}[Proof of Theorem \ref{CLT_hat_rho_eta32}]
Let us prove the statement for $\hat \rho_n$. Since $\sqrt{n/{\log n}}(\hat\rho_n-\rho)=I_2(\theta_n)$, where $$\theta_n=\frac{1}{	\sqrt{n\log n }}\sum_{k=1}^n ( a_1 f_{k}^1\tilde \otimes f_{k}^2+a_2f_{k+s}^1\tilde \otimes f_{k}^2+a_3f_{k}^1\tilde     \otimes f_{k+s}^2),$$ then
\begin{align*}
	\Var(I_2(\theta_n))=\frac{1}{n\log n}\sum_{k_1, k_2=1}^n & \Big\< a_1 f_{k_1}^1\tilde \otimes f_{k_1}^2+a_2f_{k_1+s}^1\tilde \otimes f_{k_1}^2+a_3f_{k_1}^1\tilde     \otimes f_{k_1+s}^2,\\&\quad \quad
	a_1 f_{k_2}^1\tilde \otimes f_{k_2}^2+a_2f_{k_2+s}^1\tilde \otimes f_{k_2}^2+a_3f_{k_2}^1\tilde     \otimes f_{k_2+s}^2 \Big\>. 
\end{align*}
We study $\frac{ a_1^2}{\log(n)n} \sum_{k_1, k_2=1}^n \<f_{k_1}^1\tilde \otimes f_{k_1}^2\>\<f_{k_2}^1\tilde \otimes f_{k_2}^2\>$. We have
\begin{align*}
	&\frac{ a_1^2}{\log(n)n} \sum_{k_1, k_2=1}^n \<f_{k_1}^1\tilde \otimes f_{k_1}^2\>\<f_{k_2}^1\tilde \otimes f_{k_2}^2\>=\frac{ a_1^2}{\log(n)} \sum_{k=0}^n \Big(1-\frac{k}{n}\Big)\Big(r_{11}(k)r_{22}(k)+r_{12}(k)r_{21}(k)\Big).
\end{align*}
By Theorem \ref{Decay1} and Theorem \ref{cov}, for $H=\frac 32$, we have that 
\begin{align*}
	&\lim_{k\to +\infty} \frac{r_{11}(k)r_{22}(k)+r_{12}(k)r_{21}(k)}{k^{-1}}=\frac{\nu_1^2\nu_2^2}{4\a_1^2\a_2^2}\big((\rho^2-\eta_{12}^2)\frac {9}{16}+4(2H_1-1)(2H_2-1)H_1H_2\big)=\ell.
\end{align*}
By the condition for $\rho$ and $\eta_{12}$ in \eqref{domain}, it follows that $\ell\neq 0$ (for details, see Proposition 6.1.11 in \cite{G24}). Then
\begin{align*}
	&\lim_{n\to\infty}\frac{a_1^2}{\log n} \sum_{k=0}^{\infty} \Big(1-\frac{k}{n}\Big)\Big(r_{11}(k)r_{22}(k)+r_{12}(k)r_{21}(k)\Big)=\lim_{n\to\infty}\frac{a_1^2\ell}{\log n} \sum_{k=0}^{\infty} \frac 1k=a_1^2\ell
\end{align*} 
and $\lim_{n\to\infty} \frac{n}{\log n}\Var(\hat \rho_n-\rho)=(a_1+a_2+a_3)^2\ell$. We have to prove that $\|\theta_n\otimes_1 \theta_n\|_{L^2(\R^2;\R^2)}\to 0$. The proof is easier than proof of Lemma \ref{gammaAn}. We just prove that $$
\frac{1}{n^2\log n}\sum_{k_1,\ldots,k_4=1}^n r_{11}(k_1-k_2)r_{22}(k_2-k_3)r_{11}(k_3-k_4)r_{22}(k_4-k_1)\to 0.
$$ 
In the same way we prove that the whole $\|\theta_n\otimes_1\theta_n\|\to 0$. We observe that
\begin{align*}
&\Big|\frac{1}{n^2\log n}\sum_{k_1,\ldots,k_4=1}^n r_{11}(k_1-k_2)r_{22}(k_2-k_3)r_{11}(k_3-k_4)r_{22}(k_4-k_1)\Big|\\
&\leq C\frac{1}{(\log{n})^2}\int_{[0,1]^4} |y_1-y_2|^{2H_1-2}|y_2-y_3|^{2H_2-2}|y_3-y_4|^{2H_1-2}|y_4-y_1|^{2H_2-2}dy_1dy_2dy_3dy_4 \to 0,
\end{align*} 
for the other summands that form $\|\theta_n\otimes_1\theta_n\|$ the proof is analogous. 
The integral is finite because $2H_1-2<-1$ and $2H_2-2<-1$. The assumptions of Theorem \ref{FOURTH} hold then $\sqrt{\frac n{\log n}}(\hat\rho_n-\rho)\overset{d}{\to} N_\rho$, $\sqrt{\frac n{\log n}}(\hat\eta_{12,n}-\eta_{12})\to 
\overset{d}{\to} N_\eta$, where $N_\rho\sim \mathcal N(0,(a_1(s)+a_2(s)+a_3(s))^2\ell)$, $N_{\eta}\sim\mathcal N(0, (b_1(s)+b_2(s)+b_3(s))^2\ell)$ and, since $\lim_{n\to\infty}\frac n{\log n}\Cov(\hat \rho_n-\rho,\hat \eta_{12, n}-\eta_{12})$ exists, by Theorem \ref{jointconvergence} we can conclude that $\sqrt{\frac n{\log n}}(\hat \rho_n-\rho,\hat \eta_{12, n}-\eta_{12})\overset{d}{\to} (N_\rho, N_\eta)$. 

\end{proof}

We adapt  the approach in \cite[\S 7.3]{N13} to prove Theorem \ref{cumulantNC}. 
\begin{proof}[Proof of Theorem \ref{cumulantNC}]

Let us recall that $\tilde S_n=I_2(\theta_n)$ where 
$
\theta_n=n^{1-H}\sum_{k=1}^n ( a_1 f_{k}^1\tilde \otimes f_{k}^2+a_2f_{k+s}^1\tilde \otimes f_{k}^2+a_3f_{k}^1\tilde     \otimes f_{k+s}^2).
$
We can compute the cumulant of order $p$, $p\geq 2$, by exploiting the structure of $\tilde S_n$. Being $\tilde S_n=I_2(\theta_n)$ a double Wiener-It\^o integral and recalling that $\kappa_p(F)=(-1)^p \frac{\partial}{\partial^p t^p}|_{t=0} \log(\E[e^{itF}])$ it follows that 
$$
\kappa_p(\tilde S_n)= 2^{p-1}(p-1)! \frac{1}{n^{p(H-1)}} \sum_{k_1,\ldots, k_p=1}^n \sum_{h_1,\ldots, h_p=1}^{\infty}\prod_{i=1}^p \<\theta_n, e_{h_i}\otimes e_{h_{i+1}}\> 
$$ 
where $(e_h)_{h\in\N}$ is an orthonormal basis of $L^2(\R)$ (for the complete computations, see Theorem 6.1.14 in \cite{G24}). We study $\lim_{n\to\infty} \sum_{k_1,\ldots, k_p=1}^n \sum_{h_1,\ldots, h_p=1}^{\infty}\prod_{i=1}^p \<\theta_n, e_{h_i}\otimes e_{h_{i+1}}\>$. Substituting the expression of $\theta_n$, we obtain that $\kappa_p(\tilde S_n)$ is a finite sum of \begin{align}\label{genC}
	&\frac{2^{p-1} (p-1)!C(a_1,a_2,a_3)}{n^{p(H-1)} }\sum_{h_1, \ldots, h_p=1}^{\infty}  \sum_{k_1,\ldots, k_p=1}^n  \prod_{i=1}^p \<f_{ k_i^1}^{r_i^1}\otimes f_{ k_i^2}^{r_i^2}, e_{h_i}\otimes e_{h_i+1}\>\\
	&=\frac{2^{p-1} (p-1)!C(a_1,a_2,a_3)}{n^{p(H-1)} } \sum_{k_1,\ldots, k_p=1}^n \<f_{k_1^2}^{r_1^2}, f_{k_2^1}^{r_2^1}\>\<f_{k_2^2}^{r_2^2},f_{k_3^1}^{r_3^1}\>\cdots \<f_{k_{p-1}^2}^{r_{p-1}^2}, f_{k_p^1}^{r_{p}^1}\>\<f_{k_p^2}^{r_p^2}, f_{k_1^1}^{r_1^1}\>.
\end{align}
Here $C(a_1,a_2,a_3)=a_1^{p_1}a_2^{p_2}a_3^{p_3}$ with $p_1+p_2+p_3=p$ while for all $i=1,\ldots, p $ we have $(r_i^1, r_i^2)\in \{(1,2), (2,1)\}$ and  $(k_i^1, k_i^2)\in \{(k_i, k_i), (k_i+s, k_i), (k_i, k_i+s)\}$. There is a restriction on the couples $(r_i^2, r_{i+1}^1)$. We notice that 
\begin{itemize}
	\item if $(r_i^2, r_{i+1}^1)=(1,1)$, then $(r_{i+1}^2, r_{i+2}^1)\in\{(2,2), (2,1)\}$; 
	\item if $(r_i^2, r_{i+1}^1)=(2,2)$, then $(r_{i+1}^2, r_{i+2}^1)\in\{(1,1), (1,2)\}$;
	\item if $(r_i^2, r_{i+1}^1)=(1,2)$, then $(r_{i+1}^2, r_{i+2}^1)\in\{(1,2), (1,1)\}$;
	\item if $(r_i^2, r_{i+1}^1)=(2,1)$, then $(r_{i+1}^2, r_{i+2}^1)\in\{(2,1), (2,2)\}$;
	\item if $p$ is odd, there exists at least one couple $(r_i^2, r_{i+1}^2)\in\{(1,2), (2,1)\}$ and the number of couples in $\{(1,2), (2,1)\}$ is odd;
	\item if there exists a couple equal to $(1,1)$, then there exists a couple equal to $(2,2)$, and if the number of couples $(1,1)$ is $m$, then the number of couples $(2,2)$ is $m$.
\end{itemize} 
We split $\kappa_p(\tilde S_n)$ in two different sums. The first one is 
$$
 A_n= \frac1{n^{p(H-1)}}\sum_{h_1, \ldots, h_p=1}^{\infty}  \sum_{\underset{\exists i:|k_{i+1}-k_i|\leq s+2}{k_1,\ldots, k_p=1}}^n   \prod_{i=1}^p \<f_{ k_i^1}^{r_i^1}\otimes f_{ k_i^2}^{r_i^2}, e_{h_i}\otimes e_{h_i+1}\>. 
$$
Since 
			\begin{align*}
	& |A_n|\leq \frac {1}{n^{(H-1)p}}\Big|\sum_{\underset{|k_1-k_p|\leq s+2}{k_1,\ldots, k_p=1}}^n \<f_{k_1^1}^{i_{1}^1}, f_{k_p^2}^{i_{p}^2}\>\<f_{k_1^2}^{i_{1}^2}, f_{k_2^1}^{i_2^1}\>\cdots \<f_{k_{p-1}^2}^{i_{p-1}^2}, f_{k_p^1}^{i_p^1}\>\Big|\\
	&=\frac {1}{n^{(H-1)p}}\Big|\sum_{\underset{|k_1-k_p|\leq s+2}{k_1, k_p=1}}^n \<f_{k_1^1}^{i_{1}^1}, f_{k_p^2}^{i_{p}^2}\>
	\sum_{h_1, \ldots, h_{p-1}=1}^{+\infty} \<f_{k_1^2}^{i_{1}^2}, e_{h_1}\> \<f_{k_2^1}^{i_2^1}, e_{h_1}\>\cdots \<f_{k_{p-1}^2}^{i_{p-1}^2}, e_{h_{p-1}}\>\<f_{k_p^1}^{i_p^1}, e_{h_{p-1}}\>\Big|\\
	&=\frac {1}{n^{(H-1)p}}\Big|\sum_{\underset{|k_1-k_p|\leq s+2}{k_1,\ldots, k_p=1}}^n \<f_{k_1^1}^{i_{1}^1}, f_{k_p^2}^{i_{p}^2}\> \sum_{h_1, \ldots, h_{p-1}=1}^{+\infty} \<f_{k_1^2}^{i_{1}^2}, e_{h_1}\> \<f_{k_2^1}^{i_2^1}\otimes f_{k_2^2}^{i_2^2}, e_{h_1}\otimes e_{h_2}\>\cdots\\
	&\cdots\<f_{k_{p-1}^1}^{i_{p-1}^1}\otimes f_{k_{p-1}^2}^{i_{p-1}^2}, e_{h_{p-2}}\otimes e_{h_{p-1}}\> \<f_{k_p^1}^{i_p^1}, e_{h_{p-1}}\>\Big|\\
	&=\frac {1}{n^{(H-1)2}}\Big|\sum_{\underset{|k_1-k_p|\leq s+2}{k_1, k_p=1}}^n \<f_{k_1^1}^{i_{1}^1}, f_{k_p^2}^{i_{p}^2}\> \Big\<\cdots\Big( f_{k_1^2}^{i_1^2}\otimes_1\Big(\frac{1}{n^{H-1}}\sum_{k_2=}^n f_{k_2^1}^{i_2^1}\otimes f_{k_2^2}^{i_2^2}\Big)\otimes_1\cdots\\
	&\cdots\otimes_1\Big(\frac1{n^{H-1}}\sum_{k_{p-1}}^n f_{k_{p-1}^1}^{i_{p-1}^1}\otimes f_{k_{p-1}^2}^{i_{p-1}^2}\Big)\Big), f_{k_p^1}^{i_p^1}\Big\> \\
	&\leq \frac {1}{n^{(H-1)2}}\sum_{\underset{|k_1-k_p|\leq s+2}{k_1, k_p=1}}^n   \|f_{k_1^1}^{i_{1}^2}\|\| f_{k_p^2}^{i_{p}^1}\| \Big\|\frac{1}{n^{H-1}}\sum_{k_2=1}^n f_{k_2^1}^{i_2^1}\otimes f_{k_2^2}^{i_2^2}\Big\|\cdots\Big\|\frac1{n^{H-1}}\sum_{k_{p-1}}^n f_{k_{p-1}^1}^{i_{p-1}^1}\otimes f_{k_{p-1}^2}^{i_{p-1}^2}\Big\|\\
	&\leq \frac{C}{n^{(H-1)2}L^p(n)} \sum_{\underset{|k_1-k_p|\leq s+2}{k_1, k_p=1}}^n  1\leq \frac{C_2}{n^{2H-3}}
\end{align*}
when $H>\frac 32$, then $A_n \to 0$. Then we study 
$$
B_n=\lim_{n\to+\infty} \frac{1}{n^{p(H-1)}}\sum_{k_1,\ldots, k_p=1, \forall i |k_{i+1}-k_i|\geq s+3}^n \prod_{i=1}^{p} \<f_{k_{i}^2}^{r_i^2}, f_{k_{i+1}^1}^{r_{i+1}^1}\>.
$$
We recall that $\<f_{k_{i}^2}^{j}, f_{k_{i+1}^1}^{j}\>=r_{jj}(k_{i+1}^1-k_i^2)$, $j=1,2$ and $\<f_{k_{i}^2}^{1}, f_{k_{i+1}^1}^{2}\>=r_{12}(k_{i}^2-k_{i+1}^1)$ and also that $|r_{jj}(k)|\leq |r_{jj}(0)|, |r_{12}(k)|\leq |r_{12}(0)|$ and, by Theorem \ref{cov} and Theorem \ref{Decay1}, there exist $\lim_{k\to \infty} k^{2-2H_j}r_{jj}(k)=\ell_{jj}$, $\lim_{k\to \infty} k^{2-H}r_{12}(k)=\ell_{12}$, then we can write 
\begin{align*}
&r_{11}(k)=k^{2H_1-2} L_{11}(k)\quad{r_{22}(k)=k^{2H_2-2}L_{22}(k)}	\\
&r_{12}(k)=k^{H-2} L_{12}(k)\quad{r_{21}(k)=k^{H-2}L_{21}(k)}.
\end{align*}   
Then $
B_n=\frac{1}{n^p} \sum_{\underset{\forall i |k_{i+1}-k_i|\geq s+3}{k_1,\ldots, k_p=1}}^n \prod_{i=1}^{p}\Big(\frac{|k_{i+1}^2-k_i^1|}{n}\Big)^{-\b_i} L^i(k_{i+1}^2-k_i^1)$
where $\b_i\in\{2-2H_1, 2-2H_2, 2-H\}$ for all $i$ and $\sum_{i=1}^n \beta_i =p(2-H)$. Then
\begin{align*}
B_n&=\int_{\R_+^p}\sum_{\underset{\forall i |k_{i+1}-k_i|\geq s+3}{k_1,\ldots, k_p=1}}^n  \prod_{i=1}^{p} \Big(\frac{|k_{i+1}^2-k_i^1|}{n}\Big)^{-\b_i} L^i(k_{i+1}^2-k_i^1)\1_{[\frac{k_1-1}n, \frac{k_1}n)}(x_1) \cdots  \1_{[\frac{k_p-1}n, \frac{k_p}n)}(x_p) dx\\
&=\int_{\R_+^p} \lambda_n(x_1,\ldots, x_p)dx_1\ldots dx_p,
\end{align*}
and
	$$
|\lambda_n(x_1,\ldots, x_p)| \leq C(s+2)^{p(2-H)} \1_{[0,1]^p}(x_1,\ldots, x_p) \prod_{i=1}^p |x_{i+1}-x_i|^{-\b_i}.
$$
It follows from the fact that the functions $L_{11}$, $L_{22}$, $L_{12}$ and $L_{21}$ are bounded. 
Then we apply Lebesgue's Theorem, noticing that
			\begin{align*}
	&\lim_{n\to +\infty} \lambda_n(x_1,\ldots, x_p)\\
	&= \sum_{\underset{\forall i |k_{i+1}-k_i|\geq s+3}{k_1,\ldots, k_p=1}}^n \prod_{i=1}^{p} \Big(\frac{|k_{i+1}^2-k_i^1|}{n}\Big)^{-\b_i} L^i(k_{i+1}^2-k_i^1)\1_{[\frac{k_1-1}n, \frac{k_1}n)}(x_1) \cdots  \1_{[\frac{k_p-1}n, \frac{k_p}n)}(x_p) \\
	&= \prod_{i=1}^p\Big ( \ell_{21}\1_{\underset{ x_{i+1}>x_i}{\b_i=2-H}}+\ell_{12}\1_{\underset{x_{i+1}\leq x_1}{\b_i=2-H} }+\ell_{11}\1_{\b_i=2-2H_1}+\ell_{22}\1_{\b_i=2-2H_2} \Big)|x_{i+1}-x_i|^{-\b_i}\1_{[0,1]}(x_i) .
\end{align*}
Then, by adding together all the sums that form $\kappa_p(\tilde S_n)$, we obtain the first part of the statement. 
The analysis of the variance is approached recalling that $\kappa_2(\tilde S_n)=\Var(\tilde S_n)$. In this case we can compute explicitly the integrals appearing in the limit, obtaining the values in the statement. By the conditions in \eqref{domain}, it follows that, for $H>\frac 32$, that limit is not $0$.  
\end{proof}
\begin{proof}[Proof of Theorem \ref{th:nogaus}]
	As a consequence of Theorem \ref{cumulantNC}, for suitable choices of $a_1, a_2, a_3$, for all $p\geq 2$,
	 \begin{align*}
	\lim_{n\to \infty}\kappa_p( n^{2-H}(\hat\rho_n-\rho))=2^{p-1}(a_1(s)+&a_2(s)+a_3(s))^p\sum_{i_1, \ldots, i_p=1}^2\\& \sum_{\underset{i_2\neq j_1,\ldots, i_p\neq j_{p-1}, i_1\neq j_p}{j_1,\ldots, j_p=1}}^2\int_{[0,1]^p}  z_{i_1 j_1}(x_1, x_2) z_{i_2 j_2}(x_2, x_3)\ldots z_{i_p j_p}(x_p, x_1) dx
\end{align*}
and
\begin{align*}
	\lim_{n\to \infty}\kappa_p( n^{2-H}(\hat\eta_{12,n}-\eta_{12}))=2^{p-1}(b_1(s)+&b_2(s)+b_3(s))^p\sum_{i_1, \ldots, i_p=1}^2 \\& \sum_{\underset{i_2\neq j_1,\ldots, i_p\neq j_{p-1}, i_1\neq j_p}{j_1,\ldots, j_p=1}}^2\int_{[0,1]^p}  z_{i_1 j_1}(x_1, x_2) z_{i_2 j_2}(x_2, x_3)\ldots z_{i_p j_p}(x_p, x_1) dx
\end{align*}
where 
$$
z_{11}(x,y)=  \frac{2\nu_1^2H_1(2H_1-1)}{\a_1^2} |x-y|^{2H_1-2}
$$
$$
z_{22}(x,y)= \frac{2\nu_2^2H_2(2H_2-1)}{\a_2^2} |x-y|^{2H_2-2},
$$
$$
z_{12}(x, y)=  \frac{2\nu_1\nu_2H(H-1) }{\a_1\a_2} \times \begin{cases}
(\rho-\eta_{12})(x-y)^{H-2} \quad{x>y}\\
(\rho+\eta_{12})  (y-x)^{H-2} \quad{x\leq y}
\end{cases}
$$
$$
z_{21}(x, y)=  \frac{2\nu_1\nu_2H(H-1) }{\a_1\a_2} \times \begin{cases}
(\rho+\eta_{12})(x-y)^{H-2} \quad{x>y}\\
(\rho-\eta_{12})  (y-x)^{H-2} \quad{x\leq y}.
\end{cases}
	$$
	Moreover, for all $\varepsilon >0$,  
	\begin{align*}
	&\P\Big(n^{2-H}|\hat\rho_n-\rho|>\sqrt{\frac{\sup_n \kappa_2(n^{2-H}(\hat\rho_n-\rho))}{\varepsilon}}\Big)\leq \frac{\Var(n^{2-H}(\hat\rho_n-\rho))\varepsilon}{\sup_n \kappa_2(n^{2-H}(\hat\rho_n-\rho))}\leq \varepsilon. 
	\end{align*}
	The same holds for $\hat \eta_{12,n}$. Then the sequence is tight, and there exists a subsequence that converges in distribution. The limit is given by $Z_\rho=I_2(f_\rho)+N_\rho$, where $I_2(f_\rho)$ is a double Wiener-It\^o integral and $N_\rho$ is an independent (of $I_2(f_\rho)$) Gaussian random variable. This fact is proved in \cite{PN12}. Being $Z_\rho$ determined by its cumulants, we can apply Proposition 5.2.2 in \cite{NP12} (Method of moments and cumulants) that implies the convergence of the whole sequence to $Z_\rho$. It is not easy to establish that 
		$\lim_{n\to \infty}\kappa_4(n^{2-H}(\hat\rho_n-\rho))\neq 0$; if it is equal to $0$, the Fourth Moment theorem holds (Theorem \ref{FOURTH}), then the limit is Gaussian (i.e. $f_\rho=0$).  
	
\end{proof}

Let us now collect the proofs of the results related to the high frequency estimators. For some details of the proofs, we refer to Section 7 in \cite{G24}.
\begin{proof}[Proof of Lemma \ref{short-time-inverse}]
This is a direct consequence of Lemma \ref{shorttime}, inverting the asymptotic relations for $\Cov(Y_t^1, Y_{t+s}^2)$ and $\Cov(Y_{t+s}^1. Y_t^2)$, $t\in \R$, $s\to 0$. 
\end{proof}
\begin{proof}[Proof of Proposition \ref{unbiasEst}]
	We have that
	\begin{align*}	
	\E[\tilde\rho_n]&=\frac{1}{\nu_1\nu_2n\Delta_n^H}\sum_{k=0}^{n-1} \E\Big[\Big(Y_{(k+1)\Delta_n}^1-Y_{k\Delta_n}^1\Big)\Big(Y_{(k+1)\Delta_n}^2-Y_{k\Delta_n}^2\Big)\Big]\\
	&=\frac{1}{\nu_1\nu_2n\Delta_n^H}\sum_{k=0}^{n-1} \Big(2\Cov(Y_0^1,Y_0^2)-\Cov(Y_{\Delta_n}^1, Y_0^2)-\Cov(Y_0^1,Y_{\Delta_n}^2)\Big)\\
	&=\frac{2\Cov(Y_0^1,Y_0^2)-\Cov(Y_{\Delta_n}^1, Y_0^2)-\Cov(Y_0^1,Y_{\Delta_n}^2)}{\nu_1\nu_2\Delta_n^H}\\
	&=\rho+O(\Delta_n^{\min(1,2-H)})\overset{n\to \infty}{\to} \rho.
\end{align*} 
The same holds for $\tilde \eta_{12, n}$.
\end{proof}
Now we focus on the results related to $\tilde \rho_n$. We recall that \begin{equation}\label{decom}
Y_{(k+1)\Delta_n}^i=Y_{k\Delta_n}^ie^{-\a_i\Delta_n}+\xi_{(k+1)\Delta_n}^i,
\end{equation}
where $\xi_{k\Delta_n}^i=\nu_i\int_{(k-1)\Delta_n}^{k\Delta_n} e^{-k\Delta_n \a_i} e^{\a_i u} dB_u^{H_i}$. Moreover 
$$
\xi_{k\Delta_n}^i=\nu_i(B_{k\Delta_n}^{H_i}-B_{(k-1)\Delta_n}^{H_i})+R_{k\Delta_n}^{i},
$$
where $\Var(R_{k\Delta_n}^i)=C(H_i)\Delta_n^{2H_i+2}+o(\Delta_n^{2H_i+2})$ and $o(\Delta_n^{2H_i+1})$ does not depend on $k$. Using \eqref{decom}, we can write
\begin{align}\label{Repr_rho}
	\tilde \rho_n-\rho&=\frac{1}{\nu_1\nu_2n\Delta_n^H}\sum_{k=1}^n (Y_{(k+1)\Delta_n}^1-Y_{k\Delta_n}^1)(Y_{(k+1)\Delta_n}^2-Y_{k\Delta_n}^2)-\rho\\
	&=\frac{(e^{-\a_1\Delta_n}-1)(e^{-\a_2\Delta_n}-1)}{\nu_1\nu_2n\Delta_n^H}\sum_{k=1}^{n-1}Y_{k\Delta_n}^1 Y_{k\Delta_n}^2+\frac{e^{-\a_1\Delta_n}-1}{\nu_1\nu_2n\Delta_n^H} \sum_{k=1}^{n-1}Y_{k\Delta_n}^1 \xi_{(k+1)\Delta_n}^2\notag\\
	&+\frac{e^{-\a_2\Delta_n}-1}{\nu_1\nu_2n\Delta_n^H} \sum_{k=1}^{n-1}Y_{k\Delta_n}^2 \xi_{(k+1)\Delta_n}^1+\frac{1}{\nu_1n\Delta_n^H}\sum_{k=1}^{n-1} R_{(k+1)\Delta_n}^1 (B_{(k+1)\Delta_n}^2-B_{k\Delta_n}^2)\notag\\
	&\frac{1}{\nu_2n\Delta_n^H}\sum_{k=1}^{n-1} R_{(k+1)\Delta_n}^2 (B_{(k+1)\Delta_n}^1-B_{k\Delta_n}^1)+\frac{1}{\nu_1\nu_2n\Delta_n^H}\sum_{k=1}^{n-1} R_{(k+1)\Delta_n}^1R_{(k+1)\Delta_n}^2\notag\\
	&+\frac{1}{n\Delta_n^H}\sum_{k=1}^{n-1}  (B_{(k+1)\Delta_n}^1-B_{k\Delta_n}^1)(B_{(k+1)\Delta_n}^2-B_{k\Delta_n}^2)-\rho\notag
\end{align}
\begin{proof}[Proof of Theorem \ref{Cons_rho}]
The consistency of $\tilde \rho_n$ follows from the representation given in \eqref{Repr_rho}. From straightforward computations, it follows that each sum in the representation besides \\
$\frac{1}{n\Delta_n^H}\sum_{k=1}^{n-1}  (B_{(k+1)\Delta_n}^1-B_{k\Delta_n}^1)(B_{(k+1)\Delta_n}^2-B_{k\Delta_n}^2)$ converges to $0$ in $L^2(\P)$ and then in probability when Assumptions \eqref{AssumptionDelta_a} hold. Moreover, by Theorem \ref{cov_nfBm},  
\begin{align*}
&\frac{1}{n\Delta_n^H}\sum_{k=1}^{n-1}  \E[(B_{(k+1)\Delta_n}^1-B_{k\Delta_n}^1)(B_{(k+1)\Delta_n}^2-B_{k\Delta_n}^2)]=\frac{1}{n}\sum_{k=1}^{n-1}  \E[(B_{k+1}^1-B_{k}^1)(B_{k+1}^2-B_{k}^2)]=\rho,
\end{align*}
and 
\begin{align*}
&\Var\Big(\frac{1}{n\Delta_n^H}\sum_{k=1}^{n-1}  \E[(B_{(k+1)\Delta_n}^1-B_{k\Delta_n}^1)(B_{(k+1)\Delta_n}^2-B_{k\Delta_n}^2)]\Big)\leq\frac{C_1}{n^2}\sum_{k, h=1}^{n-1} |k-h|^{2H-4}\leq \frac{C_2}{n} \sum_{k=1}^{n-1} \frac 1{k^{4-2H}}
\end{align*}
and the right-hand side tends to $0$ for $H\in (0,2)$. For $H\in(0,1)$ we prove the consistency of $\tilde \eta_{12, n}$ in a similar way. 
\end{proof}
\begin{proof}[Proof of Theorem \ref{CLT_tilde_rho}]
	By following the proof of Theorem \ref{Cons_rho}, we deduce that $\Var(\tilde\rho_n-\rho)=O(n^{-1})$. We consider $\sqrt{n}(\tilde\rho_n-\rho)$. Besides $\frac{1}{n\Delta_n^H}\sum_{k=1}^{n-1}  (B_{(k+1)\Delta_n}^1-B_{k\Delta_n}^1)(B_{(k+1)\Delta_n}^2-B_{k\Delta_n}^2)-\rho$, the terms in \eqref{Repr_rho} multiplied by $\sqrt{n}$ converge to $0$ in probability, then in distribution when Assumptions \eqref{AssumptionDelta_a} and Assumptions \eqref{AssumptionDelta_b} hold. Thanks to the self-similarity of the mfBm, we have that
\begin{align*}
	&\frac{1}{\sqrt{n}\Delta_n^H}\sum_{k=1}^{n-1}  \Big((B_{(k+1)\Delta_n}^1-B_{k\Delta_n}^1)(B_{(k+1)\Delta_n}^2-B_{k\Delta_n}^2)-\rho\Big)\sim \frac{1}{\sqrt{n}}\sum_{k=1}^{n-1}  \Big((B_{k+1}^1-B_{k}^1)(B_{k+1}^2-B_{k}^2)-\rho\Big).
\end{align*}
	It converges to a Gaussian random variable by the same arguments as Theorem \ref{CLT_hat_rho_eta}, replacing $Y_k^1$ with $B_{k+1}^1-B_k^1$ and $Y_k^2$ with $B_{k+1}^2-B_k^2$. Also in this case the sequence can be written as a double Wiener-It\^o integral, with respect to a different kernel. Since for $i, j=1,2$, $i\neq j$, when $|k-h|\to \infty$ we have 
	\begin{align*}
	&\E[(B_{k+1}^i-B_k^i)(B_{h+1}^i-B_h^i)]\sim 	\E[Y_k^i Y_h^i] \sim |k-h|^{2H_i-2}\\
	&\E[(B_{k+1}^i-B_k^i)(B_{h+1}^j-B_h^j)]\sim 	\E[Y_k^i Y_h^j] \sim |k-h|^{H-2},
	\end{align*}   
we can use the same arguments to conclude that 
\begin{align*}
	&\lim_{n\to \infty}\Var\Big(\frac{1}{\sqrt{n}\Delta_n^H}\sum_{k=1}^{n-1}  \Big((B_{(k+1)\Delta_n}^1-B_{k\Delta_n}^1)(B_{(k+1)\Delta_n}^2-B_{k\Delta_n}^2)-\rho\Big)\Big)\\
	&=\Var(B_1^{H_1} B_1^{H_2})+2\sum_{k=1}^{+\infty} \Cov\Big(B_1^{H_1} B_1^{H_2}, (B_{k+1}^{H_1}-B_k^{H_1})(B_{k+1}^{H_2}-B_k^{H_2})\Big)
\end{align*}
and 
$$
\lim_{n\to\infty}\kappa_4\Big(\frac{1}{\sqrt{n}\Delta_n^H}\sum_{k=1}^{n-1}  \Big((B_{(k+1)\Delta_n}^1-B_{k\Delta_n}^1)(B_{(k+1)\Delta_n}^2-B_{k\Delta_n}^2)-\rho\Big)=0.
$$
The statement follows from Theorem \ref{FOURTH}.
\end{proof}
\begin{proof}[Proof of Theorem \ref{NU_1}]
	We consider $\hat \rho_n$ in \eqref{estRHO} when $Y^2=Y^1$, i.e. $\a_1=\a_2, \HH=H_1=H_2$, $\nu_1=\nu_2, \eta_{12}=0$ and $\rho=1$. By Theorem \ref{consistency}, we have that $\hat\rho_n\to \rho=1$ in $L^2(\P)$ and then in probability. Then $
	\frac{\hat \nu_n^2}{\nu^2}=\hat \rho_n\to 1   
	$
	and so $\hat  \nu_n^2 \to \nu^2 $ in $L^2(\P)$ and then in probability. The second part of the statement follows from Theorem \ref{S_n_conv} (or equivalently from Theorem \ref{CLT_hat_rho_eta}), having that, for $\HH<\frac 34$,  
	\begin{align*}
		\sqrt{n}(\hat \nu_n^2 -\nu^2)=\nu^2\sqrt{n}(\hat \rho_n-1)\overset{d}{\to} \nu^2 \mathcal N(0,\hat \sigma^2).
	\end{align*}
When $\HH=\frac 34$, as a consequence of Theorem \ref{CLT_hat_rho_eta32}, we have that $\sqrt{\frac{n}{\log n }}(\hat \nu_n^2-\nu^2)\to \mathcal N(0,\sigma^2)$, where $\sigma^2=\lim_{n\to\infty}\frac{n}{\log n }\Var(\hat \nu_n^2-\nu^2)$. 
Instead when $\HH>\frac 34$, the statement follows as a consequence of Theorem \ref{cumulantNC} and Theorem \ref{th:nogaus}. In this case we also provide the precise law of the limit random variable $R^{\HH}$. We have that
\begin{align*}
\kappa_p(R^{\HH})=2^{p-1}(p-1)!  \int_{[0,1]^p} \prod_{i=1}^p f(x_i, x_{i+1}) dx_1\ldots dx_p  
\end{align*}
where $x_{p+1}=x_1$ and 
$$
f(x,y)= \frac{2(1-e^{-\a s})\nu^2}{\a^{2}e^{-\a s}I(s)\beta(2-2H, H-1/2)}  |x-y|^{2\HH-2}.
$$ 
From standard computations, we have that
\begin{align*}
	\kappa_p(I_2(g))=2^{p-1}(p-1)!  \int_{[0,1]^p} \prod_{i=1}^p f(x_i, x_{i+1}) dx_1\ldots dx_p
\end{align*}
when $g(t,t')=\frac{2 (1-e^{-\a s})\nu^2}{\a^{2}e^{-\a s}I(s)\beta(2-2\HH, \HH-1/2)}\int_0^1(u-t)_+^{\HH-\frac 32}(u-t')_+^{\HH-\frac 32}du$ and $I_2$ denotes the double Wiener-It\^o integral. Then, being $R^{\HH}$ and $I_2(g)$ determined by their cumulants, $R^{\HH} \overset{d}{=} I_2(g)$. 
\end{proof}
\begin{proof}[Proof of Theorem \ref{NU_2}]
We consider $\tilde \rho_n$ in \eqref{hat_rho_n} when $Y^2=Y^1$, i.e. $\a_1=\a_2, \HH=H_1=H_2$, $\nu_1=\nu_2, \eta_{12}=0$ and $\rho=1$. Under these assumptions, we have that
$
\tilde\nu_n^2=\nu^2 \tilde \rho_n.
$
Theorem \ref{Cons_rho} ensures that $\tilde \rho_n \to \rho=1$ in $L^2(\P)$ and then in probability when $\Delta_n\to 0$ and $n\Delta_n\to +\infty$ as $n\to +\infty$. Then 
$
\tilde\nu_n^2=\nu^2 \tilde \rho_n\to \nu^2
$
in $L^2(\P)$ and in probability.
If we assume that $n\Delta_n^2 \to 0$ and $n\Delta_n^{4-4\HH}\to 0$, Assumptions \ref{AssumptionDelta_b} hold (there $\HH=H_1+H_2$ whereas in this setting we denote $\HH=H_1=H_2$.) Then the assumptions of Theorem \ref{CLT_tilde_rho} are verified, then, when $\HH<\frac 34$, we have 
$$
\sqrt{n}(\tilde \rho_n-1)\overset{d}{\to} N
$$
where $N\sim \mathcal N(0,\sigma^2)$ and $\sigma^2=\lim_{n\to \infty} \Var(\sqrt{n}(\tilde \rho_n-1))$. Then
\begin{align*}
	&\sqrt{n}(\tilde\nu_n^2-\nu^2)=\nu^2 \sqrt{n}(\tilde \rho_n-1)\overset{d}{\to} \nu^2 N=\tilde N
\end{align*}
where $\Var(\tilde N)=\nu^4 \Var(N)$.
\end{proof}

\section{Acknowledgements}
We are grateful to Lucia Caramellino, Stefano De Marco, Georgios Fellouris, Ivan Nourdin, Mikko Pakkanen, Tommaso Proietti for valuable comments and suggestions. 
We thank Jean-François Coeurjolly for sharing the code for simulating the multivariate fractional Brownian Motion.

Funding: PP was supported by the project PRICE, financed by the Italian Ministry of University and Research under the program PRIN 2022, Prot. 2022C799SX. RD and PP were supported by the project E83C25000470005 financed by University of Rome Tor Vergata.

\appendix

\section{Gaussian Chaos}

We recall here some facts related to Malliavin calculus on the Wiener space. The following remark recalls some properties of Hilbert spaces and tensor product.	

\begin{remark}
	Let $\H$ be a separable Hilbert space. For an integer $p\geq 2$, the Hilbert spaces $\H^{\otimes p}$ and $\H^{\odot p}$ are the $p$th \textit{tensor product} of $\H$ and the $p$th \textit{symmetric tensor product} of $\H$ respectively. If $f\in \H^{\otimes p}$, then $f=\sum_{i_1,\ldots, i_p=1}^{\infty} a(i_1, \ldots, i_p) e_{i_1}\otimes e_{i_2} \cdots \otimes e_{i_p}$, where $(e_{i_1}\otimes \cdots \otimes e_{i_p})_{i_1,\ldots,i_p=1}^{\infty}$ is an orthonormal basis of $\H$. The symmetrization $\tilde f$ of $f$ is the element of $\H^{\odot p}$ such that 
	$$
	\tilde f=\frac{1}{p!}\sum_{\sigma}\sum_{i_1,\ldots,i_p=1}^{\infty} a_{i_1,\ldots, i_p} e_{\sigma(i_1)}\otimes\cdots \otimes e_{\sigma(i_p)}. 
	$$ 
	where the first sum runs over all $\sigma$ permutation of $\{1,\ldots, p\}$. The $r$th contraction of two tensor products $e_{i_1}\otimes\cdots e_{i_p}$ and $e_{j_1}\otimes \cdots e_{j_q}$ is an element of $\H^{p+q-2r}$ such that
	\begin{align}\label{contraction}
		&(e_{i_1}\otimes\cdots e_{i_p})\otimes_{r}(e_{j_1}\otimes \cdots e_{j_q})\\
		&=\Big(\prod_{k=1}^{r}\<e_{i_k}, e_{j_k}\>_{\H}\Big) e_{i_{r+1}}\otimes \cdots e_{i_p}\otimes e_{j_{r+1}}\otimes \cdots e_{j_q}\notag.
	\end{align}
\end{remark}

Let us consider a complete probability space $(\Omega, \mathcal F, \P)$ and a real separable Hilbert space $\mathbb{H}$ with inner product denoted by $\<\cdot, \cdot\>_\mathbb{H}$. From now on we denote $L^2(\Omega):=L^2(\Omega, \mathcal F, \P)$.	 
\begin{definition}\label{IsonormalGaussian}
	An \textbf{isonormal Gaussian field} over $\mathbb{H}$ is a family $X=\{\, X(h) \,:\, h\in \mathbb{H}\}$ of centered jointly Gaussian random variables on $(\Omega, \mathcal F, \P)$, whose covariance structure is given by 
	$$
	\E[X(h)X(h' )] = \< h, h' \>_\mathbb{H}, \quad{\forall h, h'\in \mathbb{H}}.
	$$ 
\end{definition} 
The following example introduces the representation as an isonormal Gaussian process of the 2fOU. 

\begin{example}\label{Bi_isonormal}
	Let $W_1, W_2$ be two independent Brownian motions on $\R$ and denote $W=(W^1, W^2)$. Let us consider $\H=L^2(\R; \R^2)$ (the space of the functions $f$ from $\R$ to $\R^2$ such that $\int_{\R} \|f(t)\|^2 dt <\infty$). We have the isonormal Gaussian field $X$ on $\H$ given by the $L^2(\Omega)$-closure of the linear space generated by the $W$. Moreover, $X$ is the family of Ito's integrals with respect to $W$. 
	
	Now, let $\{f^i(t, \cdot)\}_{t\in \R, i=1,2}$ be a family of functions in $L^2(\R^2; \R^2)$. We define a bivariate Gaussian process $Y=(Y^1, Y^2)$ as		 
	\begin{equation}\label{BivPr}
		Y^i_t=\int_{\R} \< f^i(t,s),W(ds)\>_{\R^2}. 
	\end{equation}
	We can look at $(Y^i_t)$, $i=1,2$, $t\in\R$ as a particular expression for the field $X$: for $i=1,2$ and $t\in \R$, we consider $g_t^i\in L^2(\R;\R^2)$ such that $g_t^i=f^i(t,\cdot)$. Then $Y_t^i=X(g_t^i)$. Moreover the process $Y=(Y^1, Y^2)$ defined in \eqref{BivPr} is a centered Gaussian process.
	
\end{example}  

Now we introduce the multiple Wiener-It\^o integrals. We refer to \cite[\S 2.7]{NP12} for the formal definition. Here we use the following approach. 

\begin{definition}\label{def:mul.int}
	Let us consider an isonormal Gaussian process $X$ on a separable Hilbert space $\H$ with inner product $\<\cdot, \cdot\>_{\H}$. For $p \in \N$ we denote by $H_p$ the Hermite polynomial of order $p$ (details can be found in \cite[\S 1.3]{NP12}). The Wiener-It\^o integral of order $p$ is defined as
	$$
	I_p(h^{\otimes p})=H_p(X(h))
	$$ 
	with $h \in\H$ such that $\|h\|_{\H}=1$ and $h^{\otimes p}$ is the $p$th tensor product of $h$, i.e. $\underbrace{h\otimes\cdots h}_{p \text{ times}}$. 
\end{definition}   
The multiple integrals have the following properties (see \cite{NP12}):
\begin{itemize}
	\item \textit{Isometry property}: Fix integers $p,q\geq 1$ and $f\in \H^{\odot p}$ and $g\in\H^{\odot q}$, then
	$$
	\E[I_p(f)I_q(g)]=p!\<f,g\>_{\H^{\odot p}}
	$$
	when $p=q$, $0$ otherwise;
	\item \textit{Product formula}: let $p,q\geq 1$ and  $f\in \H^{\odot p}$ and $g\in\H^{\odot q}$, then
	\begin{equation}\label{product_formula}
		I_p(f)I_q(g)=\sum_{r=0}^{p\wedge q} r!\binom{p}{r}\binom{q}{r} I_{p+q-2r}(f \tilde \otimes_{r} g).
	\end{equation}
	
\end{itemize}

Let us recall the definitions of three probability metrics over the space of probability measures: the Kolmogorov distance ($d_K$), the Total Variation distance ($d_{TV}$) and the Wasserstein distance ($d_W$). For further details, we refer to \cite{NP12}. Let $X,Y$ be two real random variables. The Kolmogorov distance between $X$ and $Y$ is defined as 
\begin{equation*}
	d_{K}(X,Y):=\sup_{z\in\R} |\P(X\leq z)-\P(Y\leq z)|.
\end{equation*}
The Total Variation distance between $X$ and $Y$ is defined as
\begin{equation*}
	\dTV(X,Y) := \sup_{A\in \mathcal B(\mathbb R)}\left | \mathbb P(X\in A) - \mathbb P(Y\in A)\right |.
\end{equation*}
When $X,Y$ are integrable, the Wasserstein distance between $X$ and $Y$ is defined as 
\begin{equation*}
	\dW(X,Y) := \sup_{h\in \text{Lip}(1)} \left |\mathbb E[h(X)] - \mathbb E[h(Y)] \right |,
\end{equation*}
where $\text{Lip}(1)$ denotes the space of functions $h:\mathbb R\to \mathbb R$ which are Lipschitz continuous with Lipschitz constant $\le 1$.

\begin{proposition}\label{dWTVKgaussian}
	Let $N_1\sim\mathcal N(0,\sigma_1^2)$ and $N_2\sim \mathcal N(0,\sigma_2^2)$. Then
	\begin{align*}
		&d_K(N_1,N_2)\leq \frac{1}{\sigma_1^2\vee\sigma_2^2}|\sigma_1^2-\sigma_2^2|\\
		&\dW(N_1,N_2)\leq \frac{\sqrt{\frac{2}{\pi}}}{\sigma_1\vee\sigma_2}|\sigma_1^2-\sigma_2^2|\\
		&\dTV(N_1,N_2)\leq \frac{2}{\sigma_1^2\vee\sigma_2^2}|\sigma_1^2-\sigma_2^2|.
	\end{align*}
\end{proposition}

Theorem 5.13 in \cite{NP12} gives a direct connection between stochastic calculus and probability metrics.

\begin{theorem}\label{principal_0}
	Let $F\in\mathbb D^{1,2}$ such that $\E[F]=0$ and $\E[F^2]=\sigma^2<+\infty$. Then we have for $N\sim \mathcal N(0,1)$  
	$$
	\dW(F, N)\leq\sqrt{ \frac2{\pi \sigma^2}}\E[|\sigma^2-\<DF, -DL^{-1}F\>_\H|].
	$$
	Also, assuming that $F$ has a density, we have
	\begin{align*}
		&\dTV(F,  N)\leq \frac 2{\sigma^2}\E[|\sigma^2-\<DF, -DL^{-1}F\>_\H|] \\
		&d_{K}(F,  N)\leq \frac 1{\sigma^2}\E[|\sigma^2-\<DF, -DL^{-1}F\>_\H|] 
	\end{align*}
	Moreover, if $F\in\mathbb D^{1,4}$, we have 
	$$
	\E[|\sigma^2-\<DF, -DL^{-1}F\>_\H|] \leq \sqrt{\Var(\<DF, -DL^{-1}F\>)}.
	$$
\end{theorem}

When $F=I_q(f)$ for $f\in \H^{\odot q}$, $q\geq 2$, Theorem 5.2.6 in \cite{NP12} ensures that
$$
\E[|\sigma^2-\<DF, -DL^{-1}F\>_\H|]\leq \sqrt{\Var\Big(\frac 1q\|DF\|_\H^2\Big)}
$$
and from Lemma 5.2.4 
\begin{align}\label{v_4_est}
	\Var\Big(\frac 1q\|DF\|_\H^2\Big)&=\frac1{q^2}\sum_{r=1}^{q-1}r^2 r!^2\binom{q}{r}^4 (2q-2r)!\|f\tilde \otimes_r f\|^2_{\H^{\otimes 2q-2r}}\notag \\
	&\leq \frac1{q^2}\sum_{r=1}^{q-1}r^2 r!^2\binom{q}{r}^4 (2q-2r)!\|f \otimes_r f\|^2_{\H^{\otimes 2q-2r}}.
\end{align}
Finally, we conclude that
\begin{align}\label{c_4_est}
	\Var\Big(\frac 1q\|DF\|_\H^2\Big)\leq \frac{q-1}{3q} \kappa_4(F)\leq (q-1)\Var\Big(\frac 1q\|DF\|_\H^2\Big)
\end{align}
where 
$$
\kappa_4(F)=\E[F^4]-3\E[F^2]^2
$$
is the fourth cumulant of $F$. Then 
\begin{theorem}\label{Principal}
	Let $\{F_n\}_{n\in\N}$ a sequence of random variables belonging to a fixed $q$-th Wiener chaos, for fixed integer $q\geq 2$. Then 
	\begin{align*}
		&d_M\Big(\frac{F_n}{\sqrt{\Var(F_n)}},N\Big) \leq C_{M}(q)\sqrt{\frac{\kappa_4(F_n)}{\Var(F_n)^2}} \\
	\end{align*}
	where $N\sim \mathcal N(0,1)$ and $M$ stays for $K, TV, W$. In particular, when 
	$
	\frac{\kappa_4(F_n)}{\Var(F_n)^2}\to 0
	$ 
	then 
	$$
	\frac{F_n}{\sqrt{\Var(F_n)}}\overset{d}{\to} N .
	$$
\end{theorem}
A fundamental consequence of the above theorem is the Fourth-Moment Theorem (Theorem 5.2.7 in \cite{NP12}).
\begin{theorem}[Fourth-Moment Theorem]\label{FOURTH}
	Let $F_n=I_q(f_n)$, $n\geq 1$, be a sequence of random variables belonging to the $q$th chaos of $X$, for some fixed integer $q\geq 2$ (so that $f_n\in \H^{\odot q}$). Assume, moreover, that $\E[F_n^2]\to \sigma^2>0$ as $n\to+\infty$. Then, as $n\to +\infty$ the following assertions are equivalent:
	\begin{enumerate}
		\item $F_n$ converges in distribution to $N\sim \mathcal N(0,\sigma^2)$;
		\item $\E[F_n^4]	\to 3\sigma^4$ or equivalently $\kappa_4(F_n)\to 0$;
		\item $\Var(\|DF_n\|^2_\H)\to 0$;
		\item $\|f_n\tilde \otimes_r f_n\|_{\H^{\otimes (2q-2r)}}\to 0$, for all $r=1,\ldots, q-1$;
		\item $\|f_n \otimes_r f_n\|_{\H^{\otimes (2q-2r)}}\to 0$, for all $r=1,\ldots, q-1$.
	\end{enumerate}
\end{theorem}

The last result that we recall shows that, if we consider a random vector sequence, whose components are Wiener-It\^o integrals, then the componentwise convergence to Gaussian variables always implies the joint convergence.

\begin{theorem}[Theorem 6.2.3 in \cite{NP12}]\label{jointconvergence}
	Let $d\geq 2$ and $q_1, \ldots , q_d \geq 1$ be some fixed integers. Consider vectors $F_n= (F_{1,n}, \ldots, F_{d,n}) = (I_{q_1}( f_{1,n}),\ldots , I_{q_d} ( f_{d,n}))$, $n \geq 1$, with $f_{i,n}\in \H^{\odot q_i}$. Let $C \in \mathcal M_d(\R)$ be a symmetric non-negative definite matrix, and let $N \sim N_d (0,C)$. Assume that
	$$
	\lim_{n\to \infty} \E[F_{i,n}F_{j,n}]=C(i,j), \quad{1\leq i,j\leq d}.
	$$
	Then, as $n\to \infty$, the following two conditions are equivalent:
	\begin{itemize}
		\item[a)] $F_n$ converges in law to $N$.
		\item[b)]  For every $i=0,\ldots, d$, $F_{i,n}$ converges in law to $\mathcal N(0, C(i,i))$.
	\end{itemize}
\end{theorem}

\section{The univariate fractional Ornstein-Uhlenbeck process}\label{sec:1dimfou}

Here we recall the definition of the univariate fOU process and its main properties. In this discussion, we primarily follow the work by Cheridito et al. \cite{cheridito2003}. 
Let $(\Omega, \mathcal F, \P)$ be a probability space. Let us fix $\a\in \R_+$. Then for all $t>b$, $t,b\in\R$, the random variable
$$
X^b_t=\int_{b}^t e^{\a u} dB_u^H(\omega).
$$ 
exists as a pathwise Riemann-Stieltjes integral, as per Proposition A.1 in \cite{cheridito2003}. Moreover, it can be expressed as:
\begin{equation}\label{FormX}
	X_t = e^{\alpha t} B_t^H(\omega) - e^{\alpha b} B_b^H(\omega) - \alpha \int_{b}^t B_u^H(\omega) e^{\alpha u} du.
\end{equation}
Now, consider $\alpha, \nu > 0$, and $\psi \in L^0(\Omega)$. The solution to the Langevin equation
\begin{equation}\label{Langeq}
	Y_t^H = \psi - \alpha\int_{0}^t Y_s ds + \nu B_t^H, \quad t\geq 0
\end{equation}
exists as a path-wise Riemann-Stieltjes integral, and it is given by
$$
Y_t^{H, \psi}= e^{-\a t}\Big(\psi+\nu \int_{0}^t e^{\a u}dB_u^H \Big),\,\,\, t\geq 0.
$$
It is the unique almost surely continuous process that solves \eqref{Langeq}. In particular, the process
\begin{equation}\label{fOU1}
	Y_t^H = \nu\int_{-\infty}^t e^{-\alpha(t-u)} dB_u^H, \quad t\in \R,
\end{equation}
solves \eqref{Langeq} with the initial condition $\psi = Y_0^H$. From \eqref{FormX}, it follows that $Y_t^H$ has the following almost surely representation:
\begin{equation}\label{repY}
	Y_t^H = \nu \left( B_t^H - \alpha\int_{-\infty}^t B_u^H e^{-\alpha(t- u)} du \right).
\end{equation}
The process $Y^H$ is the stationary fOU process.  Let us recall Lemma 2.1 in \cite{cheridito2003}.

\begin{lemma}\label{repLemma21}
	Let $H\in (0,\frac 12)\cup (\frac 12, 1]$, $\a>0$ and $-\infty\leq a<b\leq c<d<+\infty$. Then
	$$
	\E\Big[\int_a^b e^{\a u} dB_u^H \int_c^de^{\a v}dB_v^H\Big]=H(2H-1)\int_a^b e^{\a u}\Big(\int_c^d e^{\a v}(v-u)^{2H-2}dv\Big) du.
	$$
\end{lemma}
Now, let us provide an explicit expression for the autocovariance function and the variance of $Y_t^H$, where $t\in \R$. Due to the stationarity of $Y^H$, we recall that $\Var(Y_t^H) = \Var(Y_0^H)$ for all $t\in \R$. From \cite{PipirasTaqqu}, we have
\begin{equation}\label{PTformula}
	\Cov(Y^H_t, Y^H_{t+s}) = \nu^2\frac{\Gamma(2H+1)\sin\pi H}{2\pi}\int_{-\infty}^{+\infty} e^{\mathbf{i}sx} \frac{|x|^{1-2H}}{\alpha^2+x^2} dx.
\end{equation}
and it follows that the variance of the process is given by
\begin{equation}\label{varfOU1}
	\Var(Y_t^H)=\Var(Y^H_0)=\nu^2\frac{\Gamma(2H+1)}{2\a^{2H}}.
\end{equation}

Next, we recall Theorem 2.3 in \cite{cheridito2003}, which provides the asymptotic behavior of the autocovariance function of $Y^H$. Additionally, we present a result regarding the regularity of the covariance function.

\begin{theorem}\label{Decay1}[Theorem 2.3 in \cite{cheridito2003}]
	Let $H\in (0,\frac 12)\cup (\frac 12,1]$ and $N\in\N$. Let $Y^H$ the fOU in \eqref{fOU1}. Then for $t\in \R$ and $s\to \infty$, 
	\begin{equation}\label{cov1dim}
		\Cov(Y_t^H,Y_{t+s}^H)=\frac{1}{2}\nu^2 \sum_{n=1}^N \a^{-2n}\Big(\prod_{k=0}^{2n-1}(2H-k)   \Big)s^{2H-2n} +O(s^{2H-2N-2}).
	\end{equation}
	
\end{theorem}
The autocovariance decays as a power-law, particularly illustrating long-range dependence for $H\in(1/2, 1]$. Theorem \ref{Decay1} implies the ergodicity of the fOU process. We can also deduce that, when $|t-s|\to 0$, $\Cov(Y_t^H, Y_s^H)\to \Var(Y_0^H)$, because of the stationarity of $Y^H$. The following result holds.
\begin{lemma}\label{shorttime1dim}
	Let $H \neq 1/2$. Then, as $t\to s$
	\begin{align*}
		&	\Cov(Y^H_{t}, Y^H_s)=\\
		&=\Var(Y^H_0)-\frac{\nu^2}{2}|t-s|^{2H}+\frac{\a^2}{2}\Var(Y^H_0)|t-s|^2 - \frac{\a^2\nu^2|t-s|^{2H+2}}{4(H+1)(1+2H)} +o(|t-s|^{4}).
	\end{align*}
\end{lemma}

\section{Some properties of the multivariate fractional Brownian Motion}\label{sec:app_mfbm}

\begin{remark}\label{cov_existence_mfbm}
The functions in \eqref{cH} and \eqref{cH1} define a covariance matrix if and only if the parameters $H_i$, $\rho_{ij}$, and $\eta_{ij}$, $i,j=1,\dots,d$, deliver a Hermitian matrix $Q$ with entries given by $[Q]_{ij}=\left(\Gamma(H_i+H_j+1)\tau_{ij}(1)\right)$, where
$$
\tau_{ij}(1)= 
\begin{cases}
\rho_{ij}\sin\left(\frac{\pi}{2}(H_i+H_j)\right)-\text{i}\eta_{ij}\cos\left(\frac{\pi}{2}(H_i+H_j)\right)		&\text{if } H_i+H_j\ne 1 \\ 
\rho_{ij}-\text{i}\frac{\pi}{2}\eta_{ij}												&\text{if } H_i+H_j= 1 \\ 
\end{cases},
$$
and $\text{i}$ is the imaginary unit, which is positive semi-definite. See Proposition 9 in \cite{multivar} for more.
\end{remark}

Let us focus here on a bivariate fractional Brownian motion (2fBm) $(B^{H_1}, B^{H_2})$, where we denote $\rho=\rho_{12}=\rho_{21}$ and $H=H_1+H_2$. Here, $H_1$ and $H_2$ denote the Hurst indexes of $B^{H_1}$ and $B^{H_2}$ respectively. Additionally, without loss of generality, we set $\sigma_1=\sigma_2=1$. 

\begin{remark}
	
	In the bivariate case, the condition given in Remark \ref{cov_existence_mfbm} reduces to $C_{12}=C_{21}<1$, where $C_{12}$ is called coherency function and is defined as follows: 
	when $H\neq 1$ 
	\begin{equation}\label{domain}
		C_{12}=\frac{\Gamma(H+1)^2}{\Gamma(2H_1+1)\Gamma(2H_2+1)} \frac{\rho^2 \sin^2 \big(\frac{\pi}{2}H\big)+\eta_{12}^2 \cos^2(\frac{\pi}{2}H)}{\sin \pi H_1 \sin \pi H_2}\leq 1,
	\end{equation}  
	and when $H=1$
	\begin{equation}\label{domainH1}
		C_{12}=\frac{1}{\Gamma(2H_1+1)\Gamma(2H_2+1)} \frac{\rho^2 +\frac{\pi^2}{4}\eta_{12}^2 }{\sin \pi H_1 \sin \pi H_2}\leq 1.
	\end{equation}
	Proposition 9 in \cite{multivar} establishes that \eqref{cH} and \eqref{cH1} indeed function as covariance functions when $C_{12}\leq 1$. Therefore, for given $H_1, H_2$, the parameter space of $\rho$ and $\eta_{12}$  is constrained by \eqref{domain} and \eqref{domainH1}. This parameter space forms the interior of the ellipse $\frac{\rho^2}{a^2}+\frac{\eta_{12}^2}{b^2}=1$  centered at the origin, with semi-axes length given by 
	$$
	a=\sqrt{\frac{\Gamma(2H_1+1)\Gamma(2H_2+1)\sin\pi H_1 \sin \pi H_2}{\Gamma(H+1)^2 \sin^2(\frac{\pi}2 H)}},
	$$ 
	and 
	$$
	b=\sqrt{\frac{\Gamma(2H_1+1)\Gamma(2H_2+1)\sin\pi H_1 \sin \pi H_2}{\Gamma(H+1)^2 \cos^2(\frac{\pi}2 H)}}.
	$$
\end{remark}

Let us recall that for a fixed $h\in\R$, the stationary property of the increments of the fBm ensures that the covariance of the increment process $(\overline B^{h, H_i}_t)_{t\in\R}$ with $\overline{B}^{h , H_i}_t=B^{H_i}_{t+h}-B^{H_i}_h$ is  
\begin{align*}
	\Cov(\overline{B}^{h , H_i}_t, \overline{B}^{h , H_i}_s)=\text{Cov}(B^{H_i}_{t+h}-B^{H_i}_h, B^{H_i}_{s+h}-B^{H_i}_h)=\text{Cov}(B^{H_i}_t, B^{H_i}_s).
\end{align*}

Therefore $\overline{B}^{h, H_i}$ is a fBm with Hurst index $H_i$. This property can be extended to the mfBm. 
\begin{lemma}\label{increments}
	Let $(B^{H_1}, \dots, B^{H_d})$ be the mfBm defined in \ref{fBm2}. For fixed $h\in \R$, let $(\overline{B}^{h, H_1},\dots,\overline{B}^{h, H_d})$ be the process defined as
	$$
	(\overline{B}_t^{h, H_1},\dots, \overline{B}_t^{h, H_d})=(B_{t+h}^{H_1}-B_h^{H_1},\dots,B_{t+h}^{H_d}-B_h^{H_d}), \,\,\, t\geq 0.
	$$
	Then $(\overline{B}^{h, H_1},\dots,\overline{B}^{h, H_d})$ is a mfBm as in Definition \ref{fBm2} with Hurst indices $(H_1, \dots, H_d)$. 
\end{lemma} 
\begin{proof}
	From the stationarity of the increments, we have that, $\forall i,j=1,\dots,d$,
	\begin{align*}
		&\E[\overline{B}^{h , H_i}_t \overline{B}^{h, H_j}_s]=\E[  (B_{t+h}^{H_i}-B_h^{H_i})(B_{s+h}^{H_j}-B^{H_j}_h)  ]=\E[B_t^{H_i} B_s^{H_j}].
	\end{align*}
	It follows that, for every $h\in \R$, the process $(\overline{B}^{h,H_1}, \dots, \overline{B}^{h,H_d})$ is a mfBm as in Definition \ref{fBm2}.
\end{proof}

Let us finally recall the following theorem, which provides a moving average representation of $(B^{H_1},\dots, B^{H_d})$.
\begin{theorem}\label{MovAverageFBM}[Theorem 8 in \cite{multivar}]
	Let $(B^{H_1}, \dots, B^{H_d})$ be the mfBm in Definition \ref{fBm2}. For $(H_1, \dots, H_d) \in (0,1)^d$ and $H_i\neq \frac 12,\i=1,\dots,d$ there exists $M^+$, $M^{-}$ two $d\times d$ real matrices such that, for $i=1,\dots,d$, 
	\begin{align}\label{movinRep}
		&B^{H_i}_t=\sum_{j=1}^d\int_{\R} M_{i,j}^+((t-x)_+^{H_i-\frac 12}-(-x)_{+}^{H_i-\frac 12})+M_{i,j}^-((t-x)_-^{H_i-\frac 12}-(-x)_{-}^{H_i-\frac 12}) W_j(dx)
	\end{align}
	where $W=(W_1,\dots,W_d)$ is a Gaussian white noise with zero mean, independent components and covariance $\E[W_i(dx)W_j(dx)]=\delta_{ij} dx$. 
\end{theorem}

\bibliographystyle{abbrv}
\bibliography{roughvol}
\end{document}